\documentclass[11pt]{article}

\usepackage[utf8]{inputenc}
\usepackage{amsfonts,amsthm}
\usepackage{amsmath}   
\usepackage{latexsym}
\usepackage{amssymb}
\usepackage{mathrsfs}
\usepackage{pifont}
\usepackage{tikz}
\usepackage[all]{xy}
\usepackage[hidelinks]{hyperref}

\newtheoremstyle{rem}{3pt}{3pt}{}{}
{\bfseries}{.}{.5em}{}

\newtheorem{theo}{Theorem}[section]
\newtheorem*{theo*}{Theorem}
\newtheorem{defi}[theo]{Definition}

\newtheorem{prop}[theo]{Proposition}

\newtheorem{coro}[theo]{Corollary}
\newtheorem{conj}[theo]{Conjecture}

\newtheorem{rema}[theo]{Remark}
\newtheorem{exam}[theo]{Example}
\theoremstyle{rem}

\newcommand{\bT}{\boldsymbol{t}}

\newcommand{\STab}{\mathrm{STab}}
\newcommand{\SSTab}{\mathrm{SSTab}}

\newcommand{\da}{\dagger}
\newcommand{\HJ}{H^J}
\newcommand{\Hmu}{H^{\mu}(S_n)}
\newcommand{\cD}{\mathcal{D}}
\newcommand{\dWJ}{W_J\backslash W/ W_J}
\newcommand{\dSmu}{S_{\mu}\backslash S_n/ S_{\mu}}
\newcommand{\Hk}{\mathrm{Hook}}
\newcommand{\tP}{\widetilde{P}}
\newcommand{\tQ}{\widetilde{Q}}
\newcommand{\tsh}{\widetilde{sh}}

\newcommand{\rmd}{r^-(\cD)}
\newcommand{\rmdp}{r^-(\cD')}
\newcommand{\rpd}{r^+(\cD)}
\newcommand{\rpdp}{r^+(\cD')}

\newcommand{\diag}[3]{ \foreach \t in {1,...,#3} {\draw[thick] (#1+\t,#2-1) rectangle (#1+\t-1,#2);} }
\newcommand{\diagg}[4]{ \foreach \t in {1,...,#3} {\draw[thick] (#1+\t,#2-1) rectangle (#1+\t-1,#2);} \foreach \t in {1,...,#4} {\draw[thick] (#1+\t,#2-1) rectangle (#1+\t-1,#2-2);} }
\newcommand{\diaggg}[5]{ \foreach \t in {1,...,#3} {\draw[thick] (#1+\t,#2-1) rectangle (#1+\t-1,#2);} \foreach \t in {1,...,#4} {\draw[thick] (#1+\t,#2-1) rectangle (#1+\t-1,#2-2);}
                         \foreach \t in {1,...,#5} {\draw[thick] (#1+\t,#2-2) rectangle (#1+\t-1,#2-3);} }

\usepackage{geometry}
\title{Kazhdan--Lusztig bases of parabolic Hecke algebras and applications to Schur--Weyl duality}

\author{J. Guilhot\footnote{J\'er\'emie Guilhot passed away on 27th July 2025 when most of the paper was already written. The second author dedicates this paper to his memory. Tribute to J\'er\'emie can be found here: https://www.idpoisson.fr/hommage-guilhot/}\ \footnote{Institut Denis Poisson, UMR CNRS 7013, Universit\'e de Tours, 37200 Tours, France}, L. Poulain d'Andecy\footnote{Laboratoire de math\'ematiques de Reims, UMR CNRS 9008, Universit\'e de Reims Champagne-Ardenne, 51100 Reims, France. \emph{email adress}: loic.poulain-dandecy@univ-reims.fr}}

\date{}

\begin{document}

\maketitle

\begin{abstract}
With an eye to applications to type A and Schur--Weyl duality, we study Kazhdan--Lusztig bases for a general parabolic Hecke algebra. Parabolic Hecke algebras are idempotent subalgebras of Hecke algebras corresponding to parabolic subgroups, and for type A they coincide with the fused Hecke algebras appearing in a generalisation of the Schur--Weyl duality with the quantum group of $GL(N)$. In this paper we investigate two different Kazhdan--Lusztig bases for parabolic Hecke algebras, together with the associated cells and the corresponding representations. We quickly specialise to type A, for which we describe the cells in terms of the RSK correspondence generalising thus the well-known description for the symmetric group. As a first application we recover the classification of irreducible representations of parabolic Hecke algebras of type A and provide a new construction of these representations. Next we turn to the Schur--Weyl duality and describe the kernel in terms of one the basis studied precedently. Moreover, we formulate some conjectures about a generator of these kernels in terms of Kazhdan--Lusztig basis elements, give some evidence and prove these conjectures in some special cases.
\end{abstract}

\section{Introduction}

Parabolic Hecke algebras were introduced in \cite{CIK71} as double-cosets algebras relative to parabolic subgroups of reductive groups over finite fields. In general, they received far less attention than their well-known particular case: the usual Iwahori--Hecke algebra (corresponding to a Borel subgroup). Parabolic Hecke algebras can also be defined generically (with a generic parameter $q$) as idempotent subalgebras of usual Hecke algebras associated to Coxeter groups. In formulas, the definition is
\[\HJ(W)=e_jH(W)e_j\,,\]
where $H(W)$ is the usual Hecke algebra associated to a Coxeter group $W$ and $e_J$ is the $q$-symmetriser corresponding to a parabolic subgroup $W_J$ of $W$. The generic parabolic Hecke algebras $\HJ(W)$ are the main objects of study in this paper.

In \cite{Cur85}, Curtis extended to the parabolic case the Lusztig's isomorphism between a group algebra and the Hecke algebra. To do so, he considered a certain basis of the parabolic Hecke algebra, made of a certain subset of Kazhdan--Lusztig elements of the usual Hecke algebra. This can be seen as one Kazhdan--Lusztig basis for the parabolic Hecke algebra. Parabolic Hecke algebras have also been studied more recently \cite{APV13}, but not from the point of view of Kazhdan--Lusztig theory (see also \cite{Gom98}). 

Our interest for parabolic Hecke algebra stems from the Schur--Weyl duality studied in \cite{CP23}. It was shown there that the parabolic Hecke algebra of type A (that was called \emph{fused Hecke algebras} in \cite{CP21,CP23}) allows to obtain the centralisers of tensor products of some representations of the quantum group $U_q(gl_N)$, namely we have a surjective morphism
\begin{equation}\label{intro_SWpar}
\pi^{J}\ :\ \HJ(S_n)\to \text{End}_{U_q(gl_N)}\bigl(S^{\mu_1}_qV\otimes\dots\otimes S^{\mu_d}_qV\bigr)\,,
\end{equation}
where the representations appearing are $q$-symmetrised powers of the vector representation $V$ of $U_q(gl_N)$. The relevant parabolic subgroup of $S_n$ corresponds to the choice of $\mu_1,\dots,\mu_d$. This generalises the usual quantum Schur--Weyl duality involving the usual Hecke algebra
\begin{equation}\label{intro_SW}
\pi\ :\ H(S_n)\to \text{End}_{U_q(gl_N)}\bigl(V^{\otimes n})\,.
\end{equation}
The surjectivity of the morphisms $\pi$ and $\pi^J$ is seen as the first fundamental theorem of invariant theory, while the second fundamental theorem would be the description of the kernel of these morphisms. It is well-known that the map $\pi$ is not injective as soon as $n>N$. Similarly, the map $\pi^J$ is not injective as soon as $d>N$ \cite{CP23} and it remains to understand its kernel.

For the usual quantum Schur--Weyl duality in (\ref{intro_SW}), the quotient of the Hecke algebra appearing is well-understood, see for example \cite{BEG20,EMTW20,GW93,Har99,Jim86, Mat99,Mur95,RSS12,Res87}. For example for $N=2$, it is the Temperley--Lieb algebra. An explicit generator of the corresponding ideal of the Hecke algebra is known, as well as a linear basis. It turns out that all this can be described entirely and quite naturally in terms of Kazhdan--Lusztig basis elements. Recall that we have two Kazhdan--Lusztig bases for a Hecke algebra
\[\{C_w\}_{w\in W}\ \ \ \ \ \ \ \text{and}\ \ \ \ \ \ \ \{C^{\da}_w\}_{w\in W}\ .\]
With our conventions, the kernel of the map $\pi$ is generated by the $q$-antisymmetriser on $N+1$ letters, which turns out to be the element $C^{\da}_{w_{N+1}}$ corresponding to the longest element of the symmetric group $S_{N+1}$ (properly embedded into $S_n$).

In contrast, outside of some particular cases \cite{CP23,Dem25,LZ10,PZ24}, the kernel of $\pi^J$ in the parabolic Schur--Weyl duality (\ref{intro_SWpar}) is not well-understood. Building on the example of the usual Hecke algebra, one could expect again the Kazhdan--Lusztig theory to be useful here, and this is one of the motivations for this work. Note however that this is not going to be as simple as in the usual case. One first reason is that the $q$-antisymmetriser $C^{\da}_{w_{N+1}}$ becomes trivial in the parabolic Hecke algebra: we have $e_JC^{\da}_{w_{N+1}}e_J=0$. So obviously the kernel has to be described differently.

A second more serious reason for the increase of the difficulty is the following. In the usual Schur--Weyl duality, roughly speaking, we basically have to forget a single irreducible representation of the Hecke algebra, the one corresponding to the one-column partition of $N+1$ boxes. This becomes different fo the parabolic Schur--Weyl duality where the kernel of $\pi^J$ in (\ref{intro_SWpar}) contains in general more than one irreducible representation, even at the first level where it is non-trivial. It is best illustrated with a simple example. If we take $d=3$ and $\mu_1=\mu_2=\mu_3=2$ in (\ref{intro_SWpar}) then here are the irreducible representations (with their dimensions) of the parabolic Hecke algebras:
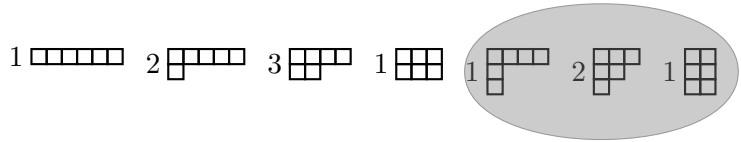
\begin{center}
 \begin{tikzpicture}[scale=0.2]
\diag{-25}{-9}{6};\node at (-26,-9.5) {$1$};\diagg{-16}{-9}{5}{1};\node at (-17,-10) {$2$};\diagg{-8}{-9}{4}{2};\node at (-9,-10) {$3$};
\diagg{-1}{-9}{3}{3};\node at (-2,-10) {$1$};\diaggg{5}{-9}{4}{1}{1};\node at (4,-10.5) {$1$}; \diaggg{12}{-9}{3}{2}{1};\node at (11,-10.5) {$2$};\diaggg{18}{-9}{2}{2}{2};\node at (17,-10.5) {$1$};

\draw[thin, fill=gray,opacity=0.4] (3.5,-10.5)..controls +(0,6) and +(0,6) .. (21.5,-10.5) .. controls +(0,-6) and +(0,-6) .. (3.5,-10.5);

\end{tikzpicture}
\end{center}
For $N=2$, the kernel contains the three irreducible representations in the shaded area. Therefore, when we had a single canonical choice to obtain the ideal in the usual Schur--Weyl duality, now we are left with more freedom and no clear indication on what will be a (natural) generator of the ideal. In \cite{CP23} a conjectural generator of the ideal was given relying on diagrammatical considerations. In the present paper we will build on Kazhdan--Lusztig theory (that we need to develop a little for parabolic Hecke algebras) to try to obtain the ideal in a different way. 

To describe our strategy in a few words, we notice that the sought-for ideal, even if consisting of several representations, is made up of all those representations which are smaller in the dominance order than a certain hook shape (see the example above), and moreover this hook shape is of dimension 1. Therefore, building on the ideas of cellular algebras (or Kazhdan--Lusztig cells in our case), we look for the unique element in the cell corresponding to this hook shape and promote it as our best candidate for generating the ideal. Remarkably, we conjecture and prove in some cases that this was exactly the element found diagrammatically in \cite{CP23}.

\paragraph{Content of the paper.} We describe in more details the content of the paper. To follow the program sketched above, we need to develop a little bit a theory of Kazhdan--Lusztig cells for parabolic Hecke algebras in order to subsequently apply it to the Schur--Weyl duality.

So we start with Kazhdan--Lusztig bases, cells and representations for parabolic Hecke algebras in general, keeping in mind our goal to deal specifically with the type A, where the Kazhdan--Lusztig theory works especially well. We consider in this paper two different Kazhdan--Lusztig bases for a general parabolic Hecke algebra $\HJ(W)$:
\begin{equation}\label{intro_bases}
\{C_{\rpd}\}\ \ \ \ \ \text{and}\ \ \ \ \ \{e_JC^{\da}_{\rmd}e_J\}\,,\ \ \ \ \ \text{indexed by $\cD\in\dWJ$.}
\end{equation}
They are indexed by the double cosets of the parabolic subgroup $W_J$ in $W$, and the first one involves the maximal-length representatives $\rpd$ of such cosets, while the second one involves the minimal-length representatives $\rmd$. The first basis will certainly be considered as the natural Kazhdan--Lusztig basis for parabolic Hecke algebra and indeed it is the one appearing in \cite{Cur85}. The second one seems to be new and may appear at first to be both less practical due to the presence of the idempotent $e_J$. However, note first that the elements $\rmd$ of the Coxeter group $W$ involved are smaller than their counterparts $\rpd$. More importantly, we stress that the second basis is the only one that will be relevant in the Schur--Weyl duality context. Maybe it is a good place to emphasise that the symmetry between the two bases $\{C_w\}$ and $\{C^{\da}_w\}$ which holds in the usual Hecke algebra is broken in the parabolic setting, due to the presence of the idempotent $e_J$. In particular the two bases in (\ref{intro_bases}) behave quite differently.

As far as general theory is concerned, we show that the two bases (\ref{intro_bases}) have indeed the expected property, namely, they are uniquely characterised by the bar-invariance and a unitriangular decomposition with respect to a standard basis with coefficients having the required polynomial property.

Having those two bases, we can proceed with the usual notions of (left, right, two-sided) cells and associated cell representations for a general parabolic Hecke algebra. We prove a general result, namely that with this theory the cell modules for the parabolic Hecke algebras are the projections (with $e_J$) of the cell modules of the usual Hecke algebras. Here appears the fact that the second basis is more delicate to handle, and we need an assumption (irreducibility of the cell modules) which is going to be satisfied in type A.

At this point, we specialise in the rest of the paper to the parabolic Hecke algebra of type A. First we build on the general theory to study the cell structure. Our main results are the following: 
\begin{itemize}
\item We completely describe the cells corresponding to our two different bases. For both bases, there is a nice and clean description using two different Robinson--Schensted--Knuth (RSK) correspondences involving semistandard Young tableaux. This we see as the generalisation of the well-known description of cells for the usual symmetric group, and we see the parabolic Hecke algebra (and its Kazhdan--Lusztig theory) as the algebraic incarnation of the RSK correspondence between pairs of semistandard Young tableaux and double cosets in the symmetric group.

\item We recover the classification of \cite{CP23} of the irreducible representations in the semisimple setting. This is done from the point of view of the cell representations associated to the Kazhdan--Lusztig bases and in particular provides an alternative construction, compared to \cite{CP23}, of the representations. 

\item Using the two bases in (\ref{intro_bases}) and the RSK correspondence, we make explicit two cellular bases in the sense of Graham--Lehrer \cite{GL96}.
\end{itemize}

Finally, we turn to our initial goal, the study of the kernel of the Schur--Weyl duality in (\ref{intro_SWpar}). We work in the semisimple situation in this part. 

Our first main result is that the detailed study of cells and associated representations leads very quickly to a natural description of a linear basis of the ideal in the Schur--Weyl duality. We stress again that this is all in terms of the second basis in (\ref{intro_bases}). This could be the end of the story but we would like also to have an algebraic generator of this ideal.

The second basis provides a natural candidate for such a generator. We introduce explicitly this natural candidate and formulate two conjectures: the first one is that it does provide a generator of the ideal; the second one is that (somewhat miracuously) this generator coincides with the element introduced diagrammatically in \cite{CP23}. We provide some evidence in general and we fully prove these two conjectures in the following cases:

\begin{itemize}
\item in general for $N=2$; therefore the centraliser of any tensor product of $U_q(gl_2)$-representations is described in this way.
\item for any $N\geq 1$ when  $\mu$ is of the form $(\mu_1,1,1,\dots,1)$ (the one-boundary case).
\end{itemize}

\paragraph{Organisation.} Necessary notations and known results on the Kazhdan--Lusztig cells for the symmetric group are collected in Section \ref{sec_prel}. The general theory of parabolic Hecke algebras and its Kazhdan-Lusztig bases is developed in Sections \ref{sec_para} and \ref{sec_KLpar}. This general theory is applied to type A in Section \ref{sec-typeA}, while the applications to Schur--Weyl duality are developed in Section \ref{sec-SW}.

\paragraph{Acknowledgements.} Both authors were supported by Agence National de la Recherche Projet AHA ANR-18-CE40-0001 in the course of this investigation.

\setcounter{tocdepth}{2}
\tableofcontents

\section{Preliminaries and notations on Hecke algebras}\label{sec_prel}

Let $(W,S)$ be a Coxeter system, for which we denote by $\ell$ the length function. The left descent $\mathcal{L}(x)$ of an element $x\in W$ consists of the simple transpositions $s\in S$ such that $\ell(sx)<\ell(x)$, and similarly for the right descent $\mathcal{R}(x)$. 

We denote by $\leq$ the strong Bruhat order between elements in $W$, which means that $x\leq y$ if a reduced expression for $y$ contains a reduced expression for $x$ (see for example \cite{BB05}).

\subsection{Definition and standard basis}

We work with an indeterminate $q$. The Hecke algebra $H(W)$ is the $\mathbb{Z}[q,q^{-1}]$-algebra with basis $\{T_w\}_{w\in W}$, and with multiplication given by:
\begin{equation}\label{multH}
T_sT_w=\left\{\begin{array}{ll} T_{sw} & \text{if $\ell(sw)>\ell(w)$,}\\
(q-q^{-1})T_w+T_{sw} & \text{if $\ell(sw)<\ell(w)$.}\end{array}\right.
\end{equation} 
It is well-known that such an algebra exists, see for example \cite{GP00}. We denote $1=T_{e_W}$. As consequences of the definition, we have: 
\begin{equation}\label{presH}\begin{array}{l}T_wT_{w'}=T_{ww'}\ \ \ \ \text{if $\ell(ww')=\ell(w)+\ell(w')$,}\\[0.5em]
T_s^2=(q-q^{-1})T_s+1\ \ \text{or equivalently, }\ (T_s-q)(T_s+q^{-1})=0.\end{array}
\end{equation}
For $w=s_{i_1}\dots s_{i_r}\in W$ written as a reduced expression, we have $T_w=T_{s_{i_1}}\dots T_{s_{i_r}}$. The basis $\{T_w\}_{w\in W}$ is the \emph{standard basis} of $H(W)$.

\subsection{Kazhdan--Lusztig bases}\label{subsec_KLusual}

We recall the construction of Kazhdan--Lusztig bases from \cite{KL79}. We have the following two involutive ring automorphisms of the Hecke algebra $H(W)$ given by their action on the generators and the indeterminate $q$:
\begin{equation}\label{involutions}\overline{\phantom{\,}\cdot\phantom{\,}}\ :\ T_i\mapsto T_i^{-1}\,,\ q\mapsto q^{-1}\ \ \ \ \ \text{and}\ \ \ \ \ \cdot^{\dagger}\ :\ T_i\mapsto -T_i\,,\ q\mapsto q^{-1}\ .
\end{equation}
The first Kazhdan--Lusztig basis $\{C_w\}_{w\in W}$ is the unique basis satisfying
\begin{equation}\label{firstKL}\overline{C_w}=C_w\ \ \ \text{and}\ \ \ C_w=T_w+\sum_{x<w}p_{x,w}T_x\ \ \ \text{with $p_{x,w}\in q^{-1}\mathbb{Z}[q^{-1}]$.}
\end{equation}
The second Kazhdan--Lusztig basis $\{C^\da_w\}_{w\in W}$ is obtained by applying the involution $\cdot^{\dagger}$ and is the unique basis satisfying
\begin{equation}\label{secondKL}\overline{C^\da_w}=C^\da_w\ \ \ \text{and}\ \ \ C^\da_w=(-1)^{\ell(w)}T_w+\sum_{x<w}(-1)^{\ell(x)}\overline{p}_{x,w}T_x\ \ \ \text{with $\overline{p}_{x,w}\in q\mathbb{Z}[q]$.}
\end{equation}
It is well-known \cite{KL79,Lus03} that:
\begin{equation}\label{TC=qC}
\begin{array}{c} T_sC_w=q C_w\ \ \ \text{and}\ \ \ T_sC^\da_w=-q^{-1}C^\da_w\ \ \ \ \ \ \ \ \text{if $s\in\mathcal{L}(w)$,}\\[0.5em]
C_wT_s=q C_w\ \ \ \text{and}\ \ \ C^\da_wT_s=-q^{-1}C^\da_w\ \ \ \ \ \ \ \ \text{if $s\in\mathcal{R}(w)$.}
\end{array}\end{equation}
It follows easily from (\ref{multH}) and (\ref{TC=qC}) that:
\begin{equation}\label{psyx=qpyx}
\begin{array}{c}p_{sy,x}=qp_{y,x}\ \ \ \ \ \ \ \ \ \text{if $s\in\mathcal{L}(x)$ and $s\notin\mathcal{L}(y)$,}\\[0.5em]
p_{ys,x}=qp_{y,x}\ \ \ \ \ \ \ \ \ \text{if $s\in\mathcal{R}(x)$ and $s\notin\mathcal{R}(y)$.}\end{array}
\end{equation}

\subsection{Cells and representations}\label{subsec_cellsreps}

\paragraph{Orders and cells.} We can define orders and cells using either of the two bases $\{C_w\}$ and $\{C^\da_w\}$ of $H(W)$. The definitions are exactly similar and we write them only for the basis $\{C_w\}$. Each time we are going to use these notions, we will be careful to indicate which basis we are using.

For two elements $w,w'\in W$ and $h\in H(W)$, we denote 
\[C_w\rightarrow_{h}C_{w'}\ \ \ \ \qquad\text{if $hC_w=\sum_{u\in W}\alpha_uC_u$ with $\alpha_{w'}\neq 0$.}\]
In words, $C_w\rightarrow_{h}C_{w'}$ means that $C_{w'}$ appears with a non-zero coefficient in $hC_w$ when expanded in the basis $\{C_w\}$. Similarly we denote $C_{w'}\leftarrow_{h}C_w$ if $C_{w'}$ appears with a non-zero coefficient in $C_wh$.

The left, right and two-sided orders are defined as follows: we set 
\[w'\preceq_{\mathcal{L}} w\ \ \ \ \qquad\text{if $C_w\rightarrow_{h_1}C_{w_1}\dots\rightarrow_{h_p}C_{w'}$ for some $h_1,\dots,h_p\in H(W)$.}\]
Similarly, we denote $w'\preceq_{\mathcal{R}} w$ if there is a sequence of arrows of the form $\leftarrow_{h}$ going from $C_w$ to $C_{w'}$, and $w'\preceq_{\mathcal{LR}} w$ if there is a sequence using both types of arrows.

Let $\mathcal{X}$ stands for respectively $\mathcal{L},\mathcal{R},\mathcal{LR}$. The left, right and two-sided cells are defined as the equivalence classes in $W$ for the equivalence relations defined by $w\sim_{\mathcal{X}}w'$ if and only if $w'\preceq_{\mathcal{X}} w$ and $w\preceq_{\mathcal{X}} w'$. Finally, a strict order relation $w\prec_{\mathcal{X}}w'$ means $w\preceq_{\mathcal{X}}w'$ and $w\nsim_{\mathcal{X}}w'$.

A property relating the cells and the descents is the following \cite[Lemma 8.6]{Lus03}:
\begin{equation}\label{descents}
w\preceq_{\mathcal{L}} w'\ \Rightarrow\ \mathcal{R}(w')\subseteq\mathcal{R}(w)\ \ \ \ \text{and}\ \ \ \ \ w\preceq_{\mathcal{R}} w'\ \Rightarrow\ \mathcal{L}(w')\subseteq\mathcal{L}(w)\ .
\end{equation}
This holds whichever of the two bases $\{C_w\}$ and $\{C^\da_w\}$ we use. In particular, it follows that if $w$ and $w'$ are in the same left cell then we have $\mathcal{R}(w)=\mathcal{R}(w')$ (and similarly for right cells and left descents).

\paragraph{Ideals and representations.} For a left cell $\Gamma$ in $W$ obtained from the basis $\{C_w\}$, define the following left ideals:
\[I_{\preceq_{\mathcal{L}}\Gamma}=\bigoplus_{w\preceq_{\mathcal{L}} y}\mathbb{Z}[q,q^{-1}]C_w\ \ \ \ \text{and}\ \ \ \ I_{\prec_{\mathcal{L}}\Gamma}=\bigoplus_{w\prec_{\mathcal{L}} y}\mathbb{Z}[q,q^{-1}]C_w\ ,\]
where $y$ is an element of $\Gamma$ (the resulting ideals depend only on $\Gamma$). The associated representation $V_{\Gamma}$ of $H(W)$ is constructed from the left multiplication on the quotient $I_{\preceq_{\mathcal{L}}\Gamma}/I_{\prec_{\mathcal{L}}\Gamma}$. It has by definition the following basis:
\[\{C_w+I_{\prec_{\mathcal{L}}\Gamma}\}_{w\in\Gamma}\ .\]
For a two-sided cell $\Gamma$ corresponding to the basis $\{C_w\}$, we define the two-sided ideals
\[I_{\preceq_{\mathcal{LR}}\Gamma}=\bigoplus_{w\preceq_{\mathcal{LR}} y}\mathbb{Z}[q,q^{-1}]C_w\ \ \ \ \text{and}\ \ \ \ I_{\prec_{\mathcal{LR}}\Gamma}=\bigoplus_{w\prec_{\mathcal{LR}} y}\mathbb{Z}[q,q^{-1}]C_w\ \ \ (\text{where $y\in\Gamma$}).\]

Finally, using the other basis $\{C^{\da}_w\}$, we can construct similarly from a left cell $\Gamma$ the corresponding ideals denoted $I^{\da}_{\preceq_{\mathcal{L}}\Gamma}$ and $I^\da_{\prec_{\mathcal{L}}\Gamma}$. For a two-sided cell, we have the two-sided ideals $I^\da_{\preceq_{\mathcal{LR}}\Gamma}$ and $I^\da_{\prec_{\mathcal{LR}}\Gamma}$.

The representation of $H(W)$ associated to a left cell $\Gamma$ for the basis $\{C^{\da}_w\}$ is denoted $V^{\da}_{\Gamma}$.

\subsection{The particular case of type A}\label{subsec_preltypeA}

For finite type $A$, the cells are nicely described in terms of the Robinson--Schensted correspondence and the cell representations are very well understood. We recall these known results \cite{Gec06,KL79}.

\paragraph{Generators and relations.} When the Coxeter system is of type $A_{n-1}$, thereby corresponding to the symmetric group $S_n$ on $n$ letters, we denote the corresponding Hecke algebra by $H(S_n)$. It has generators $T_1,\dots,T_{n-1}$ and relations (\ref{presH}) read in this case:
\[\begin{array}{l}T_iT_j=T_jT_i\ \ \ \ \text{if $|i-j|>1$,}\\[0.5em]
T_iT_{i+1}T_i=T_{i+1}T_iT_{i+1}\,, \ \ \ \\[0.5em]
T_i^2=(q-q^{-1})T_i+1\ .\end{array}\]
The notation $T_i$ is a shorthand notations for $T_{(i,i+1)}$ where $(i,i+1)$ is the transposition of $S_n$ swapping $i$ and $i+1$.

\paragraph{The Robinson--Schensted correspondence for permutations in $S_n$.} The Robinson--Schensted correspondence (RS for short) is a bijection between the set of permutations in $S_n$ and the set of pairs of standard Young tableaux of size $n$ and of the same shape.

Given a partition $\lambda\vdash n$, we denote by $\STab(\lambda)$ the set of standard Young tableaux of shape $\lambda$. We denote the RS correspondence as follows:
\begin{equation}\label{RS}
\begin{array}{rcl}
S_n & \leftrightarrow & \displaystyle\bigsqcup_{\lambda\vdash n} \STab(\lambda)^2\\[1em]
w & \leftrightarrow & \bigl(P(w),Q(w)\bigr)
\end{array}
\end{equation}
where we use the same convention as in \cite{Ful97,Knu70}. Given $w\in S_n$, the partition $\lambda$ which is the shape of $P(w)$ and $Q(w)$ is denoted $sh(w)$.

To avoid ambiguity, we indicate that permutations are composed from right to left, so that $s_1s_2$ is the permutation $\left(\begin{array}{ccc} 1 & 2 & 3\\ 2 & 3 & 1\end{array}\right)$. Applying the usual insertion algorithm to the sequence $231$, we have the following example of the RS correspondence, which should be enough to illustrate our conventions:
\[s_1s_2\ \ \leftrightarrow\ \ \bigl(\begin{array}{cc}
\fbox{\scriptsize{$1$}} & \hspace{-0.35cm}\fbox{\scriptsize{$3$}} \\[-0.2em]
\fbox{\scriptsize{$2$}} &
\end{array},\begin{array}{cc}
\fbox{\scriptsize{$1$}} & \hspace{-0.35cm}\fbox{\scriptsize{$2$}} \\[-0.2em]
\fbox{\scriptsize{$3$}} &
\end{array}\bigr)\]
We see in this example an illustration of the general property that the permutation $w$ has a generator $s_i$ in its left descent (in the example, $s_1$) if and only if the tableau $P(w)$ has $i$ in its descent, which means that $i+1$ is in a lower row than $i$. Similarly, $w$ has a generator $s_i$ in its right descent (in the example, $s_2$) if and only if the tableau $Q(w)$ has $i$ in its descent.

\paragraph{Cells and RS correspondence.} The cell structure of $S_n$ turns out to be intimately related to the RS correspondence. Using either the basis $\{C_w\}_{w\in S_n}$ or the basis $\{C^{\da}_w\}_{w\in S_n}$, the cells are the same and are described as follows:
\begin{itemize}
\item $w$ and $w'$ are in the same $\left\{\begin{array}{l} \text{left}\\[0.3em] \text{right}\\[0.3em] \text{2-sided} \end{array}\right.$ cell in $S_n$ if and only if $\left\{\begin{array}{l} Q(w)=Q(w')\\[0.3em] P(w)=P(w')\\[0.3em] sh(w)=sh(w') \end{array}\right.$;
\item $w\preceq_{\mathcal{LR}}w'$ if and only if $sh(w)\leq sh(w')$;
\end{itemize}
where in the last item, we use the dominance order on partitions.

\paragraph{Cells and representations.} Let $\Gamma$ be a left cell in $S_n$. As set up in Section \ref{subsec_cellsreps}, we denote $V_{\Gamma}$  the corresponding cell representation of $H(S_n)$ when we use the basis $\{C_w\}$ and we denote $V^{\da}_{\Gamma}$ the representation obtained when we use the basis $\{C^{\da}_w\}$.

In this paragraph, we consider the semisimple situation, namely we work over the field $\mathbb{C}(q)$ or with a non-zero complex number $q$ such that $q^2$ is not a root of unity whose order is between $2$ and $n$. We denote by $\{V_{\lambda}\}_{\lambda\vdash n}$ the set of irreducible representations of the Hecke algebra $H(S_n)$, using the standard indexation by partitions of $n$. In particular the one-dimensional representation $T_{w}\mapsto q^{\ell(w)}$ corresponds to the single-line partition. 

The cell representations of $H(S_n)$ are described as follows. Let $\lambda\vdash n$ and take $\Gamma$ any left cell of $S_n$ containing elements $w$ with $sh(w)=\lambda$. The isomorphism class of the cell representation of $H(S_n)$ corresponding to $\Gamma$ depends on which basis we use. We have:
\begin{itemize}
\item The cell representation $V_{\Gamma}$ is isomorphic to $V_{\lambda^t}$.
\item The cell representation $V^{\da}_{\Gamma}$ is isomorphic to $V_{\lambda}$.
\end{itemize}
where we use $\lambda^t$ to denote the transpose of the partition $\lambda$ (the partition obtained by exchanging the lines and columns of the Young diagram of $\lambda$).

\section{Parabolic Hecke algebras}\label{sec_para}

We keep $(W,S)$ an arbitrary Coxeter system. We let $J$ be a non-empty subset of $S$ and $W_J$ the corresponding parabolic subgroup of $W$. We assume that $W_J$ is finite.

\subsection{Double cosets of parabolic subgroups}

\paragraph{Minimal-length representatives.} The following classical facts can be found in \cite[chap. 2]{GP00}. We denote $X_J$ the set of distinguished representatives for the left cosets of $W_J$ in $W$. An element $x\in X_J$ is characterised by being the unique element of minimal length in its left coset $xW_J$ (or equivalently, the unique minimal element for the Bruhat order in $xW_J$). 
It satisfies that $\ell(xu)=\ell(x)+\ell(u)$ for any $u\in W_J$, and moreover (Deodhar Lemma) we have:
\[\text{for any $s\in J$,}\ \ \ \ \text{either $sx\in X_J$},\ \  \text{or $sx=xt$ for some $t\in J$.}\]
Similarly for right cosets, the set of minimal-length representative is $X_J^{-1}$.

In each double coset in $W_J\backslash W/W_J$, there is also a unique element of minimal length (or equivalently, minimal for the Bruhat order). We denote $X_{JJ}$ the set of minimal-length representatives for the double cosets of $W_J$ in $W$. We have $X_{JJ}=X_J\cap X_J^{-1}$. 

Minimal-length representatives are characterised in terms of their descents as follows:
\[x\in X_{J}\ \ \ \ \Leftrightarrow\ \ \ \ \ \mathcal{R}(x)\cap J=\emptyset\ .\]
\[x\in X_{JJ}\ \ \ \ \Leftrightarrow\ \ \ \ \ \mathcal{L}(x)\cap J=\mathcal{R}(x)\cap J=\emptyset\ .\]

\paragraph{Maximal length representatives.} Here we use that $W_J$ is finite, and thus that left, right or double cosets are finite. The following classical facts can be found in \cite[Theorem 1.2]{Cur85} or \cite{BKP+18}. There is a unique element of maximal length in each left coset (or equivalently, a unique maximal element in the Bruhat order). We denote by $\widetilde{X}_J$ these maximal-length representatives. If we denote $w_J$ the longest element of $W_J$, then we have:
\[\widetilde{X}_J:=\{xw_J\ ,\ \ x\in X_J\}\ .\]
Similarly for right cosets, the set of maximal-length representative is $\widetilde{X}_J^{-1}$.

In each double coset in $W_J\backslash W/W_J$, there is also a unique element of maximal length (or equivalently, maximal for the Bruhat order). We denote $\widetilde{X}_{JJ}$ the set of maximal-length representatives for the double cosets of $W_J$ in $W$. We have $\widetilde{X}_{JJ}=\widetilde{X}_{J}\cap \widetilde{X}^{-1}_{J}$.

In terms of descents, maximal-length representatives are characterised as follows
\[x\in \widetilde{X}_{J}\ \ \ \ \Leftrightarrow\ \ \ \ \ J\subset \mathcal{R}(x)\ .\]
\[x\in  \widetilde{X}_{JJ}\ \ \ \ \Leftrightarrow\ \ \ \ \ J\subset \mathcal{R}(x)\cap \mathcal{L}(x)\ .\]

\paragraph{Canonical expressions.} Given a double coset $\cD\in \dWJ$, we introduce a notation for its minimal-length element, and its maximal-length element:
\[\rmd\in \cD\cap X_{JJ}\ \ \ \ \text{and}\ \ \ \ \rpd\in \cD\cap \widetilde{X}_{JJ}\ .\]
Let $x=\rmd$ for some double coset $\cD$. The subset $J\cap xJx^{-1}$ of $S$ gives rise to a parabolic subgroup $W_{J\cap xJx^{-1}}$, which is a parabolic subgroup of $W_J$. As such, there is a set of distinguished representatives for left cosets of $W_{J\cap xJx^{-1}}$ in $W_J$, that we denote $X_{J\cap xJx^{-1}}^J$. 
With these notations, we have that any element $w\in \cD$ can be written uniquely as:
\begin{equation}\label{cosets}
w=w_1xw_2\,,\ \ \ \ \text{$w_2\in W_J$ and $w_1\in X_{J\cap xJx^{-1}}^J$\ .}
\end{equation}
and in this situation, we have $\ell(w)=\ell(w_1)+\ell(x)+\ell(w_2)$.  

Furthermore, the maximal-length element $\rpd$ of $\cD$ can be written as follows:
\begin{equation}\label{exp_r+}
\rpd=w_Jw_Lxw_J\,,\ \ \ \ \ \text{where $L=J\cap xJx^{-1}$},
\end{equation}
where $w_L,w_J$ denote the longest elements of the corresponding parabolic subgroups of $W$. Finally, the double coset $\cD$ consists of the full interval in the Bruhat order:
$$\cD=[\rmd,\rpd]\,,$$ 
that is, the double coset $\cD$ coincides with the set of all elements $y\in W$ such that $\rmd\leq y\leq \rpd$.

\begin{exam}
Take $W=S_6$ the symmetric group generated by $s_1,\dots,s_5$ (see Section \ref{sec-typeA} for the notations) and $W_J=S_2\times S_2\times S_2$ corresponding to $J=\{s_1,s_3,s_5\}$. We have $w_J=s_1s_3s_5$.

The element $x=s_2s_1s_4s_3s_2$ is minimal in its double coset. Here we have $J\cap xJx^{-1}=\{1\}$ (because $xs_3=s_1x$). The corresponding maximal element is $y=s_3s_5.x.s_1s_3s_5$.
\end{exam}

\paragraph{Bruhat order for double cosets.} Let $\cD$ and $\cD'$ be two double cosets in $\dWJ$. The Bruhat order on $W$ extends naturally to the double cosets $\dWJ$, through their distinguished representatives. That is, we set:
\[\cD\leq \cD'\ \ \ \ \Leftrightarrow\ \ \ \ \rmd\leq \rmdp\ .\]
In our situation ($W_J$ finite) we have another set of distinguished representatives, the maximal-length representatives, and they seem to provide an alternative choice for extending the Bruhat order to $\dWJ$. We will use the following result, asserting that the two choices are equivalent.
\begin{prop}\label{prop-order}
Let $\cD,\cD'\in\dWJ$. We have:
\[\rmd\leq \rmdp\ \ \ \ \ \Leftrightarrow\ \ \ \ \ \ \rpd\leq \rpdp\ .\]
\end{prop}
\begin{proof}
Let $\cD$ and $\cD'$ be distinct cosets. First let $w\in\cD$ and $w'\in\cD'$ such that $w< w'$. Take $s\in J\backslash \mathcal{L}(w)$. We have that:
\[\exists w''\in\cD'\ \ \ \text{such that}\ \ sw<w''\ .\]
Indeed, if $s\notin \mathcal{L}(w')$ then we can take $w''=sw'$. Whereas if $s\in \mathcal{L}(w')$ then the lifting property of the Bruhat order \cite[Prop. 2.2.7]{BB05} ensures that $sw\leq w'$. The strict inequality follows since $sw$ and $w'$ belong respectively to $\cD$ and $\cD'$ which are distinct.

Now let $x=\rmd$ and $y=\rmdp$ and assume that $x<y$. We have $xw_J<yw_J$ since $x$ and $y$ are in particular distinguished left representative. From (\ref{exp_r+}), a reduced expression for $\rpd$ is $\rpd=s_1\dots s_kxw_J$ for some $s_1,\dots,s_k$ in $J$. Since the expression is reduced, $s_i$ is not in the left descent of $s_{i+1}\dots s_kxw_J$. So we apply the above observation $k$ times and get an element $y'\in\cD'$ such that $\rpd<y'$. Since $y'\leq \rpdp$, we conclude that $\rpd\leq \rpdp$.

Finally assume that $x$ and $y$ are not comparable, and assume that $\rpd<\rpdp$. This implies by transitivity that $x<\rpdp$. Again by (\ref{exp_r+}) there is a reduced expression for $\rpdp$ of the form 
\[\rpdp=s_1\dots s_kyt_1\dots,t_l\,,\ \ \ \ \ \text{for some $s_1,\dots,s_k,t_1,\dots,t_l$ in $J$.}\] 
So a reduced expression of $x$ must be a subexpression of this, while not a subexpression of $y$. This implies that $x$ must have a reduced expression with some elements of $J$ on the left or on the right. This contradicts the fact that $x$ is a minimal-length representative.
\end{proof}

\paragraph{Poincar\'e polynomials.} The Poincar\'e polynomial of a finite Coxeter group $W'$ is:
\begin{equation}\label{Poincare-poly}
W'(q^2)=\sum_{w\in W'}q^{2\ell(w)}\ .
\end{equation}
Now take a parabolic subgroup $W_K$ of the finite Coxeter group $W_J$. Denote $X_J^K$ the set of minimal-length representatives of left cosets of $W_K$ in $W_J$. We have that any $x\in W_J$ is uniquely written as $x=du$ where $d\in X_J^K$ and $u\in W_K$, and such that $\ell(x)=\ell(d)+\ell(u)$. Therefore, we find:
\begin{equation}\label{div-poinc-pol}W_J(q^2)=W_K(q^2)\sum_{x\in X_K^J}q^{2\ell(x)}\ .
\end{equation}
So $W_K(q^2)$ divides $W_J(q^2)$ for any subset $K\subset J$ (in particular, $(1+q^2)$ always divides $W_J(q^2)$).

\subsection{Definition of parabolic Hecke algebras and standard basis}

From now on, we extend the algebra $H(W)$ over the following localization
\begin{equation}\label{def-ringA}
A:=\mathbb{Z}[q,q^{-1},W_J(q^2)^{-1}]\ ,
\end{equation}
where $W_J(q^2)$ is the Poincar\'e polynomial given in (\ref{Poincare-poly}).

\paragraph{Definition of the parabolic Hecke algebra.} We define:
\[\textbf{1}_J:=\sum_{w\in W_J}q^{\ell(w)}T_w\ .\]
The element $\textbf{1}_J$ is a quasi-idempotent, and, over the base ring $A$, we renormalise it to get an idempotent:
\begin{equation}\label{basic-prop-P}
e_{J}=\frac{1}{W_J(q^2)}\textbf{1}_J=\frac{1}{W_J(q^2)}\sum_{w\in W_J}q^{\ell(w)}T_w\ .
\end{equation}
This is sometimes called the $q$-symmetriser associated to the subalgebra of $H(W)$ generated by $T_s$ with $s\in J$, and with basis $\{T_w\}_{w\in W_J}$. The main property of $e_J$, implying that $e_J^2=e_J$, is:
\[T_we_J=e_JT_w=q^{\ell(w)}e_J\ \ \ \ \text{for all $w\in W_J$.}\]
\begin{defi}
The parabolic Hecke algebra $\HJ(W)$ is the algebra over $A$ defined by:
\[\HJ(W)=e_JH(W)e_J\ .\]
\end{defi}

\begin{rema}
More generally, one can define, as in \cite{APV13}, the parabolic Hecke algebra as the algebra  $\textbf{1}_JH(W)\textbf{1}_J$ directly over $\mathbb{Z}[q,q^{-1}]$. This algebra is not unital if $W_J(q^2)$ is not invertible. Over the extended base ring $A$, the two definitions coincide, and the algebra $\HJ(W)$ is unital with unit $e_J$.
\end{rema}

The idempotent $e_J$ is almost a Kazhdan--Lusztig basis element. Indeed, let $w_J$ be the longest element of the parabolic subgroup $W_J$. Then we have
\begin{equation}\label{relC_P}
C_{w_J}=\sum_{w\in W_J}q^{\ell(w)-\ell(w_J)}T_w=q^{-\ell(w_J)}\textbf{1}_J=q^{-\ell(w_J)}W_J(q^2)e_J\ .
\end{equation}
Note that Formula (\ref{relC_P}) shows in particular that the idempotent $e_J$ is bar invariant:
\begin{equation}\label{barPJ}
\overline{e_J}=e_J\ ,
\end{equation}
where the bar involution was defined in (\ref{involutions}).

\paragraph{Standard basis of $\HJ(W)$.} For a double coset $\cD\in\dWJ$, define:
\[
T_{\cD}:=\sum_{w\in \cD}q^{\ell(w)-\ell(\rpd)}T_w\ .
\]
The substraction of $\ell(\rpd)$ is a normalisation choice, and is such that the coefficient of the longest element $T_{\rpd}$ in $T_{\cD}$ is equal to 1 (note that all other coefficients are negative powers of $q$). 

The standard basis of $\HJ(W)$ is \cite{APV13,Cur85,CIK71}:
\begin{equation}\label{TBasis-par}
\{T_{\cD}\}_{\cD\in\dWJ}\ .
\end{equation}
It might not be immediately clear that the elements $T_D$ belong to the algebra $\HJ(W)$. In fact, if we denote $x=\rmd$ the minimal-length representative of $\cD$, from (\ref{cosets}) we deduce (see \cite{APV13}) that:
\begin{equation}\label{TD-PTxP}
T_{\cD}=\frac{W_J(q^2)^2}{W_{J\cap xJx^{-1}}(q^2)}q^{\ell(\rmd)-\ell(\rpd)}e_JT_xe_J\ .
\end{equation}
As recalled in (\ref{div-poinc-pol}), $W_{J\cap xJx^{-1}}(q^2)$ divides $W_J(q^2)$, and moreover, in the ring $A$, $W_J(q^2)$ and all its factors are invertible. So over the ring $A$, renormalising the basis elements above, we also have bases of $\HJ(W)$ of the form:
\begin{equation}\label{TBasis-par2}\{e_JT_we_J\}_{w\in R_J} \ \ \ \ \text{for any set $R_J$ of representatives of $\dWJ$.}
\end{equation}
Indeed we have that
$e_JT_we_J=q^{\ell(w)-\ell(x)}e_JT_{x}e_J$ whenever $w\in W_JxW_J$. This follows from the property (\ref{cosets}) of double cosets and from the basic property (\ref{basic-prop-P}) of $e_J$. Two natural choices for $R_J$ are of course $R_J=X_{JJ}$ and $R_J=\widetilde{X}_{JJ}$, the set of, respectively, minimal-length and maximal-length representatives.

\section{Kazhdan--Lusztig bases for parabolic Hecke algebras}\label{sec_KLpar}

\subsection{A first Kazhdan--Lusztig basis for $\HJ(W)$ and its cells}\label{subsec_KLbasispar1}

\subsubsection{The basis}
Let $\cD$ a double coset in $\dWJ$. Recall that $\rpd$ is the unique element of maximal-length in $\cD$. Since any $s\in J$  is in the left descent and in the right descent of $\rpd$, from property (\ref{TC=qC}), we have immediately that:
\begin{equation}\label{PCP=C}
e_JC_{\rpd}e_J=C_{\rpd}\ ,
\end{equation}
so that the elements $C_{\rpd}$ belong to $\HJ(W)$. The next result shows that the above set of elements forms a basis and compares it with the standard basis $\{T_{\cD}\}$ of $\HJ(W)$ defined in (\ref{TBasis-par}).
\begin{prop}\label{prop-CT-par}
The set $\{C_{\rpd}\}_{\cD\in \dWJ}$ is a basis of $\HJ(W)$, and we have:
\begin{equation}\label{CvsT-par}
C_{\rpd}=\sum_{\cD'\leq\cD}p_{\rpdp,\rpd}T_{\cD'}\ .
\end{equation}
Moreover, $C_{\rpd}$ is the unique element $B_{\cD}$ of $\HJ(W)$ satisfying
\[\overline{B_{\cD}}=B_{\cD}\ \ \ \text{and}\ \ \ B_{\cD}=T_\cD+\sum_{\cD'<\cD}a_{\cD'\cD}T_\cD'\ \ \text{with $a_{\cD'\cD}\in q^{-1}\mathbb{Z}[q^{-1}]$.}\]
\end{prop}
\begin{proof}
The fact that $\{C_{\rpd}\}_{\cD\in \dWJ}$ forms a basis of $\HJ(W)$ can be found in \cite[Theo. 1.10]{Cur85}. It follows from (\ref{CvsT-par}) which can be checked by a short direct calculation that we provide here. 

Let $x=\rpd$ and let $\cD'\in\dWJ$. Recall from (\ref{cosets}) that any element $y$ of $\cD'$ can be written as a reduced expression of the form:
\[y=s_1\dots s_k\rmd t_1\dots t_l\,,\ \ \ \ \text{with $s_i,t_i\in J\ .$}\]
Moreover any $s_i,t_i\in J$ are in the left and right descents of $x$ (maximal-length representative). Therefore we can use the property (\ref{psyx=qpyx}) of Kazhdan--Lusztig polynomials and obtain:
\begin{equation}\label{KLpol-par}
p_{y,x}=q^{\ell(y)-\ell(\rmdp)}p_{\rmdp,x}=q^{\ell(y)-\ell(\rpdp)}p_{\rpdp,x}\ \ \ \ \ \text{for any $y\in\cD'$.}
\end{equation}
The last equality uses $p_{\rpdp,x}=q^{\ell(\rpdp)-\ell(\rmdp)}p_{\rmdp,x}$, which is simply the first equality for $y=\rpdp$. Now we can calculate as follows:
\[\begin{aligned}
C_x &= \sum_{y\leq x} p_{y,x} T_y\\
&=  \sum_{\cD'\in\dWJ}\sum_{y\in\cD'} p_{y,x} T_y\\
&= \sum_{\cD'\in\dWJ}\Bigl( p_{\rpdp,x}\sum_{y\in\cD'}q^{\ell(y)-\ell(\rpdp)}T_y\Bigr)\\
&= \sum_{\cD'\in\dWJ} p_{\rpdp,x}T_{\cD'}\ .
\end{aligned}\]
In the last sum above, if $T_{\cD'}$ appears with a non-zero coefficient, this means that $T_{\rpdp}$ appeared with a non-zero coefficient in $C_x$, and this means that $\rpdp\leq \rpd$. According to Proposition \ref{prop-order}, this is equivalent to $\cD'\leq \cD$. The formula in the proposition expresses then a unitriangular change of basis, since the coefficient in front of $T_\cD$ is obviously $1$.

The element $C_{\rpd}$ is indeed stable under the bar involution, by property of the Kazhdan--Lusztig basis of $H(W)$. It is immediate that the first coefficient in the decomposition is 1 while the others are in $q^{-1}\mathbb{Z}[q^{-1}]$, from the similar properties of the polynomials $p_{x,y}$. It remains to prove the unicity statement. It is easily checked \cite[Theorem 5.2]{Lus03} that for an element $h=\sum_{w\in W}a_wT_w$ in the Hecke algebra with coefficients $a_w\in q^{-1}\mathbb{Z}[q^{-1}]$, we have that if $\overline{h}=h$ then $h=0$. If there is another element $X_{\cD}$ satisfying the required properties, then $X_{\cD}-C_{\rpd}$ is such an element $h$ and therefore is $0$.
\end{proof}

\begin{exam}\label{exam-basis1}
Take $W=S_4$ the symmetric group generated by $s_1,s_2,s_3$ and $W_J=S_2\times S_2$ corresponding to $J=\{s_1,s_3\}$. There are three double cosets: $[13]$, $[12321]$, $[121321]$, whose names reflect how their longest representatives write in terms of the generators $s_1,s_2,s_3$. The formulas illustrating the proposition are:
\[\begin{array}{l}
C_{13}=T_{[13]}\,,\\[0.5em]
C_{12321}=T_{[12321]}+(q^{-1}+q^{-3})T_{[13]}\,,\\[0.5em]
C_{121321}=T_{[121321]}+q^{-1}T_{[12321]}+q^{-4}T_{[13]}\ .
\end{array}\]
For example, $q^{-1}+q^{-3}$ is the Kazhdan--Lusztig polynomial $p_{s_1s_3,s_1s_2s_3s_2s_1}$ giving the coefficients of $T_{s_1s_3}$ in the expansion of $C_{s_1s_2s_3s_2s_1}$.
\end{exam}

\subsubsection{Cells in $\dWJ$ and representations of $\HJ(W)$} 

In this subsection, the orders, the cells and the associated representations of $H(W)$ are those constructed from the basis $\{C_{w}\}$ of $H(W)$. The orders and the cells on $\dWJ$ as well as the associated representations of $\HJ(W)$ are defined similarly, using the basis $\{C_{\rpd}\}$ of $\HJ(W)$ just obtained.

\begin{prop}\label{prop-cell1}$\ $\\
$\bullet$ Let $\mathcal{X}$ stands for $\mathcal{L}$ or $\mathcal{R}$ or $\mathcal{LR}$. We have
$$\text{$\cD\preceq_{\mathcal{X}} \cD'$ in $\dWJ$\ \ \ $\Longleftrightarrow$\ \ \ $\rpd\preceq_{\mathcal{X}}  \rpdp$ in $W$.}$$
$\bullet$ In particular, the left cells in $\dWJ$ (and similarly for right and double-sided cells) are the non-empty sets of the form:
\[\Gamma\cap \widetilde{X}_{JJ}\,,\ \ \ \ \ \ \text{for a left cell $\Gamma$ in $W$,}\]
where we have identified double cosets in $\dWJ$ with their maximal-length representatives in $\widetilde{X}_{JJ}$.
\end{prop}
\begin{proof}
Let $\cD,\cD'\in\dWJ$. If there is $h\in \HJ(W)$ such that $C_{\rpd}\rightarrow_{h}C_{\rpdp}$ in $\HJ(W)$ then seeing $h$ as an element of $H(W)$, it trivially implies the same property in $H(W)$.

Reciprocally, assume that there is $h\in H(W)$ such that $C_{\rpd}\rightarrow_{h}C_{\rpdp}$. Recall that the elements $C_{\rpd}$ and $C_{\rpdp}$ are invariant by left or right multiplication by $e_J$. So we have $e_Jhe_JC_{\rpd}=e_JhC_{\rpd}$ and thus $C_{\rpd}\rightarrow_{e_Jhe_J}e_JC_{\rpdp}=C_{\rpdp}$ in $\HJ(W)$. This shows the equivalence for the left order. A similar reasoning shows the desired result for the right multiplication, and the  first item follows. The second item is an immediate consequence of the first one.
\end{proof}

\paragraph{Example.} We take $W=S_6$ the symmetric group generated by $s_1,\dots,s_5$ and $W_J=S_2\times S_2\times S_2$ generated by $s_1,s_3,s_5$. Using the one-line notation for a permutation, here is a left cell in $S_6$:
\[\Gamma=\{615432,\ 625431,\ 635421,\ 645321,\ 546321\}\ .\]
Among these 5 elements, only two are in $\widetilde{X}_{JJ}$ (one has to look for those elements with left and right descents containing $1,3,5$) and the resulting cell in $\dWJ$ consists of the two double cosets of the following elements
\[\Gamma\cap\widetilde{X}_{JJ}=\{625431,\ 645321\}\ .\]
In this example, the left cell $\Gamma$ corresponds to permutations having their right tableau in the RS correspondence equal to $\begin{array}{cc}
\fbox{\scriptsize{$1$}} & \hspace{-0.35cm}\fbox{\scriptsize{$3$}} \\[-0.2em]
\fbox{\scriptsize{$2$}} &\\[-0.2em]
\fbox{\scriptsize{$4$}} &\\[-0.2em]
\fbox{\scriptsize{$5$}} & \\[-0.2em]
\fbox{\scriptsize{$6$}} &
\end{array}$. The intersection with $\widetilde{X}_{JJ}$ is not empty because we have chosen a tableau containing $1,3,5$ in its descent. Actually, the intersection is made of those permutations having a left tableau which also contains $1,3,5$ in its descent. In type A, the description of the cells of $\dWJ$ in terms of the RS correspondence is simple and will be explicited in the next section.

\paragraph{Representations of $\HJ(W)$.} Then we discuss the representations of $\HJ(W)$ induced by its left cells and relate them to the cell representations of $H(W)$. We note that for any representation $V$ of $H(W)$, the vector space $e_J(V)$ (which may be $\{0\}$) carries naturally a representation of $\HJ(W)$.

Let $\Gamma$ be a left cell for $W$ and $V_{\Gamma}$ the associated representation of $H(W)$, with basis $\{C_w+I_{\prec_{\mathcal{L}}\Gamma}\}_{w\in \Gamma}$. If non-empty, the subset $\Gamma\cap \widetilde{X}_{JJ}$ indexes the elements in a left cell of $\dWJ$, thanks to the proposition above. We denote $V_{\Gamma\cap \widetilde{X}_{JJ}}$ the associated representation of $\HJ(W)$. 

\begin{prop}\label{prop-rep1}
Let $\Gamma$ be a left cell of $W$ such that $\Gamma\cap \widetilde{X}_{JJ}\neq \emptyset$. As representations of $\HJ(W)$, we have 
$$V_{\Gamma\cap \widetilde{X}_{JJ}}=e_J(V_{\Gamma})\ .$$
\end{prop}
\begin{proof}
Recall that $\widetilde{X}_{J}^{-1}$ denotes the set of maximal-length representatives of right cosets in $W_J\backslash W$, or equivalently, the set of elements containing $J$ in their left descent. First we note that the set of elements:
\[\{C_x\}_{x\in \widetilde{X}_{J}^{-1}}\]
is a basis of the subspace $e_JH(W)$ (the image of $H(W)$ by left multiplication by $e_J$). The proof is similar to the proof of Proposition \ref{prop-CT-par}. Namely, we take $x\in \widetilde{X}_{J}^{-1}$ and we write, using the same arguments, that:
\[C_x = \sum_{x'\in W_J\backslash W}\sum_{y\in W_Jx'} p_{y,x} T_y= \sum_{x'\in W_J\backslash W} p_{x',x}\Bigl(\sum_{y\in W_Jx'}q^{\ell(y)-\ell(x')}T_y\Bigr)\ .\]
The sum in parenthesis, defined for any coset in $W_J\backslash W$ forms a basis of $e_JH(W)$ (since $e_JT_z$ is proportional to it for any $z\in W_Jx'$) and the above formula thus expresses a triangular change of basis.

Then we take $\Gamma$ an arbitrary left cell of $W$, and we prove that the following subset is a basis of the subspace $e_J(V_{\Gamma})$:
\begin{equation}\label{basis-rep1}
\{C_w+I_{\prec_{\mathcal{L}}\Gamma}\}_{w\in \Gamma\cap \widetilde{X}_{J}^{-1}}\ .
\end{equation}
From property (\ref{TC=qC}), we have immediately that $e_JC_w=C_w$ if $w\in \widetilde{X}_{J}^{-1}$, so that the elements in (\ref{basis-rep1}) indeed belong to $e_J(V_{\Gamma})$. They are obviously linearly independent, as a subset of the basis of $V_{\Gamma}$. Then take $w\in \Gamma$ which is not in $\widetilde{X}_{J}^{-1}$. Inside $e_JH(W)$, the element $e_JC_w$ decomposes in the basis $\{C_x\}_{x\in \widetilde{X}_{J}^{-1}}$, and moreover uses only elements in $\Gamma$ or strictly less than $\Gamma$ (for $\prec_{\mathcal{L}}$). Thus, $e_JC_w+I_{\prec_{\mathcal{L}}\Gamma}$ can be written in terms of the elements in (\ref{basis-rep1}), which is therefore a generating set of $e_J(V_{\Gamma})$.

Now we assume that $\Gamma$ is such that $\Gamma\cap \widetilde{X}_{JJ}\neq \emptyset$. It means that there is an element in $\Gamma$ which contains $J$ in its right descent. Since the right descents of all elements in the same left cell coincide (see (\ref{descents}), this means that all elements in $\Gamma$ contain $J$ in their right descent. In this case, the basis above of $e_J(V_{\Gamma})$ becomes:
\[\{C_w+I_{\prec_{\mathcal{L}}\Gamma}\}_{w\in \Gamma\cap \widetilde{X}_{JJ}}\ .\]
This identifies immediately with the basis of the cell representation $V_{\Gamma\cap \widetilde{X}_{JJ}}$ of $\HJ(W)$ and the $\HJ(W)$-module structures (by left multiplication) are the same.
\end{proof}

\subsection{A second Kazhdan--Lusztig basis for $\HJ(W)$ and its cells} \label{subsec_KLbasispar2}

In the previous subsection, we have studied a Kazhdan--Lusztig basis of $\HJ(W)$ coming from the basis $\{C_w\}$ of $H(W)$. In this section we want to use the other basis $\{C^{\da}_w\}$ of $H(W)$. We will see that it behaves quite differently with respect to the idempotent $e_J$. In type A, this basis will be relevant for the applications to Schur--Weyl duality in Section \ref{sec-SW}.

\subsubsection{The basis} The basis $\{C_{\rpd}\}$ of the previous section has the property that $e_JC_{\rpd}e_J=C_{\rpd}$, so that it answers nicely the question of describing the image of the two-sided multiplication by $e_J$ on $H(W)$.

We can also wonder about the kernel of the two-sided multiplication by $e_J$ on $H(W)$. It turns out that there is also a nice description, this time in terms of the other Kazhdan--Lusztig basis $\{C^\da_w\}$ of $H(W)$. Indeed, we have:
\[e_JC^{\da}_we_J=0\ \ \ \ \text{for all $w\notin X_{JJ}$.}\]
To see this, let $w\notin X_{JJ}$. This means that there is some $s\in J$ such that $s\in\mathcal{R}(w)$ or $s\in\mathcal{L}(w)$. Say $s\in\mathcal{R}(w)$. So we have:
\[C^{\da}_we_J=C^{\da}_w\frac{1+qT_s}{1+q^2}e_J=0\ .\]
The first equality uses $T_se_J=qe_J$ if $s\in J$, while the second equality uses $C^{\da}_wT_s=-q^{-1}C^{\da}_w$ if $s\in\mathcal{R}(w)$. Recall also that $(1+q^2)$ is invertible in $A$. A similar proof works if $s\in\mathcal{L}(w)$.

In fact the elements $C^{\da}_w$ with $w\notin X_{JJ}$ form a basis of the kernel of the two-sided multiplication by $e_J$, and we obtain a basis of $\HJ(W)$ from the remaining elements, as we show in the following proposition. Note that for this second basis, it is more natural to work with the renormalised standard basis $\{e_JT_{\rmd}e_J\}$ of $\HJ(W)$, see \eqref{TD-PTxP} and 
\eqref{TBasis-par2}.
\begin{prop}\label{prop-CT-par2}
The set $\{e_JC^{\da}_{\rmd}e_J\}_{\cD\in\dWJ}$ is a basis of $\HJ(W)$ and we have:
\begin{equation}\label{CvsT-par2}
e_JC^{\da}_{\rmd}e_J=\sum_{\cD'\leq\cD}a_{\cD',\cD}e_JT_{\rmdp}e_J\ ,\ \ \ \text{where}\ a_{\cD',\cD}=\sum_{y\in\cD'}  (-1)^{\ell(y)}q^{\ell(y)-\ell(\rmdp)}\overline{p}_{y,\rmd}\ .
\end{equation}
Moreover, $e_JC^{\da}_{\rmd}e_J$ is the unique element $B_{\cD}$ of $\HJ(W)$ satisfying
\[\overline{B_\cD}=B_\cD\ \ \ \text{and}\ \ \ B_\cD=(-1)^{\ell(\rmd)}e_JT_{\rmd}e_J+\sum_{\cD'<\cD}\alpha_{\cD'\cD}e_JT_{\rmdp}e_J\ \ \text{with $\alpha_{\cD'\cD}\in q\mathbb{Z}[q]$.}\]
\end{prop}
\begin{proof}
The proof is a calculation similar to the proof of Proposition \ref{prop-CT-par}. We use that 
$$e_JT_ye_J=q^{\ell(y)-\ell(\rmdp)}e_JT_{\rmdp}e_J\,,$$ 
if $y$ is in the double coset $\cD'$. To give details, we let $x=\rmd$ and we write:
\[\begin{aligned}
e_JC^{\da}_{x}e_J &= \sum_{y\leq x} (-1)^{\ell(y)}\overline{p}_{y,x} e_JT_ye_J\\
&=  \sum_{\cD'\in\dWJ}\sum_{y\in\cD'}  (-1)^{\ell(y)}\overline{p}_{y,x} e_JT_ye_J\\
&= \sum_{\cD'\in\dWJ}\Bigl(\sum_{y\in\cD'}  (-1)^{\ell(y)}\overline{p}_{y,x}q^{\ell(y)-\ell(\rmdp)}\Bigr) e_JT_{\rmdp}e_J\ .
\end{aligned}\]
The double cosets $\cD'$ appearing with non-zero coefficient must contain an element $y\in\cD'$ such that $y\leq x$. This implies that $\rmdp\leq x$. So the sum can be restricted to $\cD'\leq\cD$. Moreover it is immediate that $a_{\cD,\cD}=(-1)^{\ell(\rmd)}$ since the only element $y$ in $\cD$ satisfying $y\leq x$ is $y=x$, due to the minimality of $x=\rmd$. Thus the formula expresses a unitriangular (up to a sign) change of basis.

The stability of $e_JC^{\da}_{\rmd}e_J$ under the bar involution is immediate since each factor is stable, see (\ref{barPJ}). The first coefficient was calculated before and the fact that the other coefficients $a_{\cD',\cD}$ are in $q\mathbb{Z}[q]$ is immediate since the polynomials $\overline{p}_{y,\rmd}$ are in $q\mathbb{Z}[q]$ for $y\in\cD'$. The unicity statement is proved exactly as in the end of the proof of Proposition \ref{prop-CT-par} (with $q$ instead of $q^{-1}$).
\end{proof}
\begin{rema}
The coefficients in (\ref{CvsT-par2}) are polynomials in $q$ with integer coefficients, but they do not have to be in $\mathbb{Z}_{\geq 0}[q]$ or in $\mathbb{Z}_{\leq 0}[q]$ as shown in the second example below.
\end{rema}
\begin{exam}
$\bullet$ Take $W=S_4$ and $W_J=S_2\times S_2$ as in Example \ref{exam-basis1}. The three double cosets have as minimal representatives: $e, s_2, s_2s_1s_3s_2$. The formulas illustrating the proposition are:
 \[\begin{array}{l}
e_JC^\da_{e}e_J=e_J  \\[0.5em]
e_JC^{\da}_{s_2}e_J=-e_JT_{s_2}e_J+qe_J\\[0.5em]
e_JC^{\da}_{s_2s_1s_3s_2}e_J=e_JT_{s_2s_1s_3s_2}e_J-(q+q^3)e_JT_{s_2}e_J+q^2e_J
\end{array}\]
For example, in $C^{\da}_{s_2s_1s_3s_2}$, the terms corresponding to the trivial double coset are $q^2T_{s_1s_3}-q^3T_{s_1}-q^3T_{s_3}+(q^2+q^4)T_{e}$, so that the $q^2$ in front of $e_J$ above is obtained as $q^4-q^4-q^4+(q^2+q^4)$.

\vskip .1cm
$\bullet$ Take $W=S_4$ and $W_J=S_1\times S_2\times S_1$. The parabolic subgroup is generated by $s_2$. There are 7 double cosets and one of the minimal-length representatives is $s_1s_2s_3s_2s_1$. The formula in the proposition is:
\[eC^\da_{12321}e=-eT_{12321}e+q^2 eT_{123}e+q^2eT_{321}e+(q-q^3)eT_{13}e-q^2eT_1e-q^2eT_3e+q^3e\,,\]
where we have abbreviated $e_J$ by $e$, and the generators $s_i$ by their letters $i$. We note the coefficient $(q-q^3)$ in front of $eT_{13}e$ containing both signs $\pm1$. It is obtained by looking at the following coefficients in $C^\da_{12321}$:
\[C^\da_{12321}=...+(q+q^3)T_{13}-q^2 T_{213}-q^2 T_{132}+ 0T_{2132}\,,\]
which are the cofficients in front of the elements in the double coset of $s_1s_3$. The cofficient $(q-q^3)$ is obtained as $(q+q^3)-q^3-q^3$.
\end{exam}

\begin{rema}
Looking at the basis in the previous proposition, we may wonder why not considering the set of elements $\{e_JC_{\rmd}e_J\}_{\cD\in\dWJ}$. In fact, one can prove exactly as above that this is indeed a basis of $\HJ(W)$ and that we have:
\begin{equation}\label{CvsT-par3}
e_JC_{\rmd}e_J=\sum_{\cD'\leq\cD}b_{\cD',\cD}e_JT_{\rpdp}e_J\ ,
\end{equation}
where $b_{\cD',\cD}=\sum_{y\in\cD'} q^{\ell(y)-\ell(\rpdp)}p_{y,x}$. These coefficients are polynomials in $q^{-1}$ with integer coefficients, thanks to the use of the basis elements $e_JT_{\rpdp}e_J$ instead of $e_JT_{\rmdp}e_J$. Obviously, if the polynomials $p_{y,x}$ are in $\mathbb{Z}_{\geq 0}[q^{-1}]$ then so are the coefficients $b_{\cD',\cD}$.

Still, this basis is less natural and less easy to handle than the basis in Proposition \ref{prop-CT-par}, and unlike the basis in Proposition \ref{prop-CT-par2}, it does not play any role in our study of the Schur--Weyl duality in Section \ref{sec-SW}. So we will not consider it further.
\end{rema}

\subsubsection{Cells in $\dWJ$ and representations of $\HJ(W)$} 

In this subsection, the orders, the cells and the associated representations of $H(W)$ are those constructed from the basis $\{C^{\da}_{w}\}_{w\in W}$ of $H(W)$. The orders, the cells and the associated representations of $\HJ(W)$ are defined similarly, using the basis $\{e_JC^{\da}_{\rmd}e_J\}_{\cD\in\dWJ}$ of $\HJ(W)$ just obtained.

\paragraph{Cells in $\dWJ$.} Here is the statement that is valid in general for the type of cells in $\dWJ$ considered in this section.
\begin{prop}\label{prop-cell2a} Let $\mathcal{X}$ stands for $\mathcal{L}$ or $\mathcal{R}$ or $\mathcal{LR}$. We have
$$\text{$\cD\preceq_{\mathcal{X}} \cD'$ in $\dWJ$\ \ \ $\Longrightarrow$\ \ \ $\rmd\preceq_{\mathcal{X}}  \rmdp$ in $W$.}$$
\end{prop}
\begin{proof}
Assume that there exists $h\in\HJ(W)$ such that $e_JC^{\da}_{r^-(\cD)}e_J$ appears in $he_JC^{\da}_{\rmdp}e_J$, which is equal to $e_JhC^{\da}_{\rmdp}e_J$ since $h\in\HJ(W)$. Now expand $hC^{\da}_{\rmdp}$ in the basis $\{C^{\da}_w\}$ in $H(W)$ and then multiply on both sides by $e_J$. All terms $C^{\da}_w$ with $w\notin X_{JJ}$ give $0$. This shows that $C^{\da}_{r^-(\cD)}$ appears in $hC^{\da}_{r^-(\cD')}$ and this shows the implication for the left order. The verification for the right order is the same and this implies the implication for the two-sided order.
\end{proof}
Note that we do not prove the equivalence (in contrast with Proposition \ref{prop-cell1}). Nevertheless, we prove below a description of the cells under an irreducibility assumption for the cell modules. Due to this irreducibility assumption, the description is less complete than for the previous basis. However, this will be enough for type A, where all cell representations are irreducible over $\mathbb{C}(q)$. For brevity, we treat only the left cells.
\begin{prop}\label{prop-cell2b} Let $\cD\in\dWJ$ and $\Gamma$ the left cell in $W$ containing $\rmd$. Assume that the corresponding representation $V^\da_{\Gamma}$ is irreducible for $H(W)$ over $\mathbb{C}(q)$. We have
\begin{equation}\label{cellequiv2}
\text{$\cD\sim_{\mathcal{L}} \cD'$ in $\dWJ$\ \ \ $\Longleftrightarrow$\ \ \ $\rmd\sim_{\mathcal{L}}  \rmdp$ in $W$.}
\end{equation}
Assume in particular that all cell representations $V^\da_{\Gamma}$ are irreducible for $H(W)$ over $\mathbb{C}(q)$. Then the left cells in $\dWJ$ are the non-empty sets of the form:
\[\Gamma\cap X_{JJ}\,,\ \ \ \ \ \ \text{for a left cell $\Gamma$ in $W$,}\]
where we have identified double cosets in $\dWJ$ with their minimal-length representatives in $X_{JJ}$. 
\end{prop}
\begin{proof}
From Proposition \ref{prop-cell2a}, we already have the direct implication of \eqref{cellequiv2}. For the reverse implication, assume that $\rmd\sim_{\mathcal{L}}  \rmdp$ in $W$, so that both $\rmd,\rmdp$ are in the same cell $\Gamma$. We need to show that $\cD\preceq_{\mathcal{L}} \cD'$ in $\dWJ$.

First we note that we can write:
\[e_JC^\da_{\rmdp}=C^\da_{\rmdp}+\sum_{x\neq\rmdp}\alpha_xC^\da_x\ .\]
Indeed first we decompose the element $e_JC^\da_{\rmdp}$ in the basis $\{C^\da_w\}$ of $H(W)$ as $e_JC^\da_{\rmdp}=\sum_w\alpha_w C_w^{\da}$ and we multiply from left and right by $e_J$. We find that the coefficient in front of $C^\da_{\rmdp}$ must be equal to $1$, and for that we used that $e_JC^\da_{\rmdp}e_J$ is a basis element of $\HJ(W)$ and all other terms $e_JC^\da_we_J$ are either $0$ or different basis elements. 

Since $\rmdp$ is in the cell $\Gamma$, the element $e_JC^\da_{\rmdp}$ is in the ideal $I^\da_{\preceq_{\mathcal{L}}\Gamma}$, and moreover the preceding discussion shows that it is not in $I^\da_{\prec_{\mathcal{L}}\Gamma}$ (because of the non-zero coefficient in front of $C^\da_{\rmdp}$). Recalling that $V^\da_{\Gamma}$ is defined as the quotient of the left ideal $I^\da_{\preceq_{\mathcal{L}}\Gamma}$ by the left ideal $I^\da_{\prec_{\mathcal{L}}\Gamma}$, this means that
\begin{equation}\label{PJCda}
e_JC^\da_{\rmdp}+I_{\prec_{\mathcal{L}}\Gamma}\neq 0_{V^\da_{\Gamma}}\ .
\end{equation}
Now, given another element $y\in V^\da_{\Gamma}$, from the irreducibility assumption on $V^\da_{\Gamma}$ it is always possible to find an element $h\in \mathbb{C}(q)H(q)$ sending $e_JC^\da_{\rmdp}$ to this element $y$. We choose for $y$ the element $C^{\da}_{\rmd}+I^\da_{\prec_{\mathcal{L}}\Gamma}$. It is possible to do so since $\rmd$ is also in the cell $\Gamma$.
We obtain that
\[he_JC^\da_{\rmdp}=C^\da_{\rmd}+x\,,\ \ \ \ \text{with $x\in I_{\prec_{\mathcal{L}}\Gamma}$. }\]
for some $h\in \mathbb{C}(q)H(q)$. Multiplying from left and right by $e_J$, we find
\begin{equation}\label{eJheJC}
e_Jhe_JC^\da_{\rmdp}e_J=e_JC^\da_{\rmd}e_J+e_Jxe_J \,.\end{equation}
The element $x\in I_{\prec_{\mathcal{L}}\Gamma}$ writes in terms of basis elements $C_w^\da$ with $w\neq \rmd$ and therefore $e_Jxe_J$ writes in terms of basis elements $e_JC_w^\da e_J$ different from $e_JC^\da_{\rmd}e_J$ (here we used again that $e_JC_w^\da e_J$ is either $0$ or directly a basis element). Therefore, when the rhs of (\ref{eJheJC}) is written in the basis $\{e_JC_w^\da e_J\}$, the element $e_JC^\da_{\rmd}e_J$ appears with coefficient $1$.

So far, the element $e_Jhe_J$ by which we multiply belongs to $\mathbb{C}(q)\HJ(W)$. But we can multiply by a suitable element of the base ring $A$ to produce an equality in $\HJ(W)$. And we have found $e_JC^\da_{\rmd}e_J$ with a non-zero coefficient, so we conclude that $\cD\preceq_{\mathcal{L}} \cD'$ in $\dWJ$, as required.
\end{proof}

\paragraph{Example.} We take $W=S_6$ the symmetric group generated by $s_1,\dots,s_5$ and $W_J=S_2\times S_2\times S_2$ generated by $s_1,s_3,s_5$. Using the one-line notation for a permutation, here is a left cell in $S_6$:
\[\Gamma=\{231456,\ 132456,\ 142356,\ 152346,\ 162345\}\ .\]
Among these 5 elements, only two are in $X_{JJ}$ (one has to look for those elements with left and right descents disjoints from $\{1,3,5\}$) and the resulting cell in $\dWJ$ consists of the two double cosets of the following elements
\[\Gamma\cap X_{JJ}=\{132456,\ 152346\}\ .\]
In this example, the left cell $\Gamma$ corresponds to permutations having their right tableau in the RS correspondence equal to $\begin{array}{ccccc}
\fbox{\scriptsize{$1$}} & \hspace{-0.35cm}\fbox{\scriptsize{$2$}}& \hspace{-0.35cm}\fbox{\scriptsize{$4$}}& \hspace{-0.35cm}\fbox{\scriptsize{$5$}}& \hspace{-0.35cm}\fbox{\scriptsize{$6$}} \\[-0.2em]
\fbox{\scriptsize{$3$}}\end{array}$. The intersection with $X_{JJ}$ is not empty because we have chosen a tableau not containing $1,3,5$ in its descent. Actually, the intersection is made of those permutations having a left tableau which also does not contain $1,3,5$ in its descent. In type A, the description of the cells of $\dWJ$ in terms of the RS correspondence is simple and will be explicited in the next section.

\paragraph{Representations of $\HJ(W)$.}
We recall again that for any representation $V$ of $H(W)$, the vector space $e_J(V)$ carries naturally a representation of $\HJ(W)$, and moreover if $V$ is an irreducible $H(W)$-module then $e_J(V)$ (if non-zero) is an irreducible $\HJ(W)$-module, see for example \cite[\S 6.2]{Gre80}.

Let $\Gamma$ be a left cell for $W$ and $V^{\da}_{\Gamma}$ the associated representation of $H(W)$, with basis $\{C^{\da}_w+I_{\prec_{\mathcal{L}}\Gamma}\}_{w\in \Gamma}$. If non-empty, the subset $\Gamma\cap X_{JJ}$ indexes the elements in a left cell of $\dWJ$, thanks to the proposition above. We denote $V^\da_{\Gamma\cap X_{JJ}}$ the associated representation of $\HJ(W)$. 
\begin{prop}\label{prop-rep2}
Let $\Gamma$ be a left cell of $W$ such that $\Gamma\cap X_{JJ}\neq \emptyset$. As representations of $\HJ(W)$, we have 
$$V^{\da}_{\Gamma\cap X_{JJ}}=e_J(V^{\da}_{\Gamma})\ .$$
\end{prop}
\begin{proof}
Recall that $X_J^{-1}$ denotes the set of minimal-length representatives of right cosets in $W_J\backslash W$, or equivalently, the set of elements with left descent disjoint from $J$. First we note that the set of elements:
\begin{equation}\label{basis-PH}
\{e_JC^{\da}_x\}_{x\in X_J^{-1}}
\end{equation}
is a basis of the subspace $e_JH(W)$ (the image of $H(W)$ by left multiplication by $e_J$). Indeed, using the same arguments than before Proposition \ref{prop-CT-par2}, we see that $e_JC^{\da}_x=0$ if $x$ is not in $X_J^{-1}$. So the above set is a spanning set. The linear independence easily follows from the fact that $e_JC^{\da}_x$ decomposes in the standard basis using only elements $T_y$ with $y\leq w_Jx$, with an invertible coefficients in front of  $T_{w_Jx}$.

Next we deduce that the following subset is a basis of the subspace $e_J(V^{\da}_{\Gamma})$:
\begin{equation}\label{basis-rep2}
\{e_JC^{\da}_w+I^{\da}_{\prec_{\mathcal{L}}\Gamma}\}_{w\in \Gamma\cap X_J^{-1}}\ .
\end{equation}
Again we have that $e_JC^{\da}_w=0$ if $w\notin X_J^{-1}$, so that the elements in (\ref{basis-rep2}) span $e_J(V^{\da}_{\Gamma})$. Besides, a relation of linear dependency between these elements would contradict the linear independence of the basis elements of $e_JH(W)$ described above.

Since $\Gamma\cap X_{JJ}\neq \emptyset$ (the element $\rmd$ is in here), it means that there is at least one element in $\Gamma$ with right descent disjoint from $J$. Since the right descents of all elements in the same left cell coincide, this means that all elements in $\Gamma$ have their right descents disjoint from $J$. In this case, the basis above of $e_J(V^{\da}_{\Gamma})$ becomes:
\begin{equation}\label{basiscellrepproof}
\{e_JC^{\da}_w+I^{\da}_{\prec_{\mathcal{L}}\Gamma}\}_{w\in \Gamma\cap X_{JJ}}\ .
\end{equation}
This identifies naturally with the defining basis $\{e_JC^{\da}_we_J+I^{\da}_{\prec_{\mathcal{L}}\Gamma\cap X_{JJ}}\}_{w\in \Gamma\cap X_{JJ}}$ of the cell representation $V^{\da}_{\Gamma\cap X_{JJ}}$ of $\HJ(W)$. It remains to check that the actions of $\HJ(W)$ are the same.

Let $h\in\HJ(W)$. We write, for $w\in \Gamma\cap X_{JJ}$,
\[he_JC_w^\da=hC^{\da}_w=\sum_{w'\in\Gamma}\alpha_{w'}C^\da_{w'}+\sum_{y\prec_{\mathcal{L}}\Gamma}\alpha_yC^\da_y\,.\]
The second sum is in $I^{\da}_{\prec_{\mathcal{L}}\Gamma}$ and therefore the coefficients $\alpha_{w'}$ give the action of $h$ on the basis (\ref{basiscellrepproof}). Now multiply this equality by $e_J$ from left and right to get
\[he_JC_w^\da e_J=\sum_{w'\in\Gamma\cap X_{JJ}}\alpha_{w'}e_JC^\da_{w'}e_J+\sum_{\substack{y\prec_{\mathcal{L}}\Gamma\\ y\in X_{JJ}}}\alpha_ye_JC^\da_ye_J\,,\]
where we have been using again that $e_JC^\da_xe_J=0$ when $x\notin X_{JJ}$. The second sum is clearly in $I^{\da}_{\prec_{\mathcal{L}}\Gamma\cap X_{JJ}}$ by construction of this ideal, and therefore the coefficients $\alpha_{w'}$ also give the action of $h$ in the representation $V^{\da}_{\Gamma\cap X_{JJ}}$ of $\HJ(W)$.
\end{proof}

\section{Parabolic Hecke algebras in type A}\label{sec-typeA}

\subsection{Notations}

From now on, using Notations from Section \ref{subsec_preltypeA}, we take $W=S_n$ the symmetric group on $n$ letters with generators $s_i=(i,i+1)$, $i=1,\dots,n-1$. The associated Hecke algebra is denoted $H(S_n)$ and its generators are $T_1,\dots,T_{n-1}$.

\paragraph{Parabolic subgroups $S_{\mu}$.} All parabolic subgroups of $S_n$ are obtained as
\[W_J=S_{\mu}=S_{\mu_1}\times\dots\times S_{\mu_d}\ ,\]
where $d>1$ and $\mu=(\mu_1,\dots,\mu_d)\in\mathbb{Z}_{>0}$ such that $\mu_1+\dots+\mu_d=n$. 

From now on we fix such an integer $d$ and such a composition $\mu$ (with $d$ non-zero parts). The parabolic subgroup $S_{\mu}$ is naturally embedded in $S_n$ and is the parabolic subgroup corresponding to the following subset of simple transpositions:
\begin{equation}\label{def-Jmu}
J=\{1,\dots,\mu_1-1,\ \ \ \mu_1+1,\dots,\mu_1+\mu_2-1,\ \ \ \dots\dots\}\ ,
\end{equation}
where we have identified a generator $s_i$ with its index. In other words, to get $J$, we remove from $\{1,\dots,n-1\}$ the indices $\mu_1,\mu_1+\mu_2,\dots,\mu_1+\dots+\mu_{d-1}$.

\paragraph{Parabolic Hecke algebra $\Hmu$.}
In the Hecke algebra $H(S_n)$, the subalgebra generated by the subset $T_s$, with $s\in J$, is isomorphic in this case to $H(S_{\mu_1})\otimes\dots\otimes H(S_{\mu_d})$. In each subalgebra $H(S_k)$, we have the $q$-symmetriser (normalised to be an idempotent) which is:
\[e_k=\frac{1}{\sum_{w\in S_k}q^{2\ell(w)}}\sum_{w\in S_k}q^{\ell(w)}T_w\ .\]
The idempotent in $H(S_n)$ corresponding to the choice of the composition $\mu$ is:
\begin{equation}\label{P-typeA}
e_{\mu}=e_{\mu_1}\otimes\dots\otimes e_{\mu_d}=\frac{1}{\sum_{w\in S_{\mu}}q^{2\ell(w)}}\sum_{w\in S_{\mu}}q^{\ell(w)}T_w\ .
\end{equation}
The Poincar\'e polynomial of $S_n$ satisfies : $\sum_{w\in S_{n}}q^{2\ell(w)}=\prod_{a=1}^n(1+q^2+\dots+q^{2(a-1)})$. So here the normalizing factor in $e_{\mu}$ is
\[\prod_{i=1}^d\bigl(\sum_{w\in S_{k_i}}q^{2\ell(w)}\bigr)=\prod_{i=1}^d[k_i]_q!\]
where we have set $[n]_q!=[2]_q[3]_q\dots[n]_q$ and $[m]_q=\displaystyle\frac{1+q^{2m}}{1+q^2}$ for any integer $m$. So the ground ring for the parabolic Hecke algebra is in this case:
\begin{equation}\label{def-ring}A=\mathbb{Z}[q,q^{-1},([K]_q!)^{-1}]\,,\ \ \ \ \text{where $K=\text{Max}\{\mu_1,\dots,\mu_d\}$.}\end{equation}
We repeat the definition, for fixing the notations, of the parabolic Hecke algebra of type A.
\begin{defi}
We denote $\Hmu$ the parabolic Hecke algebra associated to $W=S_n$ and $J$ as above. It is the following subalgebra of $H(S_n)$, defined over $A$ in (\ref{def-ring}),
\[\Hmu=e_{\mu}H(S_n)e_{\mu}\,,\]
where $e_{\mu}$ is the idempotent defined in (\ref{P-typeA}).
\end{defi}
A diagrammatic, braid-like, description of the algebra $\Hmu$ was described in \cite{CP23} and the algebra was called the ``fused Hecke algebra''.

\begin{rema}\label{rem-spec}
In $A$ are invertible all polynomials of the form $(1+q^2+\dots+q^{2(k-1)})$, with $k$ smaller or equal to some $\mu_i$. The possible specializations of $q$ to, say, a complex number are those such that $q^2$ is not a root of unity whose order is between $2$ and $K=\text{Max}\{\mu_1,\dots,\mu_d\}$.
\end{rema}

\paragraph{Standard basis of $\Hmu$.} The standard basis $\{T_{\cD}\}$ of the parabolic Hecke algebra $\Hmu$ is indexed by double cosets of $S_n$ by the Young subgroup $S_{\mu}$, see Section \ref{sec_para}. These double cosets are naturally in bijection with the following sets of objects:
\begin{itemize}
\item Diagrams generalising the usual diagrams for permutations. They connect two lines of $n$ dots. The $i$-th dot of each line has $\mu_i$ edges attached to it.
\item $n\times n$ matrices with non-negative integer entries such that the sum of the entries in the $i$-th row is equal to the sum of the entries in the $i$-th column and is equal to $\mu_i$.
\end{itemize}
We refer to \cite{CP23,PdA20} for more details.

\subsection{Cells in $\dSmu$ and RSK correspondences}\label{subsec-RSK}

\subsubsection{Two RSK correspondences}

We will use the following map from $\{1,\dots,n\}$ to $\{1,\dots,d\}$:
\begin{equation}\label{map-mu}
\left\{\begin{array}{l}1 \\ \vdots \\ \mu_1 \end{array}\right.\mapsto 1\,,\ \ \ \ \ \left\{\begin{array}{l} \mu_1+1\\ \vdots\\ \mu_1+\mu_2\end{array}\right.\mapsto 2\,,\ \ \ \ \dots\ \ \ \ \left\{\begin{array}{l} \mu_1+\dots+\mu_{d-1}+1\\ \vdots\\ n\end{array}\right. \mapsto d\,,
\end{equation}
so that the first $\mu_1$ integers are replaced by $1$, the next $\mu_2$ integers are replaced by $2$, and so on.

Recall that a semistandard Young tableau of shape $\lambda$ and of weight $\mu$ is a filling of the Young diagram of $\lambda$ with integers, such that the integer $i$ appears $\mu_i$ times, and the filling is weakly increasing along the rows and strictly increasing along the columns. We denote:
\[\SSTab(\lambda,\mu)=\{\text{semistandard Young tableaux of shape $\lambda$ and weight $\mu$}\}\ .\]
\begin{defi}
For a standard Young tableau $\bT\in \STab(\lambda)$, we denote by $\overline{\bT}$ the tableau obtained from $\bT$ by applying the map (\ref{map-mu}) to all the entries.
\end{defi}
The resulting tableau $\overline{\bT}$ is sometimes a semistandard Young tableau in $\SSTab(\lambda,\mu)$. In fact, it happens exactly when the standard Young tableau $\bT$ does not contain in its descent any element of $J$ (the subset in (\ref{def-Jmu}) defining the subgroup $S_{\mu}$). We recall that:
\begin{equation}\label{Kostka}
\SSTab(\lambda,\mu)\neq\emptyset\ \ \ \quad\Leftrightarrow\ \ \ \quad\lambda\geq\mu^{ord}\ ,
\end{equation}
where $\mu^{ord}$ is the partition obtained from $\mu$ by ordering the parts in decreasing order. This can be proved combinatorially (see \cite{Sta99}).

For what follows, recall from Section \ref{sec_prel} that the usual Robinson--Schensted correspondence for permutations in $S_n$ is denoted:
\begin{equation}\label{RSagain}
\begin{array}{rcl}
S_n & \longleftrightarrow & \displaystyle\bigsqcup_{\lambda\vdash n} \STab(\lambda)^2\\[1em]
w & \longleftrightarrow & \bigl(P(w),Q(w)\bigr)
\end{array}\ .
\end{equation}

\paragraph{The Robinson--Schensted--Knuth correspondence \cite{Knu70}.}
We are ready to describe the procedure giving a pair of elements in $\SSTab(\lambda,\mu)$, for some $\lambda$, starting from a double coset $\cD\in\dSmu$. Schematically, we apply the following procedure:
\[\cD\ \leftrightarrow\ \rmd=w\ \longleftrightarrow\ \bigl(P(w),Q(w)\bigr)\ \leftrightarrow\ \bigl(\overline{P(w)},\overline{Q(w)}\bigr)\ .\]
Since $w= \rmd$ is a minimal-length coset representative, it does not contain any element of $J$ in its left or its right descent. So both $P(w)$ and $Q(w)$ are such that their images by the map $\overline{\cdot}$ are semistandard Young tableaux. We denote the resulting bijection by:
\begin{equation}\label{RSK1}
\begin{array}{rcl}
\dSmu & \leftrightarrow & \displaystyle\bigsqcup_{\lambda\vdash n} \SSTab(\lambda,\mu)^2\\[1.2em]
\cD & \leftrightarrow & \bigl(P(\cD),Q(\cD)\bigr)
\end{array}
\end{equation}
As recalled in \eqref{Kostka}, only shapes $\lambda$ such that $\lambda\geq\mu^{ord}$ give non-empty sets of semistandard tableaux.
\begin{rema}\label{rema-RSK}
Take a  double coset $\cD\in\dSmu$ and take its minimal representative $\rmd$. Write it with the two-line notation for permutations and then apply the map (\ref{map-mu}) to all the entries. This results in a two-line array of integers, arranged in increasing lexicographic order, with $\mu_i$ entries equal to $i$ in each line. In \cite{Knu70}, the bijection (\ref{RSK1}) is described directly from these arrays of integers. It is easy to see that the two descriptions are equivalent.
\end{rema}

\paragraph{Another correspondence.} We will use another natural bijection alternative to (\ref{RSK1}). Schematically, it goes as follows:
\[\cD\ \leftrightarrow\ \rpd=w\ \longleftrightarrow\ \bigl(P(w),Q(w)\bigr)\ \leftrightarrow\ \bigl(\overline{P(w)^t},\overline{Q(w)^t}\bigr)\ .\]
Since $w= \rpd$ is a maximal-length coset representative, it contains every element of $J$ in its left and its right descent. So both $P(w)$ and $Q(w)$ also contain $J$ in their descent. This means that their transposed tableaux (exchanging lines and columns) are such that their images by $\overline{\cdot}$ are semistandard Young tableaux. We denote the resulting bijection by:
\begin{equation}\label{RSK2}
\begin{array}{rcl}
\dSmu & \leftrightarrow & \displaystyle\bigsqcup_{\lambda\vdash n} \SSTab(\lambda,\mu)^2\\[1.2em]
\cD & \leftrightarrow & \bigl(\tP(\cD),\tQ(\cD)\bigr)
\end{array}
\end{equation}
Again, see \eqref{Kostka}, only shapes $\lambda$ such that $\lambda\geq\mu^{ord}$ give non-empty sets of semistandard tableaux.
\begin{exam}\label{exam_RSK} We consider the parabolic subgroup $S_2\times S_1\times S_1$ inside $S_4$. There are $7$ double cosets. Here we give for each double coset the minimal representative and the corresponding pair of semistandard tableaux for the RSK correspondence (\ref{RSK1}). We do not repeat the second tableau when it coincides with the first:
\[\begin{array}{|c|c|c|c|c|c|c|c|}
\hline  & \cD_1 & \cD_2& \cD_3& \cD_4& \cD_5& \cD_6& \cD_7\\
\hline \rmd & e & s_2 & s_3 & s_2s_3 & s_3s_2 & s_2s_3s_2 & s_2s_1s_3s_2 \\
\hline \bigl(P(\cD),Q(\cD)\bigr) &  (\,{\scriptstyle{1123}}\,,\,\cdot\,) 
& 
(\!\!\begin{array}{l}
\scriptstyle{113} \\[-0.4em]
\scriptstyle{2}
\end{array}\hspace{-0.1cm},\,\cdot\,)
&
(\!\!\begin{array}{l}
\scriptstyle{112} \\[-0.4em]
\scriptstyle{3}
\end{array}\hspace{-0.1cm},\,\cdot\,)
&
(\!\!\begin{array}{l}
\scriptstyle{113} \\[-0.4em]
\scriptstyle{2}
\end{array}\hspace{-0.1cm},\hspace{-0.1cm}\begin{array}{l}
\scriptstyle{112} \\[-0.4em]
\scriptstyle{3}
\end{array}\!\!)
&
(\!\!\begin{array}{l}
\scriptstyle{112} \\[-0.4em]
\scriptstyle{3}
\end{array}\hspace{-0.1cm},\hspace{-0.1cm}\begin{array}{l}
\scriptstyle{113} \\[-0.4em]
\scriptstyle{2}
\end{array}\!\!)
&
(\!\!\begin{array}{l}
\scriptstyle{11} \\[-0.4em]
\scriptstyle{2}\\[-0.4em]
\scriptstyle{3}
\end{array}\hspace{-0.1cm},\,\cdot\,)
&
(\!\!\begin{array}{l}
\scriptstyle{11} \\[-0.4em]
\scriptstyle{23}
\end{array}\hspace{-0.1cm},\,\cdot\,)\\
\hline
\end{array}\]
Then, with the same convention and numbering of the cosets, we give the maximal representative and the corresponding pair of semistandard tableaux for the other correspondence (\ref{RSK2}):
\[\begin{array}{|c|c|c|c|c|c|c|c|}
\hline  & \cD_1 & \cD_2& \cD_3& \cD_4& \cD_5& \cD_6& \cD_7\\
\hline \rpd & s_1 & s_1s_2s_1 & s_1s_3 & s_1s_2s_3s_1 & s_1s_3s_2s_1 & s_1s_2s_3s_2s_1 & s_1s_2s_1s_3s_2s_1 \\
\hline \bigl(P(\cD),Q(\cD)\bigr) & (\!\!\begin{array}{l}
\scriptstyle{11} \\[-0.4em]
\scriptstyle{2}\\[-0.4em]
\scriptstyle{3}
\end{array}\hspace{-0.1cm},\,\cdot\,)
& 
(\!\!\begin{array}{l}
\scriptstyle{112} \\[-0.4em]
\scriptstyle{3}
\end{array}\hspace{-0.1cm},\,\cdot\,)
&
(\!\!\begin{array}{l}
\scriptstyle{11} \\[-0.4em]
\scriptstyle{23}
\end{array}\hspace{-0.1cm},\,\cdot\,)

&
(\!\!\begin{array}{l}
\scriptstyle{112} \\[-0.4em]
\scriptstyle{3}
\end{array}\hspace{-0.1cm},\hspace{-0.1cm}\begin{array}{l}
\scriptstyle{113} \\[-0.4em]
\scriptstyle{2}
\end{array}\!\!)
&
(\!\!\begin{array}{l}
\scriptstyle{113} \\[-0.4em]
\scriptstyle{2}
\end{array}\hspace{-0.1cm},\hspace{-0.1cm}\begin{array}{l}
\scriptstyle{112} \\[-0.4em]
\scriptstyle{3}
\end{array}\!\!)
&
(\!\!\begin{array}{l}
\scriptstyle{113} \\[-0.4em]
\scriptstyle{2}
\end{array}\hspace{-0.1cm},\,\cdot\,)
&
 (\,{\scriptstyle{1123}}\,,\,\cdot\,) \\
\hline
\end{array}\]
\end{exam}

\subsubsection{Cells in $\dSmu$ and RSK correspondences}

\paragraph{The first basis of $\Hmu$.} In this paragraph, we use the first Kazhdan--Lusztig basis 
$$\{C_{\rpd}\}_{\cD\in\dSmu}$$
of $\Hmu$ studied in Section \ref{subsec_KLbasispar1} to define orders and cells on $\dSmu$. From the results proved in Section \ref{sec_KLpar}, in particular Proposition \ref{prop-cell1}, we can immediately describe the cells in $\dSmu$ associated to this basis $\{C_{\rpd}\}_{\cD\in\dSmu}$, since they are simply the intersection of the cells of $S_n$ with the subset of maximal-length representatives. 

Almost by construction, the relevant RSK correspondence here is the second one (\ref{RSK2}), which uses the maximal-length representatives. The notation is that to a double coset $\cD$ is associated a pair $\bigl(\tP(\cD),\tQ(\cD)\bigr)$ of semistandard Young tableaux of the same shape. We will denote $\tsh(\cD)$ this common shape.
\begin{coro}\label{prop-cells-Hmu1}We have:
\begin{itemize}
\item $\cD$ and $\cD'$ are in the same $\left\{\begin{array}{l} \text{left}\\[0.3em] \text{right}\\[0.3em] \text{2-sided} \end{array}\right.$ cell of $\dSmu$ if and only if $\left\{\begin{array}{l} \tQ(\cD)=\tQ(\cD')\\[0.3em] \tP(\cD)=\tP(\cD')\\[0.3em] \tsh(\cD)=\tsh(\cD') \end{array}\right.$;
\item $\cD\preceq_{\mathcal{LR}}\cD'$ if and only if $\tsh(\cD)\geq \tsh(\cD')$.
\end{itemize}
\end{coro}
Note that the two-sided cells in $\dSmu$ are indexed by partitions $\lambda$ such that $\lambda\geq\mu^{ord}$, see the sentence after \eqref{RSK2}. The two-sided order $\preceq_{\mathcal{LR}}$ is given in the second item by the reverse dominance order since we needed to transpose the tableaux in (\ref{RSK2}) in order to define $\tP(\cD)$ and $\tQ(\cD)$, and the dominance order is reversed by transposition.

\paragraph{The second basis of $\Hmu$.} In this paragraph, we use the second basis 
$$\{e_JC^{\da}_{\rmd}e_J\}_{\cD\in\dSmu}$$
of $\Hmu$ studied in Section \ref{subsec_KLbasispar2} to define orders and cells on $\dSmu$. As already recalled in Section \ref{sec_prel}, the left cell modules of $H(S_n)$ are all irreducible over $\mathbb{C}(q)$. So we can use Proposition \ref{prop-cell2b} from Section \ref{sec_KLpar} (and everything is similar for right cell modules). Therefore we can immediately describe the cells associated to this basis since they are simply the intersection of the cells of $S_n$ with the subset of minimal-length representatives.

Almost by construction, the relevant RSK correspondence here is the first one (\ref{RSK1}), which uses the minimal-length representatives. The notation is that to a double coset $\cD$ is associated a pair $\bigl(P(\cD),Q(\cD)\bigr)$  of semistandard Young tableaux of the same shape. We will denote $sh(\cD)$ this common shape.
\begin{prop}\label{prop-cells-Hmu2}We have
\begin{itemize}
\item $\cD$ and $\cD'$ are in the same $\left\{\begin{array}{l} \text{left}\\[0.3em] \text{right}\\[0.3em] \text{2-sided} \end{array}\right.$ cell of $\dSmu$ if and only if $\left\{\begin{array}{l} Q(\cD)=Q(\cD')\\[0.3em] P(\cD)=P(\cD')\\[0.3em] sh(\cD)=sh(\cD') \end{array}\right.$;
\item If $\cD\preceq_{\mathcal{LR}}\cD'$ then $sh(\cD)\leq sh(\cD')$;
\end{itemize}
\end{prop}
Note again that the two-sided cells are indexed by partitions $\lambda$ such that $\lambda\geq\mu^{ord}$, see the sentence after \eqref{RSK1}. This time, since we do not have a general result for the cell orders (see Proposition \ref{prop-cell2a}), we only know that the $\mathcal{LR}$-order is weaker than the dominance order on partitions.

\begin{exam}
The partition of the set $\dSmu$ into cells is different depending on which of the two bases we consider. For example, for $n=4$ and $\mu=(2,1,1)$, there are 7 cosets denoted $\cD_1,\dots,\cD_7$. There are 4 two-sided cells indexed by partitions $\lambda\geq \mu$. Builiding on the calculation of the RSK correspondences in Example \ref{exam_RSK}, we show here the partition into cells:
\[\begin{array}{llcc}
&& \text{First basis} & \text{Second basis}\\[0.4em]
\begin{tikzpicture}[scale=0.25]
\diag{0}{0}{4};
\end{tikzpicture}
 & : & \{\cD_7\} & \{\cD_1\} \\[0.4em]
\begin{tikzpicture}[scale=0.25]
\diagg{0}{0}{3}{1};
\end{tikzpicture}
 & : & \{\{\cD_2,\cD_5\},\{\cD_4,\cD_6\}\} & \{\{\cD_2,\cD_5\},\{\cD_3,\cD_4\}\} \\[0.4em]
\begin{tikzpicture}[scale=0.25]
\diagg{0}{0}{2}{2};
\end{tikzpicture}
 & : & \{\cD_3\} & \{\cD_7\} \\
\begin{tikzpicture}[scale=0.25]
\diaggg{0}{0}{2}{1}{1};
\end{tikzpicture}
 & : & \{\cD_1\} & \{\cD_6\}
\end{array}\]
On each line, we show the corresponding two-sided cells, with its decomposition into left cells.
\end{exam}

\subsection{Cells and representations of $\Hmu$}\label{subsec-rep}

Using the Kazhdan--Lusztig bases of $H(S_n)$, we discuss when the cell representations of $H(S_n)$ are killed or not by the idempotent $e_{\mu}$. Restricting to the semisimple case, we thus recover, with a different proof, the classification of irreducible representations of $\Hmu$ obtained in \cite{CP23}.

\subsubsection{Classification}

Recall that we have denoted $\{V_{\Gamma}\}$ the cell representations of $H(S_n)$ using the basis $\{C_w\}$, and $\{V^{\da}_{\Gamma}\}$ the cell representations of $H(S_n)$ using the basis $\{C^{\da}_w\}$.
\begin{theo}\label{theo-cellrep}
Let $\lambda\vdash n$ and let $\Gamma$ be a left cell of $S_n$ containing elements $w$ with $sh(w)=\lambda$. We have
$$e_{\mu}(V_{\Gamma})\neq 0\ \ \ \ \ \ \Leftrightarrow\ \ \ \ \ \ \ \ \lambda^{t}\geq\mu^{ord}\ ,$$
$$e_{\mu}(V^{\da}_{\Gamma})\neq 0\ \ \ \ \ \ \Leftrightarrow\ \ \ \ \ \ \ \ \lambda\geq\mu^{ord}\ ,$$
where $\mu^{ord}$ is the partition obtained from $\mu$ by ordering the parts in decreasing order.
\end{theo}
\begin{proof} We consider first the representation $V_{\Gamma}$. We prove that $e_{\mu}(V_{\Gamma})=0$ if  $\lambda^{t}\ngeq\mu^{ord}$. Indeed, the idempotent $e_{\mu}$ is proportional to the element $C_{w_{\mu}}$, where $w_{\mu}$ is the longest element of the parabolic subgroup $S_{\mu}$. It is easy to see that under the RS correspondence, $w_{\mu}$ is of shape $(\mu^{ord})^t$. When multiplying $C_{w_{\mu}}$ with any element and expanding in the $\{C_w\}$ basis, we find only elements $C_x$ with $sh(x)\leq(\mu^{ord})^t$. If $\lambda^{t}\ngeq\mu^{ord}$ then $\lambda\nleq(\mu^{ord})^t$, and we can never find elements $C_w$ with $sh(w)=\lambda$.

For the rest of the proof, we recall from (\ref{Kostka}) that
\begin{equation}\label{Kostkaagain}
\lambda\geq\mu^{ord}\ \ \ \quad\Leftrightarrow\ \ \ \quad\SSTab(\lambda,\mu)\neq\emptyset\ .
\end{equation}
Note then that the condition $\SSTab(\lambda,\mu)\neq\emptyset$ is equivalent to the existence of a standard tableau of shape $\lambda$ with descent disjoint from $J$. Indeed the map (\ref{map-mu}) applied to all entries provides a bijection between such standard tableaux and $\SSTab(\lambda,\mu)$. We must only check that from $\mathbf{T}\in\SSTab(\lambda,\mu)$, we can choose a preimage of $\mathbf{T}$ for the map (\ref{map-mu}) with descent disjoint from $J$. Since each set of boxes in $\mathbf{T}$ with the same entry, say $a\in\{1,\dots,d\}$, contains at most one box in each column, we can choose a preimage such that the numbers $\mu_1+\dots+\mu_{a-1}+1,\dots,\mu_1+\dots+\mu_a$ are placed strictly from left to right.

From what we just said, if $\lambda^{t}\geq\mu^{ord}$ then we can take a standard tableau of shape $\lambda^t$ with descent disjoint from $J$. Transposing, we obtain a standard tableau $\bT$ of shape $\lambda$ with descent containing $J$. Since $\Gamma$ is a left cell of $S_n$ containing elements of shape $\lambda$, there is an element $w$ in $\Gamma$ with $\bT$ as its left tableau under the RS correspondence. Therefore, $J$ is included in the left descent of $w$, and we thus have $e_{\mu}C_w=C_w$. This shows that $e_{\mu}(V_{\Gamma})\neq 0$.

The reasoning is quite similar for the representation $V^{\da}_{\Gamma}$. From the discussion above, the condition $\lambda\geq\mu^{ord}$ is equivalent to the existence of a standard tableau of shape $\lambda$ with descent disjoint from $J$, which is in turn equivalent to the existence of an element in $\Gamma$ with left descent disjoint from $J$. These elements in $\Gamma$ with left descent disjoint from $J$ index a basis of $e_J(V^{\da}_{\Gamma})$, as was shown during the proof of Proposition \ref{prop-rep2}, see (\ref{basis-rep2}). Thus if $\lambda\ngeq\mu^{ord}$, there is no such element in $\Gamma$ and we have $e_J(V^{\da}_{\Gamma})=0$ while otherwise there is such an element and we have $e_J(V^{\da}_{\Gamma})\neq0$.
\end{proof}

\paragraph{The semisimple situation.} In this paragraph, we consider the semisimple situation, namely the field $\mathbb{C}(q)$ or a non-zero complex number $q$ such that $q^2$ is not a root of unity whose order is between $2$ and $n$.

From the general results for idempotent subalgebras, (see \emph{e.g.} \cite[\S 6.2]{Gre80}) a complete set of pairwise distinct irreducible representations of $\Hmu$ are all the non-zero $e_{\mu}(V_{\lambda})$, where the $V_{\lambda}$'s are the irreducible representations of $H(S_n)$. The identification between the cell representations and the $V_{\lambda}$'s was recalled in Subsection \ref{subsec_preltypeA}. Using the previous theorem, we conclude immediately with the following corollary, which was proven in \cite{CP23} with other methods.
\begin{coro}\label{coro-classif}
A complete set of non-zero pairwise distinct irreducible representations of $\Hmu$ is:
\[\{e_{\mu}(V_{\lambda})\ \text{such that}\ \lambda\geq\mu^{ord}\}\ ,\]
where $\mu^{ord}$ is the partition obtained from $\mu$ by ordering the parts in decreasing order.
\end{coro}
\begin{rema}
In the particular situation $\mu=(k,\dots,k)$ (a positive integer $k$ repeated $d$ times), then the condition $\lambda\geq\mu^{ord}$ is easily seen to be equivalent to the condition that the number of non-zero parts of $\lambda$ is less or equal to $d$ (which is only a necessary condition for general $\mu$). We refer to \cite{CP23} for more details.
\end{rema}

The semisimple representation theory of the algebra $\Hmu$ is concisely summarised by its Bratteli diagram. We refer to \cite{CP23}. For example, if $n=6$ and $\mu=(2,2,2)$, the Bratteli diagram is: 
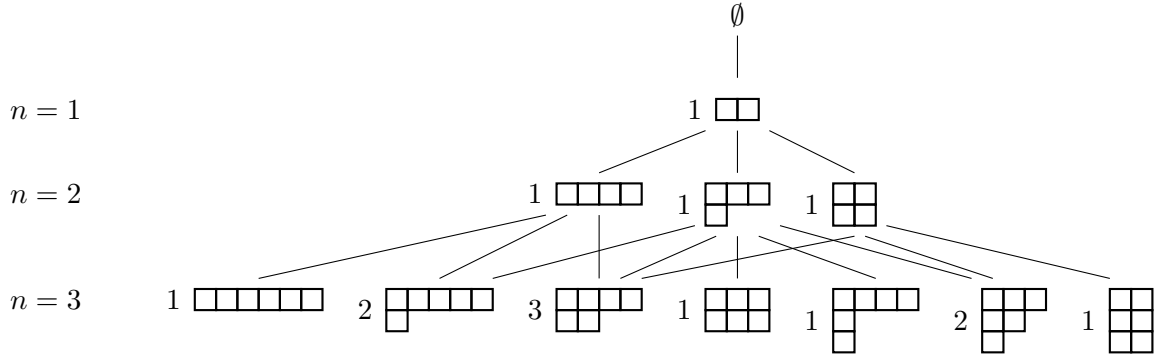
\begin{center}
 \begin{tikzpicture}[scale=0.28]
\node at (0.5,4) {$\emptyset$};
\draw ( 0.5,3) -- (0.5, 1);
\diag{-0.5}{0}{2};\node at (-1.5,-0.5) {$1$};
\draw (-1,-1.5) -- (-6,-3.5);\draw (0.5,-1.5)--(0.5,-3.5);\draw (2,-1.5) -- (6,-3.5);
\diag{-8}{-4}{4};\node at (-9,-4.5) {$1$};\diagg{-1}{-4}{3}{1};\node at (-2,-5) {$1$};\diagg{5}{-4}{2}{2};\node at (4,-5) {$1$};

\draw (-8.5,-5.5) -- (-22,-8.5);\draw (-7.5,-5.5) -- (-13.5,-8.5);\draw (-6,-5.5) -- (-6,-8.5);    \draw (-1.5,-6) -- (-11,-8.5); \draw (-0.5,-6.5) -- (-5,-8.5);\draw (0.5,-6.5) -- (0.5,-8.5);
\draw (1.5,-6.5) -- (7,-8.5);\draw (2.5,-6) -- (11.5,-8.5);\draw (6,-6.5) -- (-4,-8.5); \draw (6.5,-6.5) -- (12.5,-8.5);\draw (7.5,-6) -- (18,-8.5);

\diag{-25}{-9}{6};\node at (-26,-9.5) {$1$};\diagg{-16}{-9}{5}{1};\node at (-17,-10) {$2$};\diagg{-8}{-9}{4}{2};\node at (-9,-10) {$3$};
\diagg{-1}{-9}{3}{3};\node at (-2,-10) {$1$};\diaggg{5}{-9}{4}{1}{1};\node at (4,-10.5) {$1$}; \diaggg{12}{-9}{3}{2}{1};\node at (11,-10.5) {$2$};\diaggg{18}{-9}{2}{2}{2};\node at (17,-10.5) {$1$};


\node at (-32,-0.5) {$n=1$};\node at (-32,-4.5) {$n=2$};\node at (-32,-9.5) {$n=3$};
\end{tikzpicture}
\end{center}
The rule is that we add two boxes at each step (since $\mu=(2,2,2)$) and two partitions $\lambda'\vdash n$ and $\lambda\vdash n+2$ are connected if $\lambda'\subset \lambda$
and moreover $\lambda/\lambda'$ does not contain two boxes in the same column. For example $\lambda=(2,2)$ is not connected to $\lambda'=(3,3)$.

The dimension of $e_{\mu}(V_{\lambda})$ is a Kostka number, that is, the following number of semistandard Young tableaux:
\[\dim e_{\mu}(V_{\lambda})=|\SSTab(\lambda,\mu)|\ .\]
It follows at once from the description of the cells of $\Hmu$ (see also \cite{CP23}).

\subsubsection{Examples}

\paragraph{Example: $\lambda=(n)$.} The representation is obviously of dimension 1. It is the restriction, or projection, to $\Hmu$ of the one-dimensional representation $T_w\mapsto q^{\ell(w)}$ of the Hecke algebra $H(S_n)$. 

In terms of the first basis of $\Hmu$, this partition is at the bottom of the cell order (reverse dominance ordering). The basis element corresponding to this cell is
\[C_{w_n}\,,\ \ \ \ \text{where $w_n$ is the longest element of $S_n$.}\]
It is indeed a simple matter to see that elements of $\Hmu$ acts as follows:
\[e_{\mu}T_we_{\mu}C_{w_n}=q^{\ell(w)}C_{w_n}\ .\]

In terms of the second basis of $\Hmu$, the partition $(n)$ is at the top of the cell order. The corresponding cell consists only of the trivial double coset and the corresponding basis element is
\[e_{\mu}C^{\da}_{e}e_{\mu}=e_{\mu}\ .\]
The action of $\Hmu$ is by left multiplication modulo all the basis elements different from $e_{\mu}$, and thus, in this picture, the action of the basis elements of $\Hmu$ is given as follows
\[e_{\mu}C_{w}^{\da}e_{\mu}\mapsto\left\{\begin{array}{ll}
1 & \text{if $w=e$,}\\[0.4em]
0 & \text{otherwise\,.}\end{array}\right.\]

\begin{rema}
Note that the other one-dimensional representation $T_w\mapsto (-q^{-1})^\ell(w)$ of $H(S_n)$ does not give a representation of $\Hmu$ since the projector $e_{\mu}$ is $0$ in this representation. However, there are other one-dimensional representations of $\Hmu$. One of them for $\lambda=\mu^{ord}$ is discussed below, and the other one (probably the most interesting because not at one extremity of the cell ordering) is for $\lambda$ a hook partition of maximal-length and will play a prominent role in the next section.
\end{rema}

\paragraph{Example: $\lambda=\mu^{ord}$.} This partition also corresponds to a one-dimensional representation of $\Hmu$ (even if the corresponding representation of $H(S_n)$ is not one-dimensional). To see this, recall that the dimension is the cardinality of $\SSTab(\lambda,\mu)$ and that this number does not depend on the ordering of the composition $\mu$ (see \cite[Theorem 7.10.2]{Sta99}). So for simplicity we assume that $\mu=\mu^{ord}$, that is $\mu=(\mu_1,\dots,\mu_d)$ with $\mu_1\geq \dots\geq\mu_d$. In this case, one easily sees that there is a single semistandard tableau in $\SSTab(\lambda,\mu)$ for $\lambda=\mu$. It is obviously obtained by filling the first line with $1$'s, the second line with $2$'s and so on.

For the first basis of $\Hmu$, the unique basis element corresponding to this cell is
\[C_{w_{\mu}}=e_{\mu}\,,\]
where $w_{\mu}$ is the longest element of the subgroup $S_{\mu}$ inside $S_n$. In this case, the partition $\mu$ is at the top of the cell order and the action of $\Hmu$ is by left multiplication modulo all the basis elements different from $e_{\mu}$. Thus, in this picture, the action of the basis elements of $\Hmu$ is given as follows
\[C_{w}\mapsto\left\{\begin{array}{ll}
1 & \text{if $w=w_{\mu}$,}\\[0.4em]
0 & \text{otherwise\,.}\end{array}\right.\]

For the second basis of $\Hmu$, the partition $\lambda=\mu$ is at the bottom of the cell order. The corrrsponding basis element is
\[e_{\mu}C^{\da}_{\tilde{w}_{\mu}}e_{\mu}\ ,\]
for some permutation $\tilde{w}_{\mu}$, which is a minimal-length representative corresponding through the RSK correspondence to the unique semistandard Young tableau in $\SSTab(\mu,\mu)$. It is straightforward to check that the permutation $\tilde{w}_{\mu}$ is given, visually, as follows:
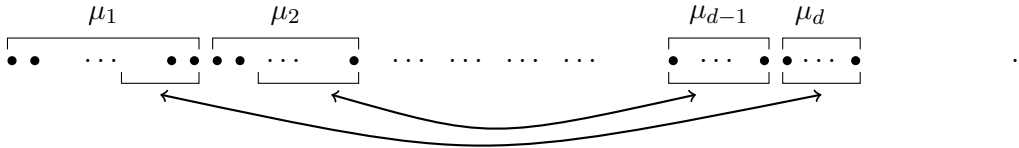
\begin{center}
 \begin{tikzpicture}[scale=0.3]
\draw (0.8,0.5) -- (0.8,1) -- (9.2,1) -- (9.2,0.5);\node at (5,2) {$\mu_1$};
\fill (1,0) circle (0.2);
\fill (2,0) circle (0.2);
\node at (5,0) {$\dots$};
\draw (5.8,-0.5) -- (5.8,-1) -- (9.2,-1) -- (9.2,-0.5);

\fill (8,0)  circle (0.2);
\fill (9,0) circle (0.2);

\draw (9.8,0.5) -- (9.8,1) -- (16.2,1) -- (16.2,0.5);\node at (13,2) {$\mu_2$};
\draw (11.8,-0.5) -- (11.8,-1) -- (16.2,-1) -- (16.2,-0.5);
\fill (10,0) circle (0.2);
\fill (11,0) circle (0.2);
\node at (13,0) {$\dots$};
\fill (16,0) circle (0.2);

\node at (18.5,0) {$\dots$};

\node at (21,0) {$\dots$};

\node at (23.5,0) {$\dots$};

\node at (26,0) {$\dots$};

\draw (29.8,0.5) -- (29.8,1) -- (34.2,1) -- (34.2,0.5);\node at (32,2) {$\mu_{d-1}$};
\draw (29.8,-0.5) -- (29.8,-1) -- (34.2,-1) -- (34.2,-0.5);
\fill (30,0) circle (0.2);
\node at (32,0) {$\dots$};
\fill (34,0) circle (0.2);

\draw (34.8,0.5) -- (34.8,1) -- (38.2,1) -- (38.2,0.5);\node at (36,2) {$\mu_{d}$};
\draw (34.8,-0.5) -- (34.8,-1) -- (38.2,-1) -- (38.2,-0.5);
\fill (35,0) circle (0.2);
\node at (36.5,0) {$\dots$};
\fill (38,0) circle (0.2);

\draw[<->,thick] (7.5,-1.5) ..controls +(14,-3) .. (36.5,-1.5);
\draw[<->,thick] (15,-1.5) ..controls +(7,-2) .. (31,-1.5);

\node at (45,0) {.};
\end{tikzpicture}
\end{center}
It is the involution sending the last $\mu_d$ integers (preserving their order) to the last $\mu_d$ integers in the subset $\{1,\dots,\mu_1\}$, and then repeating the procedure to the remaining subsets of sizes $\mu_2,\dots,\mu_{d-1}$ as long as there remains at least two subsets. For example, if $\mu=(3,2,2)$ then $\tilde{w}_{\mu}=1674523$. 

Since $\mu$ is at the bottom of the cell order, we know that left multiplication of the basis element $e_{\mu}C^{\da}_{\tilde{w}_{\mu}}e_{\mu}$ by any $h\in\Hmu$ is proportional to the basis element:
\[he_{\mu}C^{\da}_{\tilde{w}_{\mu}}e_{\mu}=\alpha(h)e_{\mu}C^{\da}_{\tilde{w}_{\mu}}e_{\mu}\ .\] 
The coefficient $\alpha(h)$ is the value of the element $h$ in this one-dimensional representation.

\subsection{Cellular bases of $\Hmu$}\label{subsec-cellular}

\subsubsection{Cellularity of the Hecke algebra $H(S_n)$}

It is well-known \cite{Gec07,KL79} that combined with the RS correspondence, the Kazhdan--Lusztig bases $\{C_w\}_{w\in S_n}$ and $\{C^{\da}_w\}_{w\in S_n}$ become cellular bases of $H(S_n)$, thus providing the Hecke algebra $H(S_n)$ with a structure of a cellular algebra in the sense of \cite{GL96}.

We reindex the two Kazhdan--Lusztig bases, using the RS correspondence $w\leftrightarrow \bigl(P(w),Q(w)\bigr)$, defining 
\[C_{P(w),Q(w)}=C_w\qquad\ \text{and}\ \qquad C^{\da}_{P(w),Q(w)}=C^{\da}_w\ .\] 
We consider the antiautomorphism $\iota$ of order 2 of $H(S_n)$, which sends any generator $T_i$ to itself, which means that it is given on the standard basis by:
\[\iota\ :\ T_w\mapsto T_{w^{-1}}\ .\]
Then the two sets of elements:
\[\{C_{\textbf{s},\textbf{t}}\}_{\textbf{s},\textbf{t}}\,,\ \ \ \{C^{\da}_{\textbf{s},\textbf{t}}\}_{\textbf{s},\textbf{t}}\ \ \ \ \ \text{where $(\textbf{s},\textbf{t})$ runs over $\displaystyle\bigsqcup_{\lambda\vdash n} \STab(\lambda)^2$\ ,}\]
form two cellular bases of $H(S_n)$ with respect to the involution $\iota$ and to the poset of partitions of $n$ with the dominance order. To be precise, this means that:
\begin{itemize}
\item $\iota(C_{\textbf{s},\textbf{t}})=C_{\textbf{t},\textbf{s}}$ for any $\textbf{s},\textbf{t}$;
\item for all $h\in H(S_n)$ and $\textbf{s},\textbf{t}\in\STab(\lambda)$, we have 
$$hC_{\textbf{s},\textbf{t}}=\sum_{\textbf{s}'\in\STab(\lambda)} r_h(\textbf{s},\textbf{s}')C_{\textbf{s}',\textbf{t}}\ \text{mod}\,I_{<\lambda}\ \quad\text{and}\ \quad hC^{\da}_{\textbf{s},\textbf{t}}=\sum_{\textbf{s}'\in\STab(\lambda)} \tilde{r}_h(\textbf{s},\textbf{s}')C^{\da}_{\textbf{s}',\textbf{t}}\ \text{mod}\,I^{\da}_{<\lambda}\ .$$
\end{itemize}
where $r_h(\textbf{s},\textbf{s}')$ and $\tilde{r}_h(\textbf{s},\textbf{s}')$ are coefficients independent of $\textbf{t}$. Above, the ideal $I_{<\lambda}$ is the span of all elements $C_{\textbf{s},\textbf{t}}$ for standard tableaux $\textbf{s},\textbf{t}$ of shape strictly smaller than $\lambda$ in the dominance order. The ideal $I^\da_{<\lambda}$ is defined similarly using elements $C^{\da}_{\textbf{s},\textbf{t}}$.

\subsubsection{Two cellular bases of $\Hmu$}

It follows immediately from its definition (\ref{P-typeA}) that the projector $e_{\mu}$ is stable by the involution $\iota$. From a general property of cellular algebras (see \cite[Propostion 4.3]{KX98}), the stability of $e_{\mu}$ by $\iota$ ensures that the algebra $\Hmu=e_{\mu}H_ne_{\mu}$ is also cellular with respect to the same involution. Our goal to conclude this section is to make explicit the two cellular bases of $\Hmu$ that we obtain using the bases previously considered.

\paragraph{The poset.} We recall that a basis of $\Hmu$ is indexed by the double cosets $\dSmu$, and that these double cosets are in bijection with a certain set of pairs of semistandard tableaux:
\[\dSmu \leftrightarrow \displaystyle\bigsqcup_{\lambda\vdash n} \SSTab(\lambda,\mu)^2\ ,\]
via either one of the two RSK correspondences from Section \ref{subsec-RSK}. Recall also that the set $\SSTab(\lambda,\mu)$ is not empty if and only if $\lambda\geq\mu^{ord}$.

Thus the poset involved in the cellular datum for $\Hmu$ is going to be the set of partitions $\lambda$ of $n$ satisfying $\lambda\geq\mu^{ord}$, endowed with either the dominance ordering on partitions or the reverse dominance ordering.

\paragraph{Cellular bases.} Let $(\textbf{S},\textbf{T})\in \SSTab(\lambda,\mu)^2$ for some $\lambda\geq\mu^{ord}$. There is a unique standard tableau $\textbf{s}$ in $\STab(\lambda)$ which does not contain $J$ in its descent and such that $\overline{\textbf{s}}=\textbf{S}$, where $\overline{\ \cdot\ }$ is the map (\ref{map-mu}). Similarly, there is a unique standard tableau $\textbf{t}$ in $\STab(\lambda)$ which does not contain $J$ in its descent and such that $\overline{\textbf{t}}=\textbf{T}$.
\begin{exam}
If $\textbf{S}=\begin{array}{cccc}1 & 1 & 2 & 3 \\ 2 & 3 \end{array}$ then $\textbf{s}=\begin{array}{ccccc}1 & 2 & 3 & 5 \\ 4 & 6 \end{array}$.
\end{exam}

Unraveling the definition of the RSK correspondence (\ref{RSK2}), we see that if the pair $(\textbf{S},\textbf{T})$ corresponds to the coset $\cD$, then the pair $(\textbf{s}^t,\textbf{t}^t)$ corresponds to $\rpd$ through the usual RS correspondence. Thus we set:
\begin{equation}\label{cellular1}C_{\textbf{S},\textbf{T}}:=C_{\rpd}=C_{\textbf{s}^t,\textbf{t}^t}\ .\end{equation}

For the second basis, unraveling the definition of the RSK correspondence (\ref{RSK1}), we see that if the pair $(\textbf{S},\textbf{T})$ corresponds to the coset $\cD$, then the pair $(\textbf{s},\textbf{t})$ corresponds to $\rmd$ through the usual RS correspondence. Thus we set:
\begin{equation}\label{cellular2}C^{\da}_{\textbf{S},\textbf{T}}:=e_{\mu}C^{\da}_{\rmd}e_{\mu}=e_{\mu}C^{\da}_{\textbf{s},\textbf{t}}e_{\mu}\ .
\end{equation}

With these notations, we have that the two following sets form two cellular bases of $\Hmu$:
\[\{C_{\textbf{S},\textbf{T}}\}\,,\ \ \ \{C^{\da}_{\textbf{S},\textbf{T}}\}\ \ \ \ \ \text{where $(\textbf{S},\textbf{T})$ runs over $\displaystyle\bigsqcup_{\substack{\lambda\vdash n\\\ \ \lambda\geq\mu^{ord}}}\!\!\SSTab(\lambda,\mu)^2$\ .}\]
Indeed we already know that these are bases of $\Hmu$. Moreover,  Formulas (\ref{cellular1}) and (\ref{cellular2}), only in terms of tableaux, give immediately that $\iota(C_{\textbf{S},\textbf{T}})=C_{\textbf{T},\textbf{S}}$ and $\iota(C^{\da}_{\textbf{S},\textbf{T}})=C^{\da}_{\textbf{T},\textbf{S}}$. 

Regarding the property with respect to the left multiplication, let us consider the first basis. For $h\in H(S_n)$, we have explicitly
\[\begin{array}{ll}e_{\mu}he_{\mu}C_{\textbf{S},\textbf{T}}=e_{\mu}he_{\mu}C_{\textbf{s}^t,\textbf{t}^t}& =\displaystyle\sum_{\textbf{u},\textbf{v}\in\STab(\lambda)}r_{e_{\mu}}(\textbf{u}^t,\textbf{v}^t)r_h(\textbf{s}^t,\textbf{u}^t)C_{\textbf{v}^t,\textbf{t}^t}\ \text{mod}\ I_{<\lambda^t}\ \\[1em]
& =\displaystyle \sum_{\textbf{V}\in\SSTab(\lambda,\mu)}\bigl(\sum_{\textbf{u}\in\STab(\lambda)}r_{e_{\mu}}(\textbf{u}^t,\textbf{v}^t)r_h(\textbf{s}^t,\textbf{u}^t)\bigr)C_{\textbf{V},\textbf{T}}\ \text{mod}\ I_{<\lambda^t}\cap \Hmu\ .
\end{array}\]
The first multiplication by $e_{\mu}$ does nothing since $C_{\textbf{s}^t,\textbf{t}^t}$ is already in $\Hmu$. Since this calculation ends up in $\Hmu$, the sum over $\textbf{v}$ can be taken over the standard tableaux in $\STab(\lambda)$ which does not contain $J$ in their descent. To each such tableau $\textbf{v}$ corresponds a unique $\textbf{V}\in\SSTab(\lambda,\mu)$. In the resulting expression, the coefficient in front of $C_{\textbf{V},\textbf{T}}$ does not depend on $\textbf{T}$.

Finally, note that, due to the transposition of tableaux appearing in (\ref{cellular1}), the ideal $I_{<\lambda^t}\cap \Hmu$ is spanned by elements $C_{\textbf{U},\textbf{U'}}$ associated to shapes $\nu$ such that $\nu^t<\lambda^t$. This is equivalent to $\nu>\lambda$ and therefore the order on partitions that we have to use is the reverse dominance ordering.

For the second basis, still with $h\in H(S_n)$, we have explicitly:
\[\begin{array}{ll}e_{\mu}he_{\mu}C^\da_{\textbf{S},\textbf{T}}=e_{\mu}he_{\mu}C^\da_{\textbf{s},\textbf{t}}e_{\mu}& =\displaystyle\sum_{\textbf{u},\textbf{v}\in\STab(\lambda)}\tilde{r}_{h}(\textbf{u},\textbf{v})\tilde{r}_{e_{\mu}}(\textbf{s},\textbf{u})e_{\mu}C^{\da}_{\textbf{v},\textbf{t}}e_{\mu}\ \text{mod}\ I_{<\lambda}\ \\[1em]
& =\displaystyle \sum_{\textbf{V}\in\SSTab(\lambda,\mu)}\bigl(\sum_{\textbf{u}\in\STab(\lambda)}\tilde{r}_{h}(\textbf{u},\textbf{v})\tilde{r}_{e_{\mu}}(\textbf{s},\textbf{u})\bigr)C^{\da}_{\textbf{V},\textbf{T}}\ \text{mod}\ I_{<\lambda}\cap \Hmu\ .
\end{array}\]
The sum over $\textbf{v}$ can be taken over the standard tableaux in $\STab(\lambda)$ which does not contain $J$ in their descent since otherwise $e_{\mu}C^{\da}_{\textbf{v},\textbf{t}}$ would be $0$. Again to each such tableau $\textbf{v}$ corresponds a unique $\textbf{V}\in\SSTab(\lambda,\mu)$. In the resulting expression, the coefficient in front of $C^{\da}_{\textbf{V},\textbf{T}}$ does not depend on $\textbf{T}$. In this case, the order on partitions that we have to use is the usual dominance ordering.

\section{Application to Schur--Weyl duality}\label{sec-SW}

Unless otherwise specified, we work in this section in the semisimple situation, namely over the field $\mathbb{C}(q)$ or with a non-zero complex number $q$ such that $q^2$ is not a root of unity whose order is between $2$ and $n$ ($q^2=1$ is allowed).

In this section $N$ is an integer such that $N\geq 1$ and we will assume that $d$ (the length of the composition $\mu=(\mu_1,\dots,\mu_d)$) satisfies 
\[d>N\ ,\]
because otherwise the ideal $I^{\mu}_{N}$ studied in this section is $\{0\}$.

\subsection{The Schur--Weyl duality for $\Hmu$}

From \cite{CP23}, the algebra $\Hmu$ appears in the Schur--Weyl duality through a representation:
\begin{equation}\label{rep-SW}
\pi_N\ :\ \Hmu\to\text{End}\bigl(S^{\mu_1}_qV\otimes\dots\otimes S^{\mu_d}_qV\bigr)\,,
\end{equation}
where $V$ is of dimension $N$ and $S^k_qV$ is the $k$-th $q$-symmetrized power of $V$, which is an irreducible representation of the quantum group $U_q(gl_N)$. The map $\pi_n$ is surjective onto the $U_q(gl_N)$-centraliser, but it is not injective as soon as $d>N$ and it remains to understand its kernel. We set:
\[I^{\mu}_{N}=\text{Ker}\pi_N\,,\]
so that the quotient of $\Hmu$ by its ideal $I^{\mu}_{N}$ is isomorphic to the centraliser. We are interested in finding a linear basis of the ideal $I^{\mu}_{N}$ and a set (as simple as possible) of generators of it.

We will need the description of the ideal $I^{\mu}_{N}$ in terms of the representations of $\Hmu$ obtained in \cite{CP23}. From the classification of its irreducible representations (Section \ref{subsec-rep}), the algebra $\Hmu$ has an Artin--Wedderburn decomposition:
\[\Hmu\cong\bigoplus_{\substack{\lambda\vdash n \\ \lambda\geq\mu^{ord}}} \text{End}\bigl(e_{\mu}(V_{\lambda})\bigr)\ .\]
It is proved in \cite{CP23} that the ideal $I^{\mu}_{N}$ is the one made of the summands above corresponding to partitions with strictly more than $N$ rows.

\begin{exam}\label{exam_ideal}
Take $n=6$ and $\mu=(2,2,2)$. The irreducible representations of $\Hmu$ are indexed by the following partitions (the partitions $\lambda\vdash 6$ with $\lambda\geq(2,2,2)$):
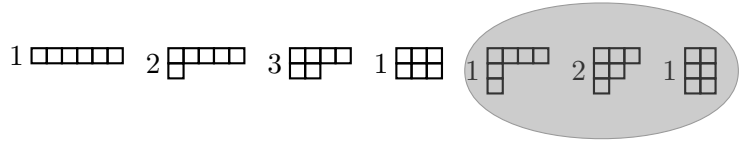
\begin{center}
 \begin{tikzpicture}[scale=0.2]
\diag{-25}{-9}{6};\node at (-26,-9.5) {$1$};\diagg{-16}{-9}{5}{1};\node at (-17,-10) {$2$};\diagg{-8}{-9}{4}{2};\node at (-9,-10) {$3$};
\diagg{-1}{-9}{3}{3};\node at (-2,-10) {$1$};\diaggg{5}{-9}{4}{1}{1};\node at (4,-10.5) {$1$}; \diaggg{12}{-9}{3}{2}{1};\node at (11,-10.5) {$2$};\diaggg{18}{-9}{2}{2}{2};\node at (17,-10.5) {$1$};

\draw[thin, fill=gray,opacity=0.4] (3.5,-10.5)..controls +(0,6) and +(0,6) .. (21.5,-10.5) .. controls +(0,-6) and +(0,-6) .. (3.5,-10.5);

\end{tikzpicture}
\end{center}
The numbers are the dimensions. When $N=2$, the ideal $I^{\mu}_{N}$ corresponds to the three partitions in the shaded area.
\end{exam}

\subsection{A linear basis for the ideal $I^{\mu}_{N}$}

We are going to use the results proved in the preceding sections for the second basis of $\Hmu$. This basis was introduced in Section \ref{subsec_KLbasispar2} for an arbitrary parabolic Hecke algebra, and discussed in Sections \ref{sec-typeA} for the type A. For simplicity of notation here, we will identify a double coset $\cD\in\dSmu$ with its minimal representative $\rmd$ in $X_{JJ}$. Recall that $X_{JJ}$ is the set of elements in $S_n$ with left and right descents disjoint from $J$, where $J$ is associated to $\mu$ by
\[J=\{1,\dots,\mu_1-1,\ \ \ \mu_1+1,\dots,\mu_1+\mu_2-1,\ \ \ \dots\dots\}\ ;\]
that is, in order to get $J$, we remove from $\{1,\dots,n-1\}$ the indices $\mu_1,\mu_1+\mu_2,\dots,\mu_1+\dots+\mu_{d-1}$.

The basis of $\Hmu$ we consider is thus $\{e_{\mu}C^{\da}_{w}e_{\mu}\}_{w\in X_{JJ}}$. Recall also that the shape $sh(w)$ of an element $w\in S_n$ is by definition the partition $\lambda$ which is the shape of the standard tableaux $(P(w),Q(w))$ corresponding to $w$ via the RS correspondence.

\begin{prop}\label{prop_basis_ideal}
A basis of $I^{\mu}_{N}$ is:
\begin{equation}\label{basis-ideal}
\{ e_{\mu}C^{\da}_{w}e_{\mu}\,,\ \ \text{with $w\in X_{JJ}$ such that $sh(w)$ has strictly more than $N$ rows}\,.\}
\end{equation}
\end{prop}
In terms of the cellular basis $\{C^\da_{\textbf{S},\textbf{T}}\}$ from Section \ref{subsec-cellular}, that is when identifying an element $w\in S_n$ with a pair of standard tableaux via the RS correspondence, the basis of $I^{\mu}_{N}$ reads:
\[\{ e_{\mu}C^{\da}_{\textbf{s},\textbf{t}}e_{\mu}\,,\ \ \text{$\textbf{s},\textbf{t}$ have strictly more than $N$ rows and descents disjoint from $J$}\,.\}\]
Note that from the RSK correspondence recalled in (\ref{RSK1}) and the classical combinatorial result recalled in (\ref{Kostka}), we know that the shape $\lambda$ of an element $w\in X_{JJ}$, or equivalently, the shape of a standard tableau with descent disjoint from $J$, must satisfy $\lambda\geq\mu^{ord}$.
\begin{proof}
We use the description of the ideal $I^{\mu}_{N}$ in terms of the representations recalled above. The ideal $I^{\mu}_{N}$ is made of the summands in the Artin--Wedderburn decomposition corresponding to partitions with strictly more than $N$ rows (see Example \ref{exam_ideal}). Luckily, this condition is simply characterised in terms of the dominance ordering. Indeed, denote by $\Hk_{N+1,n}$ the hook partition with $N+1$ rows and of size $n$. This partition indeed satisfies $\Hk_{N+1,n}\geq\mu^{ord}$ due to the condition $d>N$. Then it is easy to see that:
\[\lambda\vdash n\ \text{has strictly more than $N$ rows}\ \ \Leftrightarrow\ \ \lambda\leq \Hk_{N+1,n}\ .\]
This remark allows to write the set (\ref{basis-ideal}) above as:
\[\{ e_{\mu}C^{\da}_{w}e_{\mu}\,,\ \ \text{with $w\in X_{JJ}$ such that $sh(w)\leq \Hk_{N+1,n}$}\,.\}\]
This shows, using the cell order property from Proposition \ref{prop-cells-Hmu2}, that this set indeed spans an ideal of $\Hmu$. This is in fact the sum of the (two-sided) cell ideals associated to the two-sided cells corresponding to shapes $\lambda\leq\Hk_{N+1,n}$. The irreducible representations of $\Hmu$ corresponding to this ideal are exactly the ones corresponding to $I^{\mu}_{N}$, therefore this ideal is $I^{\mu}_{N}$.
\end{proof}

\begin{rema}
We emphasize that we can not use directly the other Kazhdan--Lusztig basis $\{C_w\}_{w\in X^{JJ}}$ of $\Hmu$ since the cell representations correspond to the representations $e_{\mu}(V_{\lambda^t})$. This means that the desired representations correspond to all cells above the shape $\Hk_{N+1,n}^t$ in the cell order, and this does not form a cell ideal.
\end{rema}

\begin{rema}
In the general situation, that is over the defining ring $A$, or for an arbitrary authorised specialisation of $q$ (see Remark \ref{rem-spec}), the representation $\pi_N$ in (\ref{rep-SW}) is still defined, and thus one can wonder if the kernel still admits the description of the previous proposition. It turns out to be true for the usual Hecke algebra; see \cite{GW93} or \cite{Mar92}. We leave this question open for the algebra $\Hmu$.
\end{rema}

\subsection{Conjectures for a generator of the ideal $I^{\mu}_{N}$}

As detailed in \cite{CP23}, in the semisimple regime, it is somehow enough to understand the ideal $I^{\mu}_{N}$ at the first level where it is not trivial, namely when $d=N+1$. Indeed, say we have an element $X$ of $\Hmu$, where $\mu=(\mu_1,\dots,\mu_N,\mu_{N+1})$, which generates the ideal $I^{\mu}_{N}$. If we increase the length of $\mu$ to $\mu^+=(\mu_1,\dots,\mu_N,\mu_{N+1},\mu_{N+2})$ then we can see naturally the algebra $\Hmu$ as a subalgebra of $H^{\mu^+}(S_{n'})$. Then the ideal $I^{\mu^+}_{N}$ of $H^{\mu^+}_{n'}$ will simply be generated by $X$, now seen as an element of $H^{\mu^+}(S_{n'})$ by the natural embedding. This is proved in \cite{CP23} under the (necessary) assumption that $\mu$ and $\mu^+$ are partitions, instead of general compositions. 

For simplicity and due to the preceding short discussion, we will now be considering that
\[d=N+1\,,\ \ \ \ \text{that is,}\ \ \ \mu=(\mu_1,\dots,\mu_N,\mu_{N+1})\ .\]
As always, we have $n=\mu_1+\dots+\mu_{N+1}$. Note that for what follows we do not need to assume that $\mu$ is a partition since we will only be speaking of the ideal $I^{\mu}_{N}$ at level $d=N+1$.

\subsubsection{A first tentative generator}

Here we study the possibility to find a generator of the ideal $I^{\mu}_{N}$ by looking  at the basis $\{e_{\mu}C^{\da}_we_{\mu}\}_{w\in X_{JJ}}$ of $\Hmu$. As mentioned before, one difficulty is that the ideal $I^{\mu}_{N}$ contains more than one irreducible representation. However, we still have a bit of luck, in the sense that the subset of partitions that we need to consider has a simple description in terms of the dominance order, that is, in terms of the cell order corresponding to the basis $\{e_{\mu}C^{\da}_we_{\mu}\}_{w\in X_{JJ}}$ of $\Hmu$.

More precisely, recall that $\Hk_{N+1,n}$ is the hook shape partition with a first column of  size $N+1$ (and total size $n$). As used in the preceding subsection, it is easy to see that
\[\{\lambda\vdash n\ \text{with strictly more than $N$ rows}\}=\{\lambda\vdash n\ \text{such that}\ \lambda\leq \Hk_{N+1,n}\}\ .\]
Therefore, to generate the ideal $I^{\mu}_{N}$, the only choice for the basis $\{e_{\mu}C^{\da}_we_{\mu}\}$ is to take elements corresponding to the cell associated to the hook shape $\Hk_{N+1,n}$.

Moreover, the irreducible representation of $\Hmu$ corresponding to $\Hk_{N+1,n}$ is one-dimensional. Indeed there is a unique semistandard tableau in $\SSTab(\Hk_{N+1,n},\mu)$ corresponsing to the unique standard tableau in $\STab(\Hk_{N+1,n})$ with descent disjoint from $J$. This standard tableau of shape $\Hk_{N+1,n}$, denoted $\tilde{\textbf{t}}_{N+1}$, has the following entries in the first column:
\[1,\ \mu_1+1,\ \mu_1+\mu_2+1,\ \dots,\ \mu_1+\dots+\mu_N+1\ .\]
Its descent, which is $\{\mu_1,\mu_1+\mu_2,\dots,\mu_1+\dots+\mu_N\}$, is indeed disjoint from $J$, which we recall is exactly $\{1,\dots,n\}$ minus the preceding subset.
\begin{exam}\label{exam_idealagain}
Take again $n=6$ and $\mu=(2,2,2)$. The irreducible representations of $\Hmu$ are:
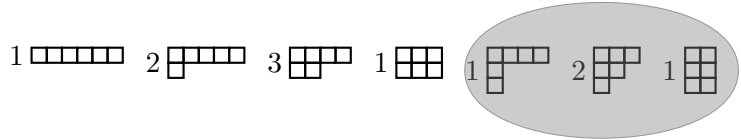
\begin{center}
 \begin{tikzpicture}[scale=0.2]
\diag{-25}{-9}{6};\node at (-26,-9.5) {$1$};\diagg{-16}{-9}{5}{1};\node at (-17,-10) {$2$};\diagg{-8}{-9}{4}{2};\node at (-9,-10) {$3$};
\diagg{-1}{-9}{3}{3};\node at (-2,-10) {$1$};\diaggg{5}{-9}{4}{1}{1};\node at (4,-10.5) {$1$}; \diaggg{12}{-9}{3}{2}{1};\node at (11,-10.5) {$2$};\diaggg{18}{-9}{2}{2}{2};\node at (17,-10.5) {$1$};

\draw[thin, fill=gray,opacity=0.4] (3.5,-10.5)..controls +(0,6) and +(0,6) .. (21.5,-10.5) .. controls +(0,-6) and +(0,-6) .. (3.5,-10.5);

\end{tikzpicture}
\end{center}
and, for $N=2$, the ideal $I^{\mu}_{N}$ corresponds to the three partitions in the shaded area. The hook shape $(4,1,1)$ dominates them all in the dominance order and corresponds to a one-dimensional representation. The unique semistandard tableau of this shape is $\begin{array}{cccc}
\fbox{\scriptsize{$1$}} & \hspace{-0.35cm}\fbox{\scriptsize{$1$}}& \hspace{-0.35cm}\fbox{\scriptsize{$2$}}& \hspace{-0.35cm}\fbox{\scriptsize{$3$}} \\[-0.2em]
\fbox{\scriptsize{$2$}} &\\[-0.2em]
\fbox{\scriptsize{$3$}} &
\end{array}$ corresponding to the unique standard tableau $\tilde{\textbf{t}}_{N+1}=\begin{array}{cccc}
\fbox{\scriptsize{$1$}} & \hspace{-0.35cm}\fbox{\scriptsize{$2$}}& \hspace{-0.35cm}\fbox{\scriptsize{$4$}}& \hspace{-0.35cm}\fbox{\scriptsize{$6$}} \\[-0.2em]
\fbox{\scriptsize{$3$}} &\\[-0.2em]
\fbox{\scriptsize{$5$}} &
\end{array}$ not containing $1,3,5$ in its descent. More examples are given below
\end{exam}

To summarise, if we want to generate the cell ideal $I^{\mu}_{N}$, there is only one choice from the point of view of the basis $e_{\mu}C^{\da}_we_{\mu}$, which we formalise in the following definition
\begin{defi}\label{defYmuN}
We define:
\[Y^{\mu}_N=e_{\mu}C^{\da}_{\tilde{t}_{N+1},\tilde{t}_{N+1}}e_{\mu}=e_{\mu}C^{\da}_{\tilde{w}_{N+1}}e_{\mu}\,,\]
where $\tilde{t}_{N+1}$ is the unique standard tableau of shape $\Hk_{N+1,n}$ with descent disjoint from $J$ and $\tilde{w}_{N+1}$ is the permutation in $S_n$ corresponding to $(\tilde{t}_{N+1},\tilde{t}_{N+1})$ under the RS correspondence.
\end{defi}
The permutation $\tilde{w}_{N+1}$ appearing in the above definition can be described explicitly. It is the permutation of order 2 doing the following transpositions:
\[\mu_1+\dots+\mu_i\ \leftrightarrow\ \mu_1+\dots+\mu_{N+1-i}+1\,,\ \ \ \ \ \ \forall i\leq \frac{N+1}{2}\ .\]
More visually, recall the decomposition of $\{1,\dots,n\}$ into consecutive subsets of sizes $\mu_1,\mu_2,\dots,\mu_{N+1}$. Then send the last letter of the first subset to the first letter of the last subset, then the last letter of the second subset to the first letter of the second to last subset, and so on. If there is an odd number of subsets ($N+1$ is odd), the subset in the middle is untouched. The following diagram should be helpful to visualize $\tilde{w}_{N+1}$ (see also the examples below):
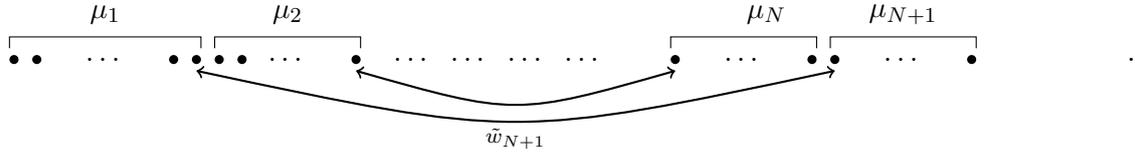
\begin{center}
 \begin{tikzpicture}[scale=0.3]
\draw (0.8,0.5) -- (0.8,1) -- (9.2,1) -- (9.2,0.5);\node at (5,2) {$\mu_1$};
\fill (1,0) circle (0.2);
\fill (2,0) circle (0.2);
\node at (5,0) {$\dots$};

\fill (8,0)  circle (0.2);
\fill (9,0) circle (0.2);

\draw (9.8,0.5) -- (9.8,1) -- (16.2,1) -- (16.2,0.5);\node at (13,2) {$\mu_2$};
\fill (10,0) circle (0.2);
\fill (11,0) circle (0.2);
\node at (13,0) {$\dots$};
\fill (16,0) circle (0.2);

\node at (18.5,0) {$\dots$};

\node at (21,0) {$\dots$};

\node at (23.5,0) {$\dots$};

\node at (26,0) {$\dots$};

\draw (29.8,0.5) -- (29.8,1) -- (36.2,1) -- (36.2,0.5);\node at (34,2) {$\mu_{N}$};
\fill (30,0) circle (0.2);
\node at (33,0) {$\dots$};
\fill (36,0) circle (0.2);

\draw (36.8,0.5) -- (36.8,1) -- (43.2,1) -- (43.2,0.5);\node at (40,2) {$\mu_{N+1}$};
\fill (37,0) circle (0.2);
\node at (40,0) {$\dots$};
\fill (43,0) circle (0.2);

\draw[<->,thick] (9,-0.5) ..controls +(14,-3) .. (37,-0.5);
\draw[<->,thick] (16,-0.5) ..controls +(7,-2) .. (30,-0.5);
\node at (23,-3.5) {\scriptsize{$\tilde{w}_{N+1}$}};

\node at (50,0) {.};
\end{tikzpicture}
\end{center}
It is straightforward to check that the RS correspondence produces the desired pair $(\tilde{\textbf{t}}_{N+1},\tilde{\textbf{t}}_{N+1})$.
\begin{exam}
When $\mu=(2,2,\dots,2)$, we show for small $N$ the standard tableau $\tilde{\textbf{t}}_{N+1}$ and the permutation $\tilde{w}_{N+1}$ in one-line notation:
\[\begin{array}{lll} N=1\ \text{and}\ \mu=(2,2)\ : & \tilde{\textbf{t}}_{N+1}=\begin{array}{cccc}
\fbox{\scriptsize{$1$}} & \hspace{-0.35cm}\fbox{\scriptsize{$2$}}& \hspace{-0.35cm}\fbox{\scriptsize{$4$}}\\[-0.2em]
\fbox{\scriptsize{$3$}}
\end{array} & \tilde{w}_{N+1}=1324\ ,\\[0.8em] 
 N=2\ \text{and}\ \mu=(2,2,2)\ : & \tilde{\textbf{t}}_{N+1}=\begin{array}{cccc}
\fbox{\scriptsize{$1$}} & \hspace{-0.35cm}\fbox{\scriptsize{$2$}}& \hspace{-0.35cm}\fbox{\scriptsize{$4$}}& \hspace{-0.35cm}\fbox{\scriptsize{$6$}} \\[-0.2em]
\fbox{\scriptsize{$3$}} &\\[-0.2em]
\fbox{\scriptsize{$5$}} &
\end{array} & \tilde{w}_{N+1}=153426\ ,\\[1.2em]
 N=3\ \text{and}\ \mu=(2,2,2,2)\ : & \tilde{\textbf{t}}_{N+1}=\begin{array}{ccccc}
\fbox{\scriptsize{$1$}} & \hspace{-0.35cm}\fbox{\scriptsize{$2$}}& \hspace{-0.35cm}\fbox{\scriptsize{$4$}}& \hspace{-0.35cm}\fbox{\scriptsize{$6$}}& \hspace{-0.35cm}\fbox{\scriptsize{$8$}} \\[-0.2em]
\fbox{\scriptsize{$3$}} &\\[-0.2em]
\fbox{\scriptsize{$5$}} &\\[-0.2em]
\fbox{\scriptsize{$7$}} &
\end{array} & \tilde{w}_{N+1}=17354628\ .\end{array} \]
\end{exam}

\subsubsection{A second tentative generator}

Here we recall the definition of another tentative generator of the ideal $I^{\mu}_{N}$ from \cite{CP23}. This is done in several steps.

\paragraph{The element $T_{\gamma_{\mu}}$.} Consider the following element of the Hecke algebra $H(S_n)$:
\begin{equation}\label{TGamma}
T_{\gamma_{\mu}}=T_{\mu_1}\dots T_2\cdot T_{\mu_1+\mu_2}\dots T_3 \cdot \ldots\ldots\cdot T_{\mu_1+\dots+\mu_{N}}\dots T_{N+1}\ , 
\end{equation}
where the dots between $T_{\mu_1+\dots+\mu_{a}}$ and $T_{a+1}$ indicate the product of the generators in decreasing order of their indices. By convention, this product is $1$ when $\mu_1+\dots+\mu_a<a+1$. Note that $T_{\gamma_{\mu}}=1$ only if $\mu=(1,1,\dots,1)$. Graphically, for example if $N=2$, the element $T_{\gamma_{\mu}}$ is depicted as:
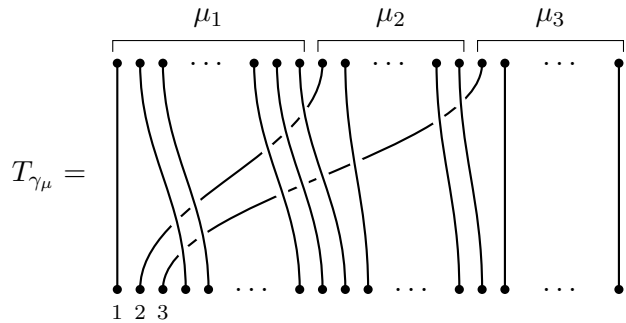
\begin{center}
 \begin{tikzpicture}[scale=0.3]
\node at (-2,0) {$T_{\gamma_{\mu}}=$};
\draw (0.8,5.5) -- (0.8,6) -- (9.2,6) -- (9.2,5.5);\node at (5,7) {$\mu_1$};
\fill (1,5) circle (0.2);\fill (1,-5) circle (0.2);\node at (1,-6) {\scriptsize{$1$}};
\fill (2,5) circle (0.2);\fill (2,-5) circle (0.2);\node at (2,-6) {\scriptsize{$2$}};
\fill (3,5) circle (0.2);\fill (3,-5) circle (0.2);\node at (3,-6) {\scriptsize{$3$}};
\node at (5,5) {$\dots$};
\fill (4,-5) circle (0.2);
\fill (5,-5) circle (0.2);\node at (7,-5) {$\dots$};
\fill (7,5)  circle (0.2);
\fill (8,5)  circle (0.2);
\fill (9,5) circle (0.2);\fill (9,-5) circle (0.2);

\draw (9.8,5.5) -- (9.8,6) -- (16.2,6) -- (16.2,5.5);\node at (13,7) {$\mu_2$};
\fill (10,5) circle (0.2);\fill (10,-5) circle (0.2);
\fill (11,5) circle (0.2);\fill (11,-5) circle (0.2);
\node at (13,5) {$\dots$};\fill (12,-5) circle (0.2);
\node at (14,-5) {$\dots$};
\fill (15,5) circle (0.2);
\fill (16,5) circle (0.2);\fill (16,-5) circle (0.2);

\draw (16.8,5.5) -- (16.8,6) -- (23.2,6) -- (23.2,5.5);\node at (20,7) {$\mu_3$};
\fill (17,5) circle (0.2);\fill (17,-5) circle (0.2);
\fill (18,5) circle (0.2);\fill (18,-5) circle (0.2);
\node at (20.5,5) {$\dots$};\node at (20.5,-5) {$\dots$};
\fill (23,5) circle (0.2);\fill (23,-5) circle (0.2);

\draw[thick] (1,5) -- (1,-5);
\draw[thick] (10,5)..controls +(0,-4) and +(0,+4) .. (2,-5);
\draw[thick] (17,5)..controls +(0,-4) and +(0,+4) .. (3,-5);

\fill[white] (9.2,2.8) circle (0.3);\fill[white] (8.4,2) circle (0.3);\fill[white] (7.6,1.2) circle (0.3);
\fill[white] (4.4,-1.2) circle (0.3);\fill[white] (3.6,-2) circle (0.3);
\fill[white] (4.7,-2.6) circle (0.3);\fill[white] (3.9,-3.3)circle (0.3);
\fill[white] (16,3.3) circle (0.3);\fill[white] (15.1,2.5) circle (0.3);
\fill[white] (11.5,0.6) circle (0.3);\fill[white] (10,0) circle (0.3);
\fill[white] (9.1,-0.3) circle (0.3);\fill[white] (8.2,-0.7) circle (0.3);

\draw[thick] (2,5)..controls +(0,-4) and +(0,+4) .. (4,-5);
\draw[thick] (3,5)..controls +(0,-4) and +(0,+4) .. (5,-5);
\draw[thick] (7,5)..controls +(0,-4) and +(0,+4) .. (9,-5);
\draw[thick] (8,5)..controls +(0,-4) and +(0,+4) .. (10,-5);
\draw[thick] (9,5)..controls +(0,-4) and +(0,+4) .. (11,-5);
\draw[thick] (11,5)..controls +(0,-4) and +(0,+4) .. (12,-5);
\draw[thick] (15,5)..controls +(0,-4) and +(0,+4) .. (16,-5);
\draw[thick] (16,5)..controls +(0,-4) and +(0,+4) .. (17,-5);
\draw[thick] (18,5)..controls +(0,-4) and +(0,+4) .. (18,-5);
\draw[thick] (23,5)..controls +(0,-4) and +(0,+4) .. (23,-5);


\end{tikzpicture}
\end{center}
As a permutation (recall that we read from bottom to top, which corresponds to composing permutations in the usual way, from right to left), we have $\gamma_{\mu}(1)=1$, $\gamma_{\mu}(2)=\mu_1+1$, ..., $\gamma_{\mu}(N+1)=\mu_1+\dots+\mu_N+1$, so the permutation $\gamma_{\mu}$ sends $1,\dots,N+1$, perserving their order, to the first number in each of the subsets of sizes $\mu_1,\mu_2,\dots,\mu_{N+1}$. The elements after $N+1$ are ``pushed to the left'', namely, they are sent, preserving their order, to the remaining available elements.

In the Hecke algebra, the element $T_{\gamma_{\mu}}$ above is obtained from a reduced expression of $\gamma_{\mu}$ by using only positive crossing (left strand above right strand), that is $T_{\gamma_{\mu}}$ involves only generators $T_i$'s and not their inverses.

We note at once the following property. We have
\begin{equation}\label{Tgammap}
e_{\mu}T_{\gamma_{\mu}}XT_{\gamma_{\mu}}^{-1}e_{\mu}=e_{\mu}T_{\gamma'_{\mu}}XT_{\gamma'_{\mu}}^{-1}e_{\mu}\,,
\end{equation}
where, for example for $N=2$, $T_{\gamma'_{\mu}}$ is any element of the form
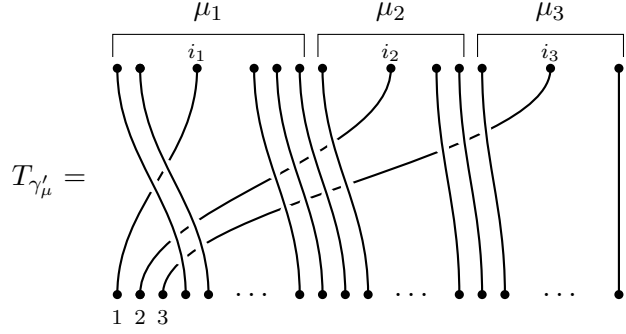
\begin{center}
 \begin{tikzpicture}[scale=0.3]
\node at (-2,0) {$T_{\gamma'_{\mu}}=$};
\draw (0.8,5.5) -- (0.8,6.5) -- (9.2,6.5) -- (9.2,5.5);\node at (5,7.5) {$\mu_1$};
\node at (4.5,5.7){\scriptsize{$i_1$}};
\fill (1,5) circle (0.2);\fill (1,-5) circle (0.2);
\fill (2,5) circle (0.2);\fill (2,-5) circle (0.2);\node at (1,-6) {\scriptsize{$1$}};
\fill (2,5) circle (0.2);\fill (2,-5) circle (0.2);\node at (2,-6) {\scriptsize{$2$}};
\fill (3,-5) circle (0.2);\node at (3,-6) {\scriptsize{$3$}};

\fill (4,-5) circle (0.2);
\fill (5,-5) circle (0.2);\node at (7,-5) {$\dots$};
\fill (7,5)  circle (0.2);
\fill (8,5)  circle (0.2);
\fill (9,5) circle (0.2);\fill (9,-5) circle (0.2);

\draw (9.8,5.5) -- (9.8,6.5) -- (16.2,6.5) -- (16.2,5.5);\node at (13,7.5) {$\mu_2$};
\node at (13,5.7){\scriptsize{$i_2$}};
\fill (10,5) circle (0.2);\fill (10,-5) circle (0.2);
\fill (11,-5) circle (0.2);
\fill (12,-5) circle (0.2);
\node at (14,-5) {$\dots$};
\fill (15,5) circle (0.2);
\fill (16,5) circle (0.2);\fill (16,-5) circle (0.2);

\draw (16.8,5.5) -- (16.8,6.5) -- (23.2,6.5) -- (23.2,5.5);\node at (20,7.5) {$\mu_3$};
\node at (20,5.7){\scriptsize{$i_3$}};
\fill (17,5) circle (0.2);\fill (17,-5) circle (0.2);
\fill (18,-5) circle (0.2);
\node at (20.5,-5) {$\dots$};
\fill (23,5) circle (0.2);\fill (23,-5) circle (0.2);

\draw[thick] (4.5,5)..controls +(0,-4) and +(0,+4) .. (1,-5);
\fill (4.5,5) circle (0.2);
\draw[thick] (13,5)..controls +(0,-4) and +(0,+4) .. (2,-5);
\fill (13,5) circle (0.2);
\draw[thick] (20,5)..controls +(0,-4) and +(0,+4) .. (3,-5);
\fill (20,5) circle (0.2);

\fill[white] (10.5,1.6) circle (0.3);\fill[white] (9.65,1.2) circle (0.3);\fill[white] (8.7,0.8) circle (0.3);\fill[white] (7.85,0.4) circle (0.3);
\fill[white] (11.1,-0.2) circle (0.3);\fill[white] (10.25,-0.6) circle (0.3);\fill[white] (9.3,-0.95) circle (0.3);\fill[white] (8.45,-1.3) circle (0.3);
\fill[white] (15.25,1.5) circle (0.3);\fill[white] (16.15,1.9) circle (0.3);\fill[white] (17.1,2.3) circle (0.3);
\fill[white] (3.8,-3.4)circle (0.3);\fill[white] (4.7,-2.9) circle (0.3);
\fill[white] (3.5,-2.4)circle (0.3);\fill[white] (4.4,-1.9) circle (0.3);
\fill[white] (2.6,-0.2)circle (0.3);\fill[white] (3.2,0.65) circle (0.3);

\draw[thick] (1,5)..controls +(0,-4) and +(0,+4) .. (4,-5);
\draw[thick] (2,5)..controls +(0,-4) and +(0,+4) .. (5,-5);
\draw[thick] (7,5)..controls +(0,-4) and +(0,+4) .. (9,-5);
\draw[thick] (8,5)..controls +(0,-4) and +(0,+4) .. (10,-5);
\draw[thick] (9,5)..controls +(0,-4) and +(0,+4) .. (11,-5);
\draw[thick] (10,5)..controls +(0,-4) and +(0,+4) .. (12,-5);
\draw[thick] (15,5)..controls +(0,-4) and +(0,+4) .. (16,-5);
\draw[thick] (16,5)..controls +(0,-4) and +(0,+4) .. (17,-5);
\draw[thick] (17,5)..controls +(0,-4) and +(0,+4) .. (18,-5);
\draw[thick] (23,5)..controls +(0,-4) and +(0,+4) .. (23,-5);

\end{tikzpicture}
\end{center}
that is, we can instead send $1$ to $i_1$, $2$ to $i_2$, ..., for any choice of $i_1$ in the first set of size $\mu_1$, any choice of $i_2$ in the following subset of size $\mu_2$ and so on. Indeed note that $T_{\gamma'_{\mu}}$ is obtained from $T_{\gamma_{\mu}}$ by precomposing by elements $T_i$ where $i\in J$. The explicit formula is:
\[T_{\gamma'_{\mu}}=T_{i_1-1}\dots T_1\cdot T_{i_2-1}\dots  T_{\mu_1+1}\cdot\ldots\cdot T_{i_{N+1}-1}\dots T_{\mu_1+\dots+\mu_N+1}\,\cdot\,T_{\gamma_{\mu}}\ .\]
All elements appearing before $T_{\gamma_{\mu}}$ have indices in $J$ and therefore satisfy $e_{\mu}T_i=q e_{\mu}$. The additional powers of $q$ appearing are cancelled by the presence of $T_{\gamma'_{\mu}}^{-1}$ next to $e_{\mu}$ on the other side. 

\paragraph{The second tentative generator $X^{\mu}_{N}$.} Now denote $w_{N+1}$ the longest element of the symmetric group $S_{N+1}$ and consider the corresponding Kazhdan--Lusztig element in $H(S_{N+1})$:
\begin{equation}\label{antisymmetriser}
C^{\da}_{w_{N+1}}=\sum_{w\in S_{N+1}}(-1)^{\ell(w)}q^{\ell(w_{N+1})-\ell(w)}T_w\ .
\end{equation}
It is called the (unnormalised) $q$-antisymmetriser of $H(S_{N+1})$. Then, we see this element as an element of $H(S_n)$ by the natural embedding of $H(S_{N+1})$ into $H(S_n)$. Finally, we are ready to define our second candidate for a generator.
\begin{defi}\label{defXmuN}
We define
\begin{equation}\label{def_Xmu}
X^{\mu}_N=e_{\mu}T_{\gamma_{\mu}}C^{\da}_{w_{N+1}}T_{\gamma_{\mu}}^{-1}e_{\mu}\ .
\end{equation}
where $T_{\gamma_{\mu}}$ was introduced above in (\ref{TGamma}).
\end{defi}
Note the other possible formulas for $X^{\mu}_N$ from (\ref{Tgammap}). We will give others below, see Remark \ref{rema_Tgamma}. We concede that the algebraic definition of the generator looks complicated. However it has a nice and relatively simple diagrammatic interpretation in terms of fused braids as developed in \cite{CP23}. We illustrate this with an example.

\begin{exam}\label{exam_gen}
For example, take $\mu=(2,2,2)$ and $N=2$. In the Hecke algebra, the element $C^{\da}_{w_{3}}$ of $H(S_3)$  is depicted as:
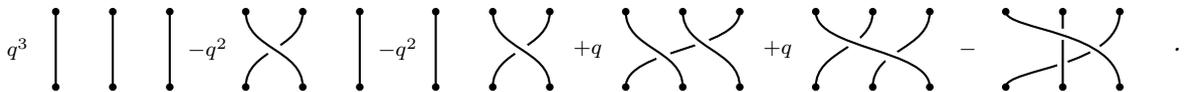
\begin{center}
 \begin{tikzpicture}[scale=0.25]
\node at (-1,10) {$\scriptstyle{q^3}$};
\fill (1,12) circle (0.2cm);\fill (1,8) circle (0.2cm);
\draw[thick] (1,12) -- (1,8);
\fill (4,12) circle (0.2cm);\fill (4,8) circle (0.2cm);
\draw[thick] (4,12) -- (4,8);
\fill (7,12) circle (0.2cm);\fill (7,8) circle (0.2cm);
\draw[thick] (7,12) -- (7,8);
\node at (9,10) {$\scriptstyle{-q^2}$};
\fill (11,12) circle (0.2cm);\fill (11,8) circle (0.2cm);
\draw[thick] (14,12)..controls +(0,-2) and +(0,+2) .. (11,8);\fill[white] (12.5,10) circle (0.4);
\draw[thick] (11,12)..controls +(0,-2) and +(0,+2) .. (14,8);
\fill (14,12) circle (0.2cm);\fill (14,8) circle (0.2cm);
\fill (17,12) circle (0.2cm);\fill (17,8) circle (0.2cm);
\draw[thick] (17,12) -- (17,8);
\node at (19,10) {$\scriptstyle{-q^2}$};
\fill (21,12) circle (0.2cm);\fill (21,8) circle (0.2cm);
\draw[thick] (21,12) -- (21,8);
\fill (24,12) circle (0.2cm);\fill (24,8) circle (0.2cm);
\draw[thick] (27,12)..controls +(0,-2) and +(0,+2) .. (24,8);\fill[white] (25.5,10) circle (0.4);
\fill (27,12) circle (0.2cm);\fill (27,8) circle (0.2cm);
\draw[thick] (24,12)..controls +(0,-2) and +(0,+2) .. (27,8);
\node at (29,10) {$\scriptstyle{+q}$};
\draw[thick] (37,12)..controls +(0,-2) and +(0,+2) .. (31,8);\fill[white] (33,9.6) circle (0.4);\fill[white] (35,10.4) circle (0.4);
\fill (31,12) circle (0.2cm);\fill (31,8) circle (0.2cm);
\draw[thick] (31,12)..controls +(0,-2) and +(0,+2) .. (34,8);
\fill (34,12) circle (0.2cm);\fill (34,8) circle (0.2cm);
\draw[thick] (34,12)..controls +(0,-2) and +(0,+2) .. (37,8);
\fill (37,12) circle (0.2cm);\fill (37,8) circle (0.2cm);
\node at (39,10) {$\scriptstyle{+q}$};
\fill (41,12) circle (0.2cm);\fill (41,8) circle (0.2cm);
\draw[thick] (47,12)..controls +(0,-2) and +(0,+2) .. (44,8);
\draw[thick] (44,12)..controls +(0,-2) and +(0,+2) .. (41,8);
\fill[white] (43,10.4) circle (0.4);\fill[white] (45,9.6) circle (0.4);
\draw[thick] (41,12)..controls +(0,-2) and +(0,+2) .. (47,8);
\fill (44,12) circle (0.2cm);\fill (44,8) circle (0.2cm);
\fill (47,12) circle (0.2cm);\fill (47,8) circle (0.2cm);
\node at (49,10) {$\scriptstyle{-}$};
\fill (51,12) circle (0.2cm);\fill (51,8) circle (0.2cm);
\fill (54,12) circle (0.2cm);\fill (54,8) circle (0.2cm);
\fill (57,12) circle (0.2cm);\fill (57,8) circle (0.2cm);
\draw[thick] (57,12)..controls +(0,-3) and +(0,+1) .. (51,8);
\fill[white] (54,9.2) circle (0.3);
\draw[thick] (54,12) -- (54,8);
\fill[white] (54,10.8) circle (0.3);\fill[white] (55.7,10) circle (0.3);
\draw[thick] (51,12)..controls +(0,-1) and +(0,+3) .. (57,8);
\node at (60,10) {.};
\end{tikzpicture}
\end{center}
This is the sum over all standard basis elements of $H(S_3)$ and the coefficient in front of the basis element $T_w$ is $(-1)^{\ell(w)}q^{\ell(w_3)-\ell(w)}$ where $w_3$ is the longest element (here of length 3).

Here is the diagrammatic depiction of the element $X^{\mu}_N$ defined above:
\begin{center}
 \begin{tikzpicture}[scale=0.25] 
\node at (-1,0) {$\scriptstyle{q^3}$};
\fill (1,2) ellipse (0.6cm and 0.2cm);\fill (1,-2) ellipse (0.6cm and 0.2cm);
\draw[thick] (0.8,2) -- (0.8,-2);\draw[thick] (1.2,2) -- (1.2,-2);
\fill (4,2) ellipse (0.6cm and 0.2cm);\fill (4,-2) ellipse (0.6cm and 0.2cm);
\draw[thick] (3.8,2) -- (3.8,-2);\draw[thick] (4.2,2) -- (4.2,-2);
\fill (7,2) ellipse (0.6cm and 0.2cm);\fill (7,-2) ellipse (0.6cm and 0.2cm);
\draw[thick] (6.8,2) -- (6.8,-2);\draw[thick] (7.2,2) -- (7.2,-2);
\node at (9,0) {$\scriptstyle{-q^{2}}$};
\fill (11,2) ellipse (0.6cm and 0.2cm);\fill (11,-2) ellipse (0.6cm and 0.2cm);!
\draw[thick] (10.8,2) -- (10.8,-2);\draw[thick] (13.8,2)..controls +(0,-2) and +(0,+2) .. (11.2,-2);\fill[white] (12.5,0) circle (0.4);
\fill (14,2) ellipse (0.6cm and 0.2cm);\fill (14,-2) ellipse (0.6cm and 0.2cm);
\draw[thick] (14.2,2) -- (14.2,-2);\draw[thick] (11.2,2)..controls +(0,-2) and +(0,+2) .. (13.8,-2);
\fill (17,2) ellipse (0.6cm and 0.2cm);\fill (17,-2) ellipse (0.6cm and 0.2cm);
\draw[thick] (16.8,2) -- (16.8,-2);\draw[thick] (17.2,2) -- (17.2,-2);
\node at (19,0) {$\scriptstyle{-q^{2}}$};
\fill (21,2) ellipse (0.6cm and 0.2cm);\fill (21,-2) ellipse (0.6cm and 0.2cm);
\draw[thick] (20.8,2) -- (20.8,-2);\draw[thick] (21.2,2) -- (21.2,-2);
\fill (24,2) ellipse (0.6cm and 0.2cm);\fill (24,-2) ellipse (0.6cm and 0.2cm);
\draw[thick] (23.8,2) -- (23.8,-2);\draw[thick] (26.8,2)..controls +(0,-2) and +(0,+2) .. (24.2,-2);\fill[white] (25.5,0) circle (0.4);
\fill (27,2) ellipse (0.6cm and 0.2cm);\fill (27,-2) ellipse (0.6cm and 0.2cm);
\draw[thick] (24.2,2)..controls +(0,-2) and +(0,+2) .. (26.8,-2);\draw[thick] (27.2,2) -- (27.2,-2);
\node at (29,0) {$\scriptstyle{+}$};
\draw[thick] (36.8,2)..controls +(0,-2) and +(0,+2) .. (31.2,-2);\draw[thick] (37.2,2) -- (37.2,-2);
\fill[white] (33,-0.4) circle (0.3);\fill[white] (35,0.4) circle (0.3);;\fill[white] (34,0) circle (0.3);
\fill (31,2) ellipse (0.6cm and 0.2cm);\fill (31,-2) ellipse (0.6cm and 0.2cm);
\draw[thick] (30.8,2) -- (30.8,-2);\draw[thick] (31.2,2)..controls +(0,-2) and +(0,+2) .. (33.8,-2);
\fill (34,2) ellipse (0.6cm and 0.2cm);\fill (34,-2) ellipse (0.6cm and 0.2cm);
\draw[thick] (33.8,2) -- (34.2,-2);\draw[thick] (34.2,2)..controls +(0,-2) and +(0,+2) .. (36.8,-2);
\fill (37,2) ellipse (0.6cm and 0.2cm);\fill (37,-2) ellipse (0.6cm and 0.2cm);
\node at (39,0) {$\scriptstyle{+q^{2}}$};
\draw[thick] (46.8,2)..controls +(0,-2) and +(0,+2) .. (44.2,-2);
\draw[thick] (43.8,2)..controls +(0,-2) and +(0,+2) .. (41.2,-2);
\fill[white] (43,0.4) circle (0.3);\fill[white] (45,-0.4) circle (0.3);
\draw[thick] (41.2,2)..controls +(0,-2) and +(0,+2) .. (46.8,-2);
\fill[white] (44,0) circle (0.3);
\fill (41,2) ellipse (0.6cm and 0.2cm);\fill (41,-2) ellipse (0.6cm and 0.2cm);
\draw[thick] (40.8,2) -- (40.8,-2);
\fill (44,2) ellipse (0.6cm and 0.2cm);\fill (44,-2) ellipse (0.6cm and 0.2cm);
\draw[thick] (44.2,2) -- (43.8,-2);
\fill (47,2) ellipse (0.6cm and 0.2cm);\fill (47,-2) ellipse (0.6cm and 0.2cm);
\draw[thick] (47.2,2) -- (47.2,-2);
\node at (49,0) {$\scriptstyle{-}$};
\draw[thick] (56.8,2)..controls +(0,-3) and +(0,+1) .. (51.2,-2);
\fill[white] (53.5,-0.8) circle (0.3);\fill[white] (54.5,-0.4) circle (0.3);
\draw[thick] (53.8,2)..controls +(-0.5,-1) and +(-0.5,1) .. (53.8,-2);
\fill[white] (53.5,0.85) circle (0.3);
\fill[white] (54,0.8) circle (0.3);\fill[white] (55.6,0) circle (0.3);
\draw[thick] (51.2,2)..controls +(0,-1) and +(0,+3) .. (56.8,-2);
\fill[white] (54.5,0.4) circle (0.3);
\draw[thick] (54.2,2)..controls +(0.5,-1) and +(0.5,1) .. (54.2,-2);
\draw[thick] (50.8,2) -- (50.8,-2);
\fill (51,2) ellipse (0.6cm and 0.2cm);\fill (51,-2) ellipse (0.6cm and 0.2cm);
\fill (54,2) ellipse (0.6cm and 0.2cm);\fill (54,-2) ellipse (0.6cm and 0.2cm);
\fill (57,2) ellipse (0.6cm and 0.2cm);\fill (57,-2) ellipse (0.6cm and 0.2cm);
\draw[thick] (57.2,2) -- (57.2,-2);
\node at (60,0) {.};
\end{tikzpicture}
\end{center}
Each term is seen as an element of $\Hmu=e_{\mu}H(S_6)e_{\mu}$ as follows. The strands represent an element of $H(S_6)$ in the usual way, while the black ellipses represent $q$-symmetrisers (here on two strands since $\mu=(2,2,2)$). So the three black ellipses on top and on bottom represents the left and right multiplication by $e_{\mu}$.

The diagrammatic procedure is as follows: Start with a standard basis element of $H(S_3)$; promote the dots into black ellipses; add to each ellipse vertical strands so that there are two strands attached to each ellipse. We add these strands on top of the picture (with respect to the sheet of paper on which it is drawn). We apply this procedure to each term in $C^{\da}_{w_{3}}$. 

Note that we have used the property $T_i^{\pm1}e_{\mu}=q^{\pm1}e_{\mu}$ if $i\in J$ to simplify some crossings and to have elements which are minimal-length representatives. In doing so, the coefficients changed a little, but the rule is simple. What we get is that in front of an element corresponding to a minimal-length representative $r^-(\cD)$, the cofficient is $(-1)^{\sharp}q^{\ell(w_3)-\sharp}$, where $\sharp$ is the number of crossings counted with signs.
\end{exam}

\begin{rema}
There is a connection between the longest element $w_{N+1}$ and the permutation $\tilde{w}_{N+1}$ introduced in Definition \ref{defYmuN}. This is through the permutation $\gamma_{\mu}$ introduced above in (\ref{TGamma}) and it goes as follows. Under the RS correspondence, the element $w_{N+1}$ corresponds to $(\textbf{t}_{N+1},\textbf{t}_{N+1})$, where $\textbf{t}_{N+1}$ is the standard tableau of shape $\Hk_{N+1,n}$ which has the following entries in the first column:
\[1,2,3,\dots,N+1\ .\]
Note that the descent of $\textbf{t}_{N+1}$ is not disjoint from $J$. The only standard tableau of this shape with descent disjoint from $J$ was denoted $\tilde{\textbf{t}}_{N+1}$ in Definition \ref{defYmuN} and corresponds to the permutation $\tilde{w}_{N+1}$. Now it is easy to check that the permutation $\gamma_{\mu}$ is exactly the permutation relating the two standard tableaux $\textbf{t}_{N+1}$ and $\tilde{\textbf{t}}_{N+1}$:
\[\tilde{\textbf{t}}_{N+1}=\gamma_{\mu}(\textbf{t}_{N+1})\ ,\]
where the right hand side means $\gamma_{\mu}$ applied to the entry of $\textbf{t}_{N+1}$. This can be seen as the heuristic behind the conjectures of the next section.
\end{rema}

\subsubsection{Conjectures and partial results}

\paragraph{The conjectures.} To summarise, we have introduced in Definitions \ref{defYmuN} and \ref{defXmuN} two elements of the algebra $\Hmu$ in the following form:
\[X^{\mu}_N=e_{\mu}T_{\gamma_{\mu}}C^{\da}_{w_{N+1}}T_{\gamma_{\mu}}^{-1}e_{\mu} \ \ \ \text{and}\ \ \ \ Y^{\mu}_N=e_{\mu}C^{\da}_{\tilde{w}_{N+1}}e_{\mu}\ .\]
Now we formulate our conjectures regarding these elements and the ideal $I^{\mu}_N$ we are interested in.

\begin{conj}\label{conjectures} $\ $
\begin{enumerate}
\item[a)] The element $X^{\mu}_N$ generates the ideal $I^{\mu}_{N}$.
\item[b)] The element $Y^{\mu}_N$ generates the ideal $I^{\mu}_{N}$.
\item[c)] The two elements $X^{\mu}_N$ and $Y^{\mu}_N$ are equal.
\end{enumerate}
\end{conj}
Statement a) above was conjectured in \cite[\S 9]{CP23}. In there, it is shown, first, that the element $X^{\mu}_N$ does belong to the ideal $I_N^{\mu}$. Moreover statement a) was proved in some particular cases, namely for $N=2$ and any $\mu$, and for any $N$ if $\mu$ either contains only $1$'s and $2$'s, or contains only a single part different than $1$. We refer to \cite{CP23} for more details.

The novelty of the present work is to conjecture statements b) and c). Needless to say that if c) is true then statements a) and b) become equivalent (but of course a) and b) can both be true while c) is not). We will manage below to give some supporting evidence for c) in general and to prove all statements a), b) and c) in some special cases.

\begin{rema}
It is also conjectured in \cite[Conjecture 9.3]{CP23} that the (conjectural) generator $X^{\mu}_N$ of the ideal is central in $\Hmu$.
\end{rema}

\paragraph{Partial results.} 
Recall that the element $Y^{\mu}_N=e_{\mu}C^{\da}_{\tilde{w}_{N+1}}e_{\mu}$ has in particular the important property of being stable by the bar involution $\overline{\,\cdot\,}$, from the general property of Kazhdan--Lusztig elements, see Proposition \ref{prop-CT-par2}. The bar involution is the morphism sending each $T_i$ to its inverse $T_i^{-1}$ and $q$ to $q^{-1}$. What we can show in general is that this stability property is satisfied by the element $X^{\mu}_N$.
\begin{prop}
We have:
\[\overline{X^\mu_N}=X^\mu_N\,,\ \ \ \text{namely,}\ \ \ e_{\mu}T_{\gamma_{\mu}}C^{\da}_{w_{N+1}}T_{\gamma_{\mu}}^{-1}e_{\mu}=e_{\mu}\overline{T_{\gamma_{\mu}}}C^{\da}_{w_{N+1}}\overline{T_{\gamma_{\mu}}}^{-1}e_{\mu}\ .\]
\end{prop}
\begin{proof} We are going to show the following identity:
\begin{equation}\label{PTCbar}
e_{\mu}T_{\gamma_{\mu}}C^{\da}_{w_{N+1}}=e_{\mu}\overline{T_{\gamma_{\mu}}}C^{\da}_{w_{N+1}}\ .
\end{equation}
Assume this identity is verified and apply to it the antiautomorphism $\iota$ of $H_n$ sending each generator $T_i$ to itself, followed by the involution $\overline{\,\cdot\,}$. The map $\iota$ leaves invariant both $e_{\mu}$ and $C^{\da}_{w_{N+1}}$ (see (\ref{basic-prop-P}) and (\ref{antisymmetriser})) as does the involution $\overline{\,\cdot\,}$ (see (\ref{barPJ})). The composition of both maps sends $T_{\gamma_{\mu}}$ to $T_{\gamma_{\mu}}^{-1}$. Therefore, we obtain from (\ref{PTCbar})
\[C^{\da}_{w_{N+1}}T_{\gamma_{\mu}}^{-1}e_{\mu}=C^{\da}_{w_{N+1}}\overline{T_{\gamma_{\mu}}}^{-1}e_{\mu}\ .\]
Combined with (\ref{PTCbar}), this proves the statement of the proposition.

The proof of (\ref{PTCbar}) will be using induction on $N$. First take $N=1$, so that $\mu=(\mu_1,\mu_2)$. In this case, we have:
\[T_{\gamma_{\mu}}=T_{\mu_1}\dots T_2\ \ \ \ \text{and}\ \ \ \ C^{\da}_{w_{N+1}}=q-T_1\ .\]
Let $i\in\{2,\dots,\mu_1\}$. We claim that:
\begin{equation}\label{relN1}
e_{\mu}T^{\pm1}_{\mu_1}\dots T^{\pm1}_{i+1}\check T_iT_{i-1} \dots T_{2}(q-T_1)=0\ .
\end{equation}
To be precise, the exponants $\pm1$ means that we can choose independently $+1$ or $-1$ for all generators to the left of $T_i$, and the notation $\check T_i$ means that $T_i$ is omitted. The statement follows from the fact that $T_{i-1} \dots T_{2}(q-T_1)$ commutes to the left and hits $e_{\mu}$ where each element $T_1,T_2,\dots,T_{i-1}$ is replaced by $q$. In particular, the factor $(q-T_1)$ gives $0$.

Now, in $e_{\mu}T_{\gamma_{\mu}}C^{\da}_{w_{N+1}}$, we can replace (from left to right) all the $T_i$'s in $T_{\gamma_{\mu}}$ by their inverses, using the relation $T_i=T_i^{-1}+(q-q^{-1})$ and (\ref{relN1}). This proves (\ref{PTCbar}) for $N=1$.

\medskip
Next, let $N>1$. Setting $\mu^-=(\mu_1,\dots,\mu_N)$, we have:
\[T_{\gamma_\mu}=T_{\gamma_{\mu^-}}\cdot T_{\mu_1+\dots+\mu_N} \dots T_{N+1}\ \ \ \ \text{with}\ T_{\gamma_{\mu^-}}\in\langle T_1,\dots,T_{\mu_1+\dots+\mu_{N-1}}\rangle\ .\]

Similarly as for $N=1$, we start by proving that
\begin{equation}\label{element_diagrammaticproof}
e_{\mu}T_{\gamma_{\mu^-}}T^{\pm1}_{\mu_1+\dots+\mu_N}\dots T^{\pm1}_{i+1}\check T_iT_{i-1} \dots T_{N+1}C^{\da}_{w_{N+1}}=0\ \ \ \ \ \forall i=N+1,\dots,\mu_1+\dots+\mu_N\ .
\end{equation}
This is the difficult part of the proof. Recall that the element $T_{\gamma_{\mu}}$ is as follows:
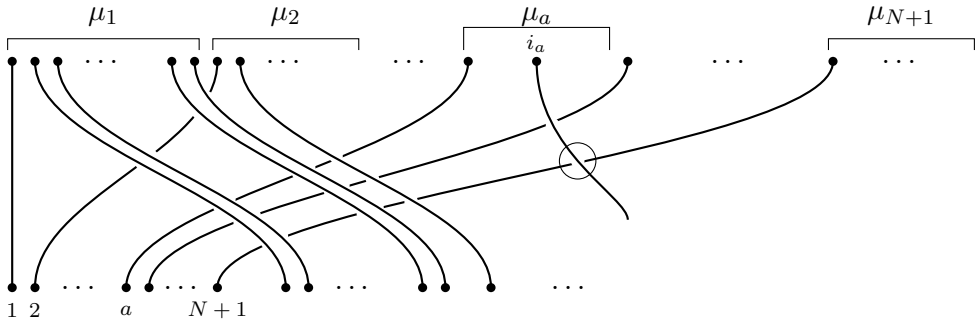
\begin{center}
 \begin{tikzpicture}[scale=0.3]
\draw (0.8,5.5) -- (0.8,6) -- (9.2,6) -- (9.2,5.5);\node at (5,7) {$\mu_1$};
\fill (1,5) circle (0.2);\fill (1,-5) circle (0.2);\node at (1,-6) {\scriptsize{$1$}};
\fill (2,5) circle (0.2);\fill (2,-5) circle (0.2);\node at (2,-6) {\scriptsize{$2$}};
\fill (3,5) circle (0.2);\node at (5,5) {$\dots$};
\node at (4,-5) {$\dots$};
\fill (6,-5) circle (0.2);
\node at (6,-6) {\scriptsize{$a$}};
\fill (7,-5) circle (0.2);
\node at (8.5,-5) {$\dots$};
\fill (10,-5) circle (0.2);
\node at (10,-6) {\scriptsize{$N+1$}};

\fill (8,5)  circle (0.2);
\fill (9,5) circle (0.2);

\draw (9.8,5.5) -- (9.8,6) -- (16.2,6) -- (16.2,5.5);\node at (13,7) {$\mu_2$};
\fill (10,5) circle (0.2);\fill (13,-5) circle (0.2);
\fill (11,5) circle (0.2);\fill (14,-5) circle (0.2);
\node at (13,5) {$\dots$};\node at (16,-5) {$\dots$};
\fill (19,-5) circle (0.2);
\fill (20,-5) circle (0.2);

\node at (18.5,5) {$\dots$};

\draw (20.8,5.5) -- (20.8,6.5) -- (27.2,6.5) -- (27.2,5.5);\node at (24,7) {$\mu_a$};
\fill (21,5) circle (0.2);\fill (22,-5) circle (0.2);
\fill (24,5) circle (0.2);\node at (24,5.8) {\scriptsize{$i_a$}};
\node at (25.5,-5) {$\dots$};
\fill (28,5) circle (0.2);

\node at (32.5,5) {$\dots$};
\draw (36.8,5.5) -- (36.8,6) -- (43.2,6) -- (43.2,5.5);\node at (40,7) {$\mu_{N+1}$};
\fill (37,5) circle (0.2);
\node at (40,5) {$\dots$};

\draw[thick] (10,5)..controls +(0,-4) and +(0,+4) .. (2,-5);
\draw[thick] (21,5)..controls +(0,-4) and +(0,+4) .. (6,-5);
\draw[thick] (28,5)..controls +(0,-4) and +(0,+4) .. (7,-5);
\draw[thick] (37,5)..controls +(0,-4) and +(0,+4) .. (10,-5);

\fill[white] (7.1,0.8) circle (0.3);\fill[white] (6.7,0.5)circle (0.3);
\fill[white] (10.5,-1.1) circle (0.3);\fill[white] (10.1,-1.4)circle (0.3);
\fill[white] (11.5,-1.8) circle (0.4);\fill[white] (11,-2.1)circle (0.3);
\fill[white] (12.7,-2.7) circle (0.4);\fill[white] (12.2,-3)circle (0.3);
\fill[white] (9.5,3.35) circle (0.35);\fill[white] (9.2,2.8)circle (0.35);
\fill[white] (14,0.3) circle (0.3);\fill[white] (13.6,0.1)circle (0.35);
\fill[white] (15.4,-0.5) circle (0.4);\fill[white] (15,-0.7)circle (0.35);
\fill[white] (16.9,-1.5) circle (0.4);\fill[white] (16.4,-1.7)circle (0.3);
\fill[white] (15.2,0.8) circle (0.3);
\fill[white] (16.8,-0.2) circle (0.3);
\fill[white] (18.5,-1.2) circle (0.3);
\fill[white] (24.7,2.2) circle (0.3);
\fill[white] (25.8,0.6) circle (0.3);
\draw (25.8,0.6) circle (0.8);

\draw[thick] (1,5) -- (1,-5);
\draw[thick] (2,5)..controls +(0,-4) and +(0,+4) .. (13,-5);
\draw[thick] (3,5)..controls +(0,-4) and +(0,+4) .. (14,-5);
\draw[thick] (8,5)..controls +(0,-4) and +(0,+4) .. (19,-5);
\draw[thick] (9,5)..controls +(0,-4) and +(0,+4) .. (20,-5);
\draw[thick] (11,5)..controls +(0,-4) and +(0,+4) .. (22,-5);
\draw[thick] (24,5)..controls +(0,-4) and +(0,+1) .. (28,-2);

\node at (50,0) {.};
\end{tikzpicture}
\end{center}
where we have circled the crossing that is removed when we remove $T_i$ in (\ref{element_diagrammaticproof}). This crossing involves the strand coming from dot number $\mu_1+\dots+\mu_N+1$ and another strand. Denote $i_a$ the dot from where this strand is coming. This $i_a$ belongs to one of the subsets, say the one of size $\mu_a$, and can not be the first element of this subset. The $\pm1$ crossings in (\ref{element_diagrammaticproof}) will also be on the strand coming from dot number $\mu_1+\dots+\mu_N+1$, and they will be to the right of the circled crossing and therefore will not interfere at all with our reasoning.

First, recall from (\ref{Tgammap}) the freedom that we have in $T_{\gamma_{\mu}}$. We use this freedom to move a little bit the starting point of the strands descending on dot number $a$. This does not change the left-hand-side of (\ref{element_diagrammaticproof}) up to an irrelevant factor (a power of $q$).  Doing so, the element in (\ref{element_diagrammaticproof}) between the idempotent $e_{\mu}$ and $C^{\da}_{w_{N+1}}$ is:
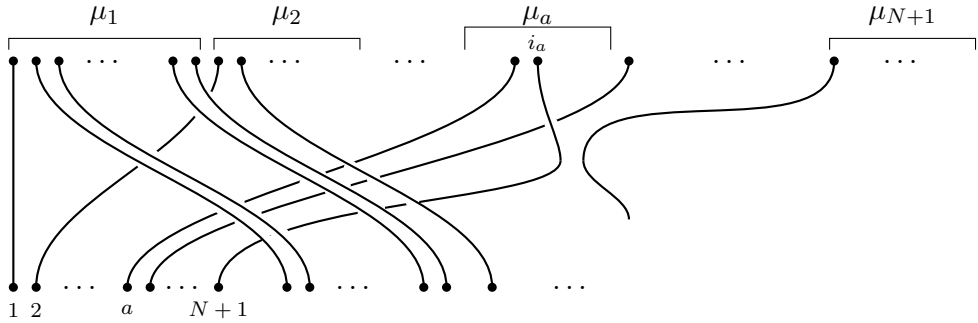
\begin{center}
 \begin{tikzpicture}[scale=0.3]
\draw (0.8,5.5) -- (0.8,6) -- (9.2,6) -- (9.2,5.5);\node at (5,7) {$\mu_1$};
\fill (1,5) circle (0.2);\fill (1,-5) circle (0.2);\node at (1,-6) {\scriptsize{$1$}};
\fill (2,5) circle (0.2);\fill (2,-5) circle (0.2);\node at (2,-6) {\scriptsize{$2$}};
\fill (3,5) circle (0.2);\node at (5,5) {$\dots$};
\node at (4,-5) {$\dots$};
\fill (6,-5) circle (0.2);
\node at (6,-6) {\scriptsize{$a$}};
\fill (7,-5) circle (0.2);
\node at (8.5,-5) {$\dots$};
\fill (10,-5) circle (0.2);
\node at (10,-6) {\scriptsize{$N+1$}};

\fill (8,5)  circle (0.2);
\fill (9,5) circle (0.2);

\draw (9.8,5.5) -- (9.8,6) -- (16.2,6) -- (16.2,5.5);\node at (13,7) {$\mu_2$};
\fill (10,5) circle (0.2);\fill (13,-5) circle (0.2);
\fill (11,5) circle (0.2);\fill (14,-5) circle (0.2);
\node at (13,5) {$\dots$};\node at (16,-5) {$\dots$};
\fill (19,-5) circle (0.2);
\fill (20,-5) circle (0.2);

\node at (18.5,5) {$\dots$};

\draw (20.8,5.5) -- (20.8,6.5) -- (27.2,6.5) -- (27.2,5.5);\node at (24,7) {$\mu_a$};
\fill (23,5) circle (0.2);\fill (22,-5) circle (0.2);
\fill (24,5) circle (0.2);\node at (24,5.8) {\scriptsize{$i_a$}};
\node at (25.5,-5) {$\dots$};
\fill (28,5) circle (0.2);

\node at (32.5,5) {$\dots$};
\draw (36.8,5.5) -- (36.8,6) -- (43.2,6) -- (43.2,5.5);\node at (40,7) {$\mu_{N+1}$};
\fill (37,5) circle (0.2);
\node at (40,5) {$\dots$};

\draw[thick] (10,5)..controls +(0,-4) and +(0,+4) .. (2,-5);
\draw[thick] (23,5)..controls +(0,-4) and +(0,+4) .. (6,-5);
\draw[thick] (28,5)..controls +(0,-4) and +(0,+4) .. (7,-5);
\draw[thick] (25,0.6)..controls +(0,-2) and +(0,+4) .. (10,-5);
\draw[thick] (37,5)..controls +(0,-4) and +(0,+4) .. (26,0.6);

\fill[white] (7.1,0.8) circle (0.3);\fill[white] (6.7,0.5)circle (0.3);
\fill[white] (10.8,-1.3) circle (0.35);\fill[white] (10.2,-1.5)circle (0.35);
\fill[white] (11.5,-1.8) circle (0.4);\fill[white] (11,-2.1)circle (0.3);
\fill[white] (12.7,-2.7) circle (0.4);\fill[white] (12.2,-3)circle (0.3);
\fill[white] (9.5,3.35) circle (0.35);\fill[white] (9.2,2.8)circle (0.35);
\fill[white] (14.4,-0.1) circle (0.35);\fill[white] (13.9,-0.4)circle (0.35);
\fill[white] (15.4,-0.5) circle (0.4);\fill[white] (15,-0.7)circle (0.35);
\fill[white] (16.9,-1.5) circle (0.4);\fill[white] (16.4,-1.7)circle (0.3);
\fill[white] (15.8,0.3) circle (0.3);
\fill[white] (16.8,-0.2) circle (0.3);
\fill[white] (18.5,-1.2) circle (0.3);
\fill[white] (24.6,2.2) circle (0.3);


\draw[thick] (1,5) -- (1,-5);
\draw[thick] (2,5)..controls +(0,-4) and +(0,+4) .. (13,-5);
\draw[thick] (3,5)..controls +(0,-4) and +(0,+4) .. (14,-5);
\draw[thick] (8,5)..controls +(0,-4) and +(0,+4) .. (19,-5);
\draw[thick] (9,5)..controls +(0,-4) and +(0,+4) .. (20,-5);
\draw[thick] (11,5)..controls +(0,-4) and +(0,+4) .. (22,-5);
\draw[thick] (24,5)..controls +(0,-2) and +(0,+1) .. (25,0.6);
\draw[thick] (26,0.6)..controls +(0,-1) and +(0,+1) .. (28,-2);


\end{tikzpicture}
\end{center}
where we have indeed removed the crossing circled in the preceding diagram. Now what we claim is that precomposing (on top) by the generator $T_{i_a-1}$ is equal to composing (at the bottom) by $T_{a}...T_{N-1}T_NT_{N-1}^{-1}\dots T_a^{-1}$. One has to look at the following picture, where one can see that the added crossing on top is transported to the braid added at the bottom.
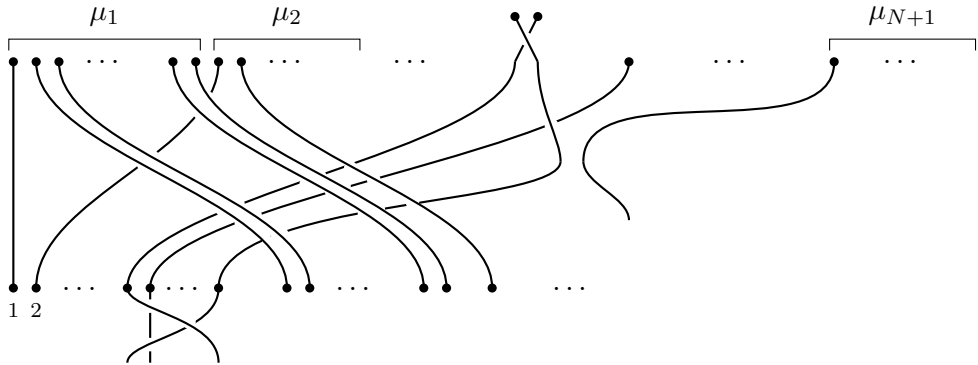
\begin{center}
 \begin{tikzpicture}[scale=0.3]
\draw (0.8,5.5) -- (0.8,6) -- (9.2,6) -- (9.2,5.5);\node at (5,7) {$\mu_1$};
\fill (1,5) circle (0.2);\fill (1,-5) circle (0.2);\node at (1,-6) {\scriptsize{$1$}};
\fill (2,5) circle (0.2);\fill (2,-5) circle (0.2);\node at (2,-6) {\scriptsize{$2$}};
\fill (3,5) circle (0.2);\node at (5,5) {$\dots$};
\node at (4,-5) {$\dots$};
\fill (6,-5) circle (0.2);
\fill (7,-5) circle (0.2);
\node at (8.5,-5) {$\dots$};
\fill (10,-5) circle (0.2);

\fill (8,5)  circle (0.2);
\fill (9,5) circle (0.2);

\draw (9.8,5.5) -- (9.8,6) -- (16.2,6) -- (16.2,5.5);\node at (13,7) {$\mu_2$};
\fill (10,5) circle (0.2);\fill (13,-5) circle (0.2);
\fill (11,5) circle (0.2);\fill (14,-5) circle (0.2);
\node at (13,5) {$\dots$};\node at (16,-5) {$\dots$};
\fill (19,-5) circle (0.2);
\fill (20,-5) circle (0.2);

\node at (18.5,5) {$\dots$};

\fill (23,7) circle (0.2);\fill (22,-5) circle (0.2);
\fill (24,7) circle (0.2);
\node at (25.5,-5) {$\dots$};
\fill (28,5) circle (0.2);

\node at (32.5,5) {$\dots$};
\draw (36.8,5.5) -- (36.8,6) -- (43.2,6) -- (43.2,5.5);\node at (40,7) {$\mu_{N+1}$};
\fill (37,5) circle (0.2);
\node at (40,5) {$\dots$};

\draw[thick] (10,5)..controls +(0,-4) and +(0,+4) .. (2,-5);
\draw[thick] (23,5)..controls +(0,-4) and +(0,+4) .. (6,-5);
\draw[thick] (28,5)..controls +(0,-4) and +(0,+4) .. (7,-5);
\draw[thick] (25,0.6)..controls +(0,-2) and +(0,+4) .. (10,-5);
\draw[thick] (37,5)..controls +(0,-4) and +(0,+4) .. (26,0.6);

\fill[white] (7.1,0.8) circle (0.3);\fill[white] (6.7,0.5)circle (0.3);
\fill[white] (10.8,-1.3) circle (0.35);\fill[white] (10.2,-1.5)circle (0.35);
\fill[white] (11.5,-1.8) circle (0.4);\fill[white] (11,-2.1)circle (0.3);
\fill[white] (12.7,-2.7) circle (0.4);\fill[white] (12.2,-3)circle (0.3);
\fill[white] (9.5,3.35) circle (0.35);\fill[white] (9.2,2.8)circle (0.35);
\fill[white] (14.4,-0.1) circle (0.35);\fill[white] (13.9,-0.4)circle (0.35);
\fill[white] (15.4,-0.5) circle (0.4);\fill[white] (15,-0.7)circle (0.35);
\fill[white] (16.9,-1.5) circle (0.4);\fill[white] (16.4,-1.7)circle (0.3);
\fill[white] (15.8,0.3) circle (0.3);
\fill[white] (16.8,-0.2) circle (0.3);
\fill[white] (18.5,-1.2) circle (0.3);
\fill[white] (24.6,2.2) circle (0.3);


\draw[thick] (1,5) -- (1,-5);
\draw[thick] (2,5)..controls +(0,-4) and +(0,+4) .. (13,-5);
\draw[thick] (3,5)..controls +(0,-4) and +(0,+4) .. (14,-5);
\draw[thick] (8,5)..controls +(0,-4) and +(0,+4) .. (19,-5);
\draw[thick] (9,5)..controls +(0,-4) and +(0,+4) .. (20,-5);
\draw[thick] (11,5)..controls +(0,-4) and +(0,+4) .. (22,-5);
\draw[thick] (24,5)..controls +(0,-2) and +(0,+1) .. (25,0.6);
\draw[thick] (26,0.6)..controls +(0,-1) and +(0,+1) .. (28,-2);

\draw[thick] (24,7)..controls +(0,0) and +(0,0) .. (23,5);
\fill[white] (23.5,6) circle (0.3);
\draw[thick] (23,7)..controls +(0,0) and +(0,0) .. (24,5);
\draw[thick] (7,-5)..controls +(0,0) and +(0,0) .. (7,-8.3);
\fill[white] (7,-7.5) circle (0.3);
\draw[thick] (10,-5)..controls +(0,-2) and +(0,+1) .. (6,-8.3);
\fill[white] (7,-5.8) circle (0.3);
\fill[white] (8.7,-6.5) circle (0.3);
\draw[thick] (6,-5)..controls +(0,-1) and +(0,+2) .. (10,-8.3);
\end{tikzpicture}
\end{center}
The move is entirely happening below the strands descending from left to right, which therefore do not interfere with the procedure. To summarize, if we denote $X$ the element in (\ref{element_diagrammaticproof}) between the idempotent $e_{\mu}$ and $C^{\da}_{w_{N+1}}$, we have shown in particular that:
\[e_{\mu}T_{i_a-1}XC^{\da}_{w_{N+1}}=e_{\mu}XT_{a}...T_{N-1}T_NT_{N-1}^{-1}\dots T_a^{-1}C^{\da}_{w_{N+1}}\ .\]
To conclude we notice that $T_{i_a-1}$ is absorbed by $e_{\mu}$ and replaced by $q$, while $T_a\dots T_{N-1}T_NT_{N-1}^{-1}\dots T_a^{-1}$ is absorbed by $C^{\da}_{w_{N+1}}$ and replaced by $-q^{-1}$ (recall that $T_iC^{\da}_{w_{N+1}}=-q^{-1}C^{\da}_{w_{N+1}}$ for any $i=1,\dots,N$). Therefore we get:
\[(q+q^{-1})e_{\mu}XC^{\da}_{w_{N+1}}=0\,,\]
which implies that the element (\ref{element_diagrammaticproof}) is 0 as desired, since we assumed in this section that we worked over the field $\mathbb{C}(q)$ or that $q^2$ is not a too small root of unity.

\medskip
Now that we have shown (\ref{element_diagrammaticproof}), we deduce that, as we did for $N=1$, in 
$$e_{\mu}T_{\gamma_{\mu}}C^{\da}_{w_{N+1}}=e_{\mu}T_{\gamma_{\mu^-}}T_{\mu_1+\dots+\mu_N}\dots \dots T_{N+1}C^{\da}_{w_{N+1}}\ ,$$
we can replace (from left to right) all appearing $T_i$ by their inverses, using the relation $T_i=T_i^{-1}+(q-q^{-1})$ and (\ref{element_diagrammaticproof}). We thus have
\begin{equation}\label{eq:proofinduction}
e_{\mu}T_{\gamma_{\mu}}C^{\da}_{w_{N+1}}=e_{\mu}T_{\gamma_{\mu^-}}\overline{T}_{\mu_1+\dots+\mu_N}\dots \dots \overline{T}_{N+1}C^{\da}_{w_{N+1}}\ .
\end{equation}
Finally, to be able to use the induction hypothesis, we write:
\[e_{\mu}=e_{\mu^-}e_{\mu_{N+1}}\ \ \ \ \text{with}\ \ e_{\mu_{N+1}}\in\langle T_{\mu_1+\dots+\mu_N+1},\dots,T_{n-1}\rangle\ ,\]
and we note that $e_{\mu_{N+1}}$ commutes with $T_{\gamma_{\mu^-}}\in\langle T_1,\dots,T_{\mu_1+\dots+\mu_{N-1}}\rangle$. Furthermore, we write
\[C^{\da}_{w_{N+1}}=C^{\da}_{w_{N}}Y\ \ \ \ \text{for some elements $Y\in H_n $.}\]
The precise form of $Y$ is not important here, the important property is that $C^{\da}_{w_{N}}$ commutes with all $T_i$ with $i\geq N+1$. In particular, it commutes with $e_{\mu_{N+1}}$. Therefore, we rewrite (\ref{eq:proofinduction}) as
$$e_{\mu}T_{\gamma_{\mu}}C^{\da}_{w_{N+1}}=e_{\mu^-}T_{\gamma_{\mu^-}}C^{\da}_{w_{N}}e_{\mu_{N+1}}\overline{T}_{\mu_1+\dots+\mu_N}\dots \dots \overline{T}_{N+1}Y\ .$$
We use the induction hypothesis, namely, $e_{\mu^-}T_{\gamma_{\mu^-}}C^{\da}_{w_{N}}=e_{\mu^-}\overline{T_{\gamma_{\mu^-}}}C^{\da}_{w_{N}}$, and we put everything back to their original position to get finally
$$e_{\mu}T_{\gamma_{\mu}}C^{\da}_{w_{N+1}}=e_{\mu}\overline{T_{\gamma_{\mu^-}}}\,\overline{T}_{\mu_1+\dots+\mu_N}\dots \dots \overline{T}_{N+1}C^{\da}_{w_{N+1}}\ .$$
This concludes the proof of (\ref{PTCbar}) and in turn of the proposition.
\end{proof}

\begin{rema}\label{rema_Tgamma}
In the preceding proof, we have actually shown that in the expression:
\[X^\mu_N=e_{\mu}T_{\gamma_{\mu}}C^{\da}_{w_{N+1}}T_{\gamma_{\mu}}^{-1}e_{\mu}\ ,\]
one can replace any generator $T_i$ appearing in $T_{\gamma_{\mu}}$ by its inverse $T_i^{-1}$, and similarly, one can replace any $T_i^{-1}$ appearing in $T_{\gamma_{\mu}}^{-1}$ by $T_i$. For example, we have
\[X^\mu_N=e_{\mu}T_{\gamma_{\mu}}C^{\da}_{w_{N+1}}T_{\gamma^{-1}_{\mu}}e_{\mu}\ ,\]
which is an expression explicitly involving only the generators $T_i$ and never their inverses.
\end{rema}

We conclude this paper by proving the preceding conjectures in some special cases.
\begin{prop}\label{prop_conj}
In the following situations:
\begin{itemize}
\item for any $N\geq 1$ with $\mu=(\mu_1,1,1,\dots,1)$.
\item for $N=1$ and any $\mu=(\mu_1,\mu_2)$;
\item for $N=2$ and any $\mu=(\mu_1,\mu_2,\mu_3)$,
\end{itemize}
we have:
\[X^\mu_N=Y^\mu_N\ ,\]
and this element generates the ideal $I^{\mu}_N$.
\end{prop}
\begin{proof}
In these specific cases, the fact that $X^\mu_N$ generates the ideal $I^{\mu}_N$ was proven in \cite{CP23}, so we only need to prove the equality in the proposition, which is
\begin{equation}\label{eq_lastprop}
e_{\mu}T_{\gamma_{\mu}}C^{\da}_{w_{N+1}}T_{\gamma_{\mu}}^{-1}e_{\mu}=e_{\mu}C^{\da}_{\tilde{w}_{N+1}}e_{\mu}\ .
\end{equation}
Recall from (\ref{Tgammap}) that we can modify $T_{\gamma_{\mu}}$ into some different element $T_{\gamma'_{\mu}}$ without changing the left hand side of (\ref{eq_lastprop}). We will use this freedom below.

$\bullet$ Let $N\geq 1$ and $\mu=(\mu_1,1,1,\dots,1)$. Here we choose to work with the following element $T_{\gamma'_{\mu}}$: 
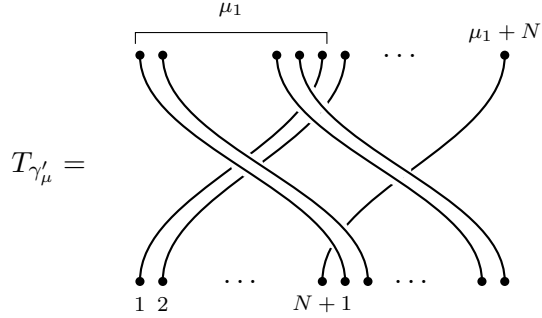
\begin{center}
 \begin{tikzpicture}[scale=0.3]
\node at (-2,0) {$T_{\gamma'_{\mu}}=$};
\draw (1.8,5.5) -- (1.8,6) -- (10.2,6) -- (10.2,5.5);\node at (6,7) {\scriptsize{$\mu_1$}};
\fill (2,5) circle (0.2);\fill (2,-5) circle (0.2);\node at (2,-6) {\scriptsize{$1$}};
\fill (3,5) circle (0.2);\fill (3,-5) circle (0.2);\node at (3,-6) {\scriptsize{$2$}};
\node at (6.5,-5) {$\dots$};
\fill (10,-5) circle (0.2);\node at (10,-6) {\scriptsize{$N+1$}};

\fill (8,5)  circle (0.2);
\fill (9,5)  circle (0.2);
\fill (10,5) circle (0.2);
\node at (13.5,5) {$\dots$};
\fill (18,5) circle (0.2);
\node at (18,6) {\scriptsize{$\mu_1+N$}};

\fill (11,-5) circle (0.2);
\fill (11,5) circle (0.2);\fill (12,-5) circle (0.2);
\node at (14,-5) {$\dots$};
\fill (17,-5) circle (0.2);
\fill (18,-5) circle (0.2);

\draw[thick] (10,5)..controls +(0,-4) and +(0,+4) .. (2,-5);
\draw[thick] (11,5)..controls +(0,-4) and +(0,+4) .. (3,-5);
\draw[thick] (18,5)..controls +(0,-4) and +(0,+4) .. (10,-5);

\fill[white] (6.7,0.4) circle (0.3);\fill[white] (6.3,0.1)circle (0.3);
\fill[white] (7.1,0.2) circle (0.3);\fill[white] (6.7,-0.1)circle (0.3);
\fill[white] (9.4,3.35) circle (0.35);\fill[white] (9.1,2.8)circle (0.35);
\fill[white] (9.9,2.8) circle (0.35);\fill[white] (9.6,2.3)circle (0.35);
\fill[white] (10.8,-2.7) circle (0.35);\fill[white] (10.5,-3.2)circle (0.35);

\fill[white] (13.6,-0.1) circle (0.3);\fill[white] (13.2,-0.4)circle (0.35);

\draw[thick] (2,5)..controls +(0,-4) and +(0,+4) .. (11,-5);
\draw[thick] (3,5)..controls +(0,-4) and +(0,+4) .. (12,-5);
\draw[thick] (8,5)..controls +(0,-4) and +(0,+4) .. (17,-5);
\draw[thick] (9,5)..controls +(0,-4) and +(0,+4) .. (18,-5);

\node at (25,0) {.};
\end{tikzpicture}
\end{center}
With a formula, we have:
\[T_{\gamma'_{\mu}}=T_{\mu_1-1}\dots T_1\cdot T_{\gamma_{\mu}}=T_{\mu_1-1}\dots T_1\cdot T_{\mu_1}\dots T_2\cdot\ldots\ldots\cdot T_{\mu_1+N-1}\dots T_{N+1}\ .\]
Using several times the braid relations or looking at the braid element above, it is easy to find that:
\[T_{\gamma'_{\mu}}T_i=T_{\mu_1+i-1}T_{\gamma'_{\mu}}\ \ \ \ \ \forall i=1,\dots,N\ .\]
Now, $w_{N+1}$ is the longest permutation of the letters $\{1,\dots,N+1\}$, and in this case $\tilde{w}_{N+1}$ is the longest permutation of the letters $\{\mu_1,\dots,\mu_1+N\}$. Therefore, we have that $C_{\tilde{w}_{N+1}}^\da$ is obtained from $C_{w_{N+1}}^\da$ by replacing the generators $T_1,\dots,T_N$ by, respectively, $T_{\mu_1},\dots,T_{\mu_1+N-1}$. From the preceding calculations, we see that:
\[T_{\gamma'_{\mu}}C_{w_{N+1}}^\da=C_{\tilde{w}_{N+1}}^\da T_{\gamma'_{\mu}}\ ,\]
and this proves (\ref{eq_lastprop}).

$\bullet$ Let $N=1$ and $\mu=(\mu_1,\mu_2)$. Here we modify again the element $T_{\gamma_{\mu}}$ into
\[T_{\gamma'_{\mu}}=T_{\mu_1-1}\dots T_1\cdot T_{\gamma_{\mu}}=T_{\mu_1-1}\dots T_1\cdot T_{\mu_1}\dots T_2\ .\]
The permutations $w_{N+1}$ and $\tilde{w}_{N+1}$ are, respectively, the transpositions $(1,2)$ and $(\mu_1,\mu_1+1)$. Therefore, we have:
\[C^\da_{w_{N+1}}=q-T_1\ \ \ \ \text{and}\ \ \ \ C^\da_{\tilde{w}_{N+1}}=q-T_{\mu_1}\ .\]
Using the braid relations, it is easy to see that $T_{\gamma'_{\mu}}T_1=T_{\mu_1}T_{\gamma'_{\mu}}$ and this proves (\ref{eq_lastprop}) with $T_{\gamma'_{\mu}}$ as required.

$\bullet$ Let $N=2$ and $\mu=(\mu_1,\mu_2,\mu_3)$. This case is more involved. Here we have:
\[T_{\gamma_{\mu}}=T_{\mu_1}\dots T_2\cdot T_{\mu_1+\mu_2}\dots T_3=T_{\mu_1+\mu_2}\dots T_{\mu_1+2}\cdot T_{\mu_1}\dots T_2\cdot T_{\mu_1+1}\dots T_3\ .\]
Again, we are going to work with the modified element
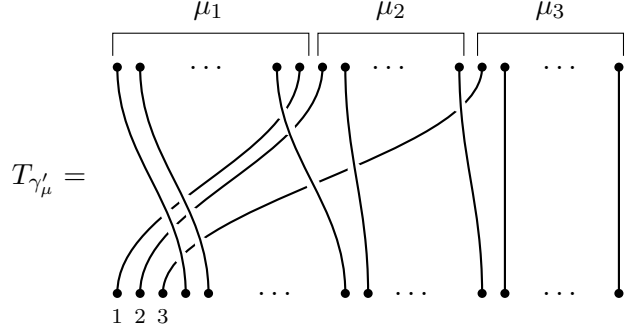
\begin{center}
 \begin{tikzpicture}[scale=0.3]
\node at (-2,0) {$T_{\gamma'_{\mu}}=$};
\draw (0.8,5.5) -- (0.8,6.5) -- (9.4,6.5) -- (9.4,5.5);\node at (5,7.5) {$\mu_1$};
\fill (1,5) circle (0.2);\fill (1,-5) circle (0.2);
\fill (2,5) circle (0.2);\fill (2,-5) circle (0.2);\node at (1,-6) {\scriptsize{$1$}};
\node at (2,-6) {\scriptsize{$2$}};
\fill (3,-5) circle (0.2);\node at (3,-6) {\scriptsize{$3$}};
\node at (5,5) {$\dots$};
\fill (4,-5) circle (0.2);
\fill (5,-5) circle (0.2);\node at (8,-5) {$\dots$};
\fill (8,5)  circle (0.2);
\fill (9,5) circle (0.2);
\draw (9.8,5.5) -- (9.8,6.5) -- (16.2,6.5) -- (16.2,5.5);\node at (13,7.5) {$\mu_2$};
\fill (10,5) circle (0.2);
\fill (11,5) circle (0.2);
\fill (11,-5) circle (0.2);
\node at (13,5) {$\dots$};
\fill (12,-5) circle (0.2);
\node at (14,-5) {$\dots$};
\fill (16,5) circle (0.2);
\draw (16.8,5.5) -- (16.8,6.5) -- (23.2,6.5) -- (23.2,5.5);\node at (20,7.5) {$\mu_3$};
\fill (17,5) circle (0.2);\fill (17,-5) circle (0.2);
\fill (18,5) circle (0.2);
\fill (18,-5) circle (0.2);
\node at (20.5,5) {$\dots$};
\node at (20.5,-5) {$\dots$};
\fill (23,5) circle (0.2);\fill (23,-5) circle (0.2);
\draw[thick] (9,5)..controls +(0,-4) and +(0,+4) .. (1,-5);
\draw[thick] (10,5)..controls +(0,-4) and +(0,+4) .. (2,-5);
\draw[thick] (17,5)..controls +(0,-4) and +(0,+4) .. (3,-5);
\fill[white] (3.8,-3.4)circle (0.3);\fill[white] (4.7,-2.8) circle (0.3);
\fill[white] (3.5,-2.2)circle (0.3);\fill[white] (4.3,-1.5) circle (0.3);
\fill[white] (3.2,-1.5)circle (0.3);\fill[white] (4.0,-0.9) circle (0.3);
\fill[white] (8.2,3)circle (0.3);\fill[white] (8.5,2.3)circle (0.3);
\fill[white] (9.6,-0.2)circle (0.3);\fill[white] (11.4,0.5)circle (0.3);
\fill[white] (16,3.3)circle (0.3);
\draw[thick] (1,5)..controls +(0,-4) and +(0,+4) .. (4,-5);
\draw[thick] (2,5)..controls +(0,-4) and +(0,+4) .. (5,-5);
\draw[thick] (8,5)..controls +(0,-4) and +(0,+4) .. (11,-5);
\draw[thick] (11,5)..controls +(0,-4) and +(0,+4) .. (12,-5);
\draw[thick] (16,5)..controls +(0,-4) and +(0,+4) .. (17,-5);
\draw[thick] (18,5)..controls +(0,-4) and +(0,+4) .. (18,-5);
\draw[thick] (23,5)..controls +(0,-4) and +(0,+4) .. (23,-5);
\end{tikzpicture}
\end{center}
The formula is:
\[T_{\gamma'_{\mu}}=T_{\mu_1-1}\dots T_1\cdot T_{\gamma_{\mu}}=T_{\mu_1+\mu_2}\dots T_{\mu_1+2}\cdot T_{\mu_1-1}\dots T_1\cdot T_{\mu_1}\dots T_2\cdot T_{\mu_1+1}\dots T_3\ .\]
Using the braid relations, we find that this element satisfies:
\[
\begin{array}{rcl}
T_{\gamma'_{\mu}}T_1 & = & T_{\mu_1} T_{\gamma'_{\mu}}\ ,\\[0.4em]
T_{\gamma'_{\mu}}T_2 & = & \tilde{T}_{\mu_1+\mu_2}T_{\gamma'_{\mu}}\,,\ \ \ \ \text{with $\tilde{T}_{\mu_1+\mu_2}=T_{\mu_1+1}^{-1}\dots T_{\mu_1+\mu_2-1}^{-1}T_{\mu_1+\mu_2}T_{\mu_1+\mu_2-1}\dots T_{\mu_1+1}$\ .}
\end{array}
\]
The element $C^\da_{w_{N+1}}$ being the $q$-antisymmetriser made out of the generators $T_1,T_2$, we have:
\begin{equation}\label{proof_beforeP}T_{\gamma'_{\mu}}C^\da_{w_{N+1}}T_{\gamma'_{\mu}}^{-1}=q^3-q^2(T_{\mu_1}+\tilde{T}_{\mu_1+\mu_2})+q (T_{\mu_1}\tilde{T}_{\mu_1+\mu_2}+\tilde{T}_{\mu_1+\mu_2}T_{\mu_1})-T_{\mu_1}\tilde{T}_{\mu_1+\mu_2}T_{\mu_1}\ .
\end{equation}
It remains to multiply on both sides by the idempotent $e_{\mu}$. Recall that $T_ie_{\mu}=e_{\mu}T_i=q e_{\mu}$ for $i\in\{\mu_1,\dots,\mu_1+\mu_2-1\}$. We have the following formulas:
\begin{equation}\label{proof_formulasP}
\begin{array}{l}
e_{\mu}\tilde{T}_{\mu_1+\mu_2}e_{\mu}=e_{\mu}T_{\mu_1+\mu_2}e_{\mu}\ ,\\[0.4em]
e_{\mu}\tilde{T}_{\mu_1+\mu_2}T_{\mu_1}e_{\mu}=q^{1-\mu_2}e_{\mu}T_{\mu_1+\mu_2}\dots T_{\mu_1}e_{\mu}\ ,\\[0.4em]
e_{\mu}T_{\mu_1}\tilde{T}_{\mu_1+\mu_2}e_{\mu}=q^{\mu_2-1}e_{\mu}T_{\mu_1}\dots T_{\mu_1+\mu_2} e_{\mu}-q^{\mu_2-1}(q^{\mu_2-1}-q^{1-\mu_2})e_{\mu}T_{\mu_1}T_{\mu_1+\mu_2} e_{\mu}\ ,\\[0.4em]
e_{\mu}T_{\mu_1}\tilde{T}_{\mu_1+\mu_2}T_{\mu_1}e_{\mu}=e_{\mu}T_{\mu_1}\dots T_{\mu_1+\mu_2}\dots T_{\mu_1} e_{\mu}-q(q^{\mu_2-1}-q^{1-\mu_2})e_{\mu}T_{\mu_1+\mu_2}\dots T_{\mu_1} e_{\mu}\ .
\end{array}
\end{equation}
We will explain how to check them at the end of the proof, and for now we use them in (\ref{proof_beforeP}) and we conclude that:
\begin{multline}
e_{\mu}T_{\gamma'_{\mu}}C^\da_{w_{N+1}}T_{\gamma'_{\mu}}^{-1}e_{\mu}=e_{\mu}\Bigl(q^3-q^2(T_{\mu_1}+T_{\mu_1+\mu_2})+q^{2-\mu_2}T_{\mu_1+\mu_2}\dots T_{\mu_1}+q^{\mu_2}T_{\mu_1}\dots T_{\mu_1+\mu_2}\\
-q^{\mu_2}(q^{\mu_2-1}-q^{1-\mu_2})T_{\mu_1}T_{\mu_1+\mu_2}-T_{\mu_1}\dots T_{\mu_1+\mu_2}\dots T_{\mu_1}+q(q^{\mu_2-1}-q^{1-\mu_2})T_{\mu_1+\mu_2}\dots T_{\mu_1}\Bigr)e_{\mu} \\
=e_{\mu}\Bigl(q^3-q^2(T_{\mu_1}+T_{\mu_1+\mu_2})+q^{\mu_2}(T_{\mu_1+\mu_2}\dots T_{\mu_1}+T_{\mu_1}\dots T_{\mu_1+\mu_2})\\
-(q^{2\mu_2-1}-q^{1})T_{\mu_1}T_{\mu_1+\mu_2}-T_{\mu_1}\dots T_{\mu_1+\mu_2}\dots T_{\mu_1}\Bigr)e_{\mu}\ .
\end{multline}
Note that in the final result, apart from the coefficient $(-1)$ in front of the longest element, the coefficients are all in $q\mathbb{Z}[q]$. To conclude the proof, we remark that the permutation $\tilde{w}_{N+1}$ is the transposition of the letters $\mu_1$ and $\mu_1+\mu_2+1$ and therefore:
\[T_{\tilde{w}_{N+1}}=T_{\mu_1}\dots T_{\mu_1+\mu_2}\dots T_{\mu_1}\ .\]
So the above calculation resulted in
\[e_{\mu}T_{\gamma'_{\mu}}C^\da_{w_{N+1}}T_{\gamma'_{\mu}}^{-1}e_{\mu}=(-1)^{\ell(\tilde{w}_{N+1})}e_{\mu}T_{\tilde{w}_{N+1}}e_{\mu}+\sum_{x<\tilde{w}_{N+1}}\alpha_{x}e_{\mu}T_{x}e_{\mu}\ \ \text{with $\alpha_{x}\in q\mathbb{Z}[q]$.}\]
By the unicity property of the Kazhdan--Lusztig basis proved in Proposition \ref{prop-CT-par2}, we conclude that this element is indeed $e_{\mu}C^\da_{\tilde{w}_{N+1}}e_{\mu}$ as was required.

To finish the proof, we briefly indicate how to find formulas (\ref{proof_formulasP}). The first two formulas are immediate. For the third one, we write
\[e_{\mu}T_{\mu_1}\tilde{T}_{\mu_1+\mu_2}e_{\mu}=q^{\mu_2-1}e_{\mu}T_{\mu_1}T_{\mu_1+1}^{-1}\dots T_{\mu_1+\mu_2-1}^{-1}T_{\mu_1+\mu_2} e_{\mu}\ ,\]
and we remove from left to right the inverses, using the following relations
\[e_{\mu}T_{\mu_1}\underbrace{T_{\mu_1+1}\dots }_{\text{$i-1$ terms}}(T_{\mu_1+i}^{-1}-T_{\mu_1+i})\underbrace{\dots T_{\mu_1+\mu_2-1}^{-1}}_{\text{$\mu_2-i-1$ terms}}T_{\mu_1+\mu_2} e_{\mu}=-(q-q^{-1})q^{-\mu_2+2i}e_{\mu}T_{\mu_1}T_{\mu_1+\mu_2} e_{\mu}\ .\]
This relies on $T_{i}^{-1}=T_{i}+(q-q^{-1})$ and the absorption property of the idempotent $e_{\mu}$. Indeed the $i-1$ underbraced factors commute to hit the idempotent on the right and produce positive powers of $q$. The remaining $\mu_2-i-1$ factors commute and hit the idempotent on the left producing negative powers of $q$. The sum over $i\in\{1,\dots,\mu_2-1\}$ concludes the verification using $(q-q^{-1})(q^{-\mu_2+2}+\dots+q^{\mu_2-2})=q^{\mu_2-1}-q^{1-\mu_2}$.

For the last formulas in (\ref{proof_formulasP}), a similar reasoning works. We write
\[e_{\mu}T_{\mu_1}\tilde{T}_{\mu_1+\mu_2}T_{\mu_1}e_{\mu}=q^{\mu_2-1}e_{\mu}T_{\mu_1}T_{\mu_1+1}^{-1}\dots T_{\mu_1+\mu_2-1}^{-1}T_{\mu_1+\mu_2}\dots T_{\mu_1} e_{\mu}\ ,\]
and we remove from left to right the inverses, using the following relations
\[e_{\mu}\underbrace{T_{\mu_1}\dots }_{\text{$i$ terms}}(T_{\mu_1+i}^{-1}-T_{\mu_1+i})\underbrace{\dots T_{\mu_1+\mu_2-1}^{-1}}_{\text{$\mu_2-i-1$ terms}}T_{\mu_1+\mu_2}\dots T_{\mu_1}  e_{\mu}=-(q-q^{-1})q^{-\mu_2+2i+1}e_{\mu}T_{\mu_1+\mu_2}\dots T_{\mu_1} e_{\mu}\ .\]
This time, using the braid relations, the $i$ factors on the left commute to the right and have their indices increased by $1$. They hit the idempotent $e_{\mu}$ and produce positive powers of $q$. The remaining $\mu_2-i-1$ terms commute and hit the idempotent on the left producing negative powers of $q$. The sum over $i\in\{1,\dots,\mu_2-1\}$ concludes the verification in the same way as before.
\end{proof}

\end{document}